\documentclass[11pt]{article}
\usepackage{amssymb, amsthm, amsmath, amscd, url}
\usepackage{tikz-cd}
\numberwithin{equation}{section}

\newtheorem{thm}{Theorem}[section]
\newtheorem{lem}[thm]{Lemma}
\newtheorem{prop}[thm]{Proposition}
\newtheorem{cor}[thm]{Corollary}
\newtheorem{NN}[thm]{}
\theoremstyle{definition}\newtheorem{df}[thm]{Definition}
\theoremstyle{definition}\newtheorem{rem}[thm]{Remark}
\theoremstyle{definition}

\renewcommand{\phi}{\varphi}

\newcommand{\red}{\textcolor{red}}
\newcommand{\blue}{\color{blue}}

\newcommand{\N}{\mathbb{N}}
\newcommand{\Z}{\mathbb{Z}}

\newcommand{\R}{\mathbb{R}}
\newcommand{\C}{\mathbb{C}}
\newcommand{\T}{\mathbb{T}}
\newcommand{\I}{{\mathbb{I}}}

\newcommand{\hm}{homomorphism}
\newcommand{\dt}{\delta}
\newcommand{\ep}{\epsilon}
\newcommand{\la}{\langle}
\newcommand{\ra}{\rangle}
\newcommand{\andeqn}{\,\,\,{\rm and}\,\,\,}
\newcommand{\rforal}{\,\,\,{\rm for\,\,\,all}\,\,\,}
\newcommand{\CA}{$C^*$-algebra}
\newcommand{\SCA}{$C^*$-subalgebra}

\newcommand{\af}{{\alpha}}
\newcommand{\bt}{{\beta}}

\newcommand{\beq}{\begin{eqnarray}}
\newcommand{\eneq}{\end{eqnarray}}
\newcommand{\tforal}{\,\,\,\text{for\,\,\,all}\,\,\,}
\newcommand{\tand}{\,\,\,\text{and}\,\,\,}

\newcommand{\Wlog}{Without loss of generality}
\renewcommand{\phi}{\varphi}

\title{{{Approximately  Macroscopically  Unique States and Quantum Mechanics}}}
\author{Huaxin Lin and Hang Wang\\
 }
\date{}
\begin{document}

\maketitle

\begin{abstract}
We show that Mumford's {{Approximately Macroscopically}}  Unique (AMU) states exist 
for quantum systems consisting of unbounded self-adjoint operators when the commutators are small. 
In particular, AMU states always exist in position and momentum systems when the Planck constant
$|\hbar|$  is sufficiently small. However, we show that 
these standard quantum mechanical systems are far away from classical (commutative) mechanical systems
%classical mechanical (commutative) systems
even when 
%xxxx 
$|\hbar|\to 0.$ 
\end{abstract}

\section{Introduction}

In quantum mechanics, the states of a physical system are represented by unit vectors in a Hilbert space $H,$ where observables such as position, momentum, and angular momentum correspond to self-adjoint operators on $H.$  The Heisenberg Uncertainty Principle arises from the non-commutativity of these operators: for self-adjoint operators $T_1,T_2,...,T_n,$  the commutators $[T_i,T_j]=T_iT_j-T_jT_i$ quantify the inherent limitations in simultaneously measuring observables. For a unit vector $v\in H$ (a vector state), the expected value of an observable $T_j$ is given by $\la T_j(v),v\ra,$ and the joint expected values form a vector in $\R^n.$ While eigenstates of individual operators are 
expected to be abundant,  joint eigenstates for multiple observables are rare.
In fact, strictly speaking,
self-adjoint operators may not have any  eigenvalues, though every spectral point is an approximate eigenvalue.

\vspace{0.1in}
{\it Approximately {{Macroscopically}}  Unique (AMU) States}
\vspace{0.1in}

A central challenge in bridging quantum mechanics with classical intuition lies in identifying states where macroscopic observables exhibit minimal uncertainty.    David Mumford \cite{Mumford1} 
recently proposed the concept of Approximately Macroscopically Unique (AMU) states, defined for macroscopic observables 
($n$-tuple of self-adjoint operators) $T_1,T_2,...,T_n$ 
%as and a tolerance $\sigma>0$ 
as:
\beq\label{DAMU1}
{\rm AMU}(\{T_1, T_2,...,T_n; \sigma\}):=\{v\in H: \, \|v\|=1, \|(T_j- \la T_j v, v\ra I)(v)\|<\sigma, 1\le j\le n\},
\eneq
with a given tolerance $\sigma>0$
(see (\cite[p.164, Ch 14]{Mumford1}).

These states avoid the paradoxical superpositions (e.g., Schr{\"o}dinger's cat) and represent a ``classical-like" regime.  One may often hear that ``when $\hbar$ goes to zero,
one recovers the classical mechanics."  
 Indeed, if observables with small commutators are close 
to commuting ones, then the associated AMU set is not empty.  
Mumford suggested  that AMU sets are non-empty where observables have  small commutators
(see \cite[Ch.14]{Mumford1}).
This motivates his broader question: When can nearly commuting operators be approximated by commuting ones?

\vspace{0.1in}
{\it Approximation by Commuting Operators}
\vspace{0.1in}

This question is related to the somewhat casual statement that classical 
%A cornerstone of semiclassical analysis is the heuristic that classical 
mechanics emerges as
$\hbar \to 0.$  
Mumford (see \cite{Mumford1} and \cite{Mumford2}) formalized this by seeking an $n$-tuple of commuting self-adjoint operators $S_1, S_2,..., S_n$ close to a given $n$-tuple of self-adjoint ${T_1,T_2,...,T_n}$ with small commutators. Specifically, for fixed $n\in \N$ and $\ep>0,$ does there exist $\dt>0$  such that if 
$\|[T_i, T_j]\|<\dt,$  then there are commuting operators $S_1, S_2,...,S_n$
 satisfying $\|T_i-S_i\|<\ep$ for $1\le i\le n$?
While this, in general, is not true as Mumford observed. Nevertheless, he would like to know when this holds.
Recent advances \cite{Linself} show this holds under certain spectral conditions: when the synthetic and essential synthetic spectra of the operator tuple are close. For $n=2,$ definitive results exist \cite{Linself}, though general cases remain subtle (see also \cite{Og}).

\vspace{0.1in}
{\it Existence of AMU States}
\vspace{0.1in}

Despite the complexity of the stability issues, Mumford’s original  question ---whether small commutators guarantee non-empty AMU sets -- has an affirmative answer. 

{\bf Theorem A} (\cite[Theorem 1.4]{Lincmp2025}) 
For $n\in\N$ and any $\ep>0,$ there exists $\dt(n,\ep)>0$ such that if $\|[T_i,T_j]\|<\dt$ with $\|T_j\|\le 1$ ($1\le i,j\le n$), then for every
$\lambda=(\lambda_1, \lambda_2,...,\lambda_n)$ in the $\ep/4$-synthetic spectrum $s{\rm Sp}^{\ep/4}((T_1,T_2,...,T_n)),$  there exists a unit vector $v\in H$ with $\|(T_j-\lambda_jI)v\|<\ep.$  

This guarantees the existence of abundant AMU states under small-commutator constraints.
%This ensures abundant AMU states under commutator constraints.

\vspace{0.1in}
{\it Unbounded self-adjoint operators}
\vspace{0.1in}

Many physical observables, such as position and momentum, are represented by unbounded self-adjoint operators. Consider $H=L^2(\R),$  where $S_1(f)(t)=tf(t)$  (position) and $S_\hbar (f)(t)=-i\hbar {\frac{df}{dt}}(t)$ (momentum) satisfy  $[S_1,S_{\hbar}]=i\hbar I.$
While $\|[S_1,S_\hbar]\|=|\hbar |$ is small,  these operators are unbounded. We generalize Theorem  A to such cases:

\begin{thm}\label{TAMU}
Let {{$n\in\N$}}, $\ep>0$ and $M>1.$  There exists $0<\eta<\ep/4$ (depending only on 
$\ep$ and $M$) and $\dt({n},\eta, M, \ep)>0$ satisfying the following:

Let $\{h_1,h_2,...,h_n\}$ be an $n$-tuple of (unbounded) self-adjoint operators densely defined 
on an infinite dimensional separable Hilbert space $H$ such that
\beq\label{TAMU-1}
\|h_ih_j-h_{ j}h_{ i}\|<\dt,\,\,\, i,j=1,2,...,n.
\eneq
Then, for any $\lambda=(\lambda_1, \lambda_2,...,\lambda_n)\in s{\rm Sp}_M^{\eta}((h_1, h_2,...,h_n)),$
there are unit vectors $v\in H$ such that
\beq\label{TAMU-1+}
%\max_{1\le i\le 2}|\la h_i(v), v\ra -\lambda_i|<\ep\tand
\max_{1\le i\le n}\|(h_i-\lambda_i)(v)\|<\ep.
\eneq
\end{thm}
Here $s{\rm Sp}_M^{\eta}((h_1, h_2,...,h_n))$ (see Definition \ref{Dsynsp2})  is the $\eta$-synthetic spectrum of the $n$-tuple 
$(h_1, h_2,...,h_n)$ in the bounded region $\{\zeta\in \R^n: \|\zeta\|_2\le M\}$. 
It forms a non-empty compact subset whenever 
${\rm sp}((\sum_{i=1}^n h_i^2)^{1/2})\cap [0, M]\not=\emptyset$ (as $\dt$ is small).
One also notes from \eqref{TAMU-1+} that $|\la h_i(v), v\ra-\lambda_i|<\ep$ ($1\le i\le n$).
In other words,  Mumford's AMU states always exist when $\dt$ is small (see Theorem \ref{thm:6main}).

\vspace{0.1in}
{\it Applications to classical quantum mechanics}
\vspace{0.1in}

For $S_1$ and $S_\hbar,$ Theorem \ref{TAMU} ensures AMU states exist when $|\hbar|$ is small.
This shows that, when $\hbar\to 0,$ the quantum mechanical systems behave similarly to the classical ones in that respect. 
However,  perhaps equally important, in this paper, we also show that, when $\hbar\to 0,$ one may not recover 
classical mechanics.
%the classical mechanics.
In fact, we show (see Theorem \ref{Tnotrecover}) that the pair $(S_1, S_\hbar)$ is far away from any commuting pairs of self-adjoint operators 
no matter how small $|\hbar|$ might be!   In particular, one should not expect a ``classical limit" 
exists when $\hbar \to 0$ in the usual sense.  Perhaps this is precisely the difference between quantum mechanics and 
classical mechanics.

For angular momentum, one considers unbounded self-adjoint operators 
$L_x, L_y, L_z$ on $L^2(\R^3).$  
They satisfy the commutator relation:
\beq
[L_i, L_j]=\sqrt{-1}\hbar \sum_{ k=1}^3\ep_{ijk}L_{k},\,\,\, i,j\in\{{x,y,z}\},
\eneq
where $\ep_{ijk}$ denotes the Levi-Civita symbol. In particular 
the commutator $[L_x, L_y]$ as an operator is not bounded. 
They appear to lie outside the original concern of Mumford.
%They look like outside of the original concern of Mumford
However, our extended main result (see  Theorem \ref{TAMU1} and Theorem \ref{LAnglemom}) also covers this case.
In other words, when $|\hbar|$ is sufficiently small, AMU states for unbounded operators 
$L_x, L_y$ and $L_z$ also exist despite the fact that the commutators are not small and in fact unbounded.

\vspace{0.1in}

{\it Technical remark} 

\vspace{0.1in}

A key technique in our proof is to use an asymptotic one-point compactification to convert  $n$-tuple of (non-commutative) unbounded operators into bounded ones, thereby enabling the application of \CA\,  theory from [4]. While this transformation is essential, it is not a panacea. The approach introduces new complexities when returning to the original non-compact spaces, and, more fundamentally, the core difficulties of unboundedness persist in a transformed guise. A central manifestation of this is that for two unbounded self-adjoint operators $h_1$ and $h_2,$ the operator $[h_1, h_2^2]=h_1h_2^2-h_2^2h_1$
may have  a large norm, or even be unbounded, despite the commutator
$[h_1,h_2]=h_1h_2-h_2h_1$
 being bounded with arbitrarily small norm. This discrepancy causes the standard functional calculus to fail. Thus, the compactification does not make the ``unbounded" issue disappear; instead, it allows us to frame and systematically address it. A significant portion of our subsequent analysis is dedicated to this 
resolution.

%%%%%%%%%
\iffalse
The main method involved in the proof is to convert unbounded operators into bounded ones so that 
\CA\, theory used in \cite{Lincmp2025} may apply.   We employ an asymptotic one-point compactification argument
prepared in Section 3 and 4. However, this also causes difficulties when we get back to non-compact spaces. In particular,  for two unbounded self-adjoint operators $h_1$ $h_1,$ the operator 
$h_1h_2^2-h_2^2h_1$ may have large norm or even unbounded even when 
the commutators  $h_1h_2-h_2h_1$ is a bounded operator with arbitarily small norm. 
So the usual functional calculus failed.  There is no magical way to make "unbounded" issue disappear.
Part of the lengthy calculations are aimed to deal this issue. 
\fi
%%%%%%%%%%%%%

\vspace{0.1in}
{\it Organization of the paper}
\vspace{0.1in}

Section 2 establishes the foundational terminology and notation for the paper. In Section~3, we introduce a method characterized as an almost commutative one-point compactification. This leads to the development in Section 4 of a synthetic spectrum for $n$-tuples of self-adjoint operators with small commutators. The proofs of the main results are presented in Sections 5 and 6, while Section 7 is dedicated to applying the findings from Section 5 to classical quantum mechanics.
\section*{Acknowledgment}
This research is partially supported by a SIMIS research grant and 
%HL is partially supported by SIMIS research grants. 
%HW is supported
by the grants 23JC1401900, NSFC 12271165 and 
%in part 
by Science and Technology Commission of
Shanghai Municipality (No. 22DZ2229014).

\section{Preliminaries}

\begin{df}\label{Dbf1}
Let $A$ be a \CA.  Denote by $A^{\bf 1}$ the closed unit ball of $A.$

\end{df}

\begin{df}\label{Dcpc}
Let $A$ and $B$ be two \CA s. A completely positive contractive linear  map $\psi: A\to B$  is called a c.p.c. map. 
\end{df}

\begin{df}\label{DXep}
Let $X$ be a metric space, $Y\subset X$ be a subset and $\eta>0.$
Set
\beq
Y_\eta=\{x\in X: {\rm dist}(x, Y)<\eta\}.
\eneq
\end{df}

\begin{df}
If $\zeta=(t_1,t_2,...,t_n)\in \R^n$ (for some $n\in \N$), we write that  $\|\zeta\|_2=\sqrt{\sum_{i=1}^nt_i^2}.$
\end{df}

\begin{df}\label{DH}
Throughout the paper, $H$ is
always 
an infinite dimensional separable Hilbert space,
$B(H)$ the \CA\, of all bounded operators on $H$ and 
${\cal K}$ the \CA\, of all compact operators on $H.$

\end{df}

\begin{df}\label{Dntuple}
Let $T_1, T_2,...,T_n$ be (unbounded) self-adjoint operators on $H.$ 
Denote by $D(T_j)$ the domain of $T_{ j},$ $j=1,2,...,n.$ 
We say  that $T_1, T_2,..., T_n$ form an $n$-tuple densely defined self-adjoint operators 
on $H,$ if $\bigcap_{j=1}^n D(T_j)$ is dense in $H.$ Denote by 
$D:=D(T_1,T_2, ...,T_n)$ the common domain of the $n$-tuple.
We also assume that there is a dense subspace $D_0\subset D$ such 
that $\sum_{j=1}^n T_j^2(D_0)\subset D.$ 
\end{df}

%\section{One-point compactification}

\begin{rem}\label{RemSn}

Let $n \ge 2.$
Let $C(S^n)$ denote the $C^*$-algebra of continuous functions on the $n$-sphere $S^n$, which we regard as a compact subset of $\mathbb{R} \times \mathbb{R}^n$ via the identification
\[
S^n = \left\{ (r, \xi) \in \mathbb{R} \times \mathbb{R}^n : -1 \le r \le 1,\, \|\xi\|_2 \le 1,\, \|\xi\|_2^2 = 1 - r^2 \right\}.
\]
Here $\xi = (s_1, s_2, \ldots, s_n) \in \mathbb{I}^n:=[-1,1]^n.$ 
We will use the metric 
on $S^n$ inherited from the Euclidean space $\mathbb{R} \times \mathbb{R}^n.$
In particular, we will write $S^n\subset \I^{n+1}.$
Also we identify the north pole $\zeta^{np}$ with the point $(1,0,0,...,0).$

Define functions $r, x_1, x_2, \ldots, x_n \in C(S^n)$ by
\[
{{r(r, s_1, s_2, \ldots, s_n) = r, \quad x_j(r, s_1, s_2, \ldots, s_n) = s_j \quad \text{for } j = 1, 2, \ldots, n.}}
\]
%\[
%r((s_0, s_1, s_2, \ldots, s_n)) = s_0, \quad x_j((s_0, s_1, s_2, \ldots, s_n)) = s_j \quad \text{for } j = 1, 2, \ldots, n.
%\]
%%%%%%%%%%%%%
%
\iffalse
As a \CA\, $C(S^2)$ is generated by  a commuting pair  consisting  of a  self-adjoint element $a$
with $-1\le a\le 1$  and a normal elements $b$ with $\|b\|\le 1$ such that
\beq\label{S2gen-1}
b^*b=1-a^2.
\eneq

In any \CA\, $A,$ if $a, b\in A$ such that $a\in A_{s.a}$ with $-1\le a\le 1,$
$b$ is normal with $\|b\|\le 1,$ and $a, \,b$ satisfy \eqref{S2gen-1},
then   there is a  surjective \hm\, 
$C(S^2)\to 
C^*(a, b)$ the \SCA\, generated by $a$ and $b.$ 
\fi
%
%
%%%%%%%%%%%%%%

The \CA\, $C(S^n)$ is generated by commuting self-adjoint operators 
$r,$ $x_1, x_2,...,x_n$ such that
$-1\le r\le 1,$ $\|x_i\|_{\infty}\le 1,$ and,
\beq \label{eq:generatorSn}
\sum_{j=1}^n x_j^*x_j=1-r^2.
\eneq

\end{rem}
%
%%%%%%%%%%%%%%%%%%%%%%%%%%%%%%%
\iffalse
%
\begin{df}\label{DR2toS2}
Define two functions $g_1, g_2$ on $\C$ as follows:
\beq
g_1(\zeta)={|\zeta|^2-1\over{|\zeta|^2+1}}\andeqn g_2(\zeta)={2\zeta\over{|\zeta|^2+1}}\rforal \zeta\in \C.
\eneq
Then 
\beq
-1\le g_1(\zeta)<1 \andeqn |g_2(\zeta)|\le 1\rforal \zeta\in \C.
\eneq
Moreover,
\beq
{\overline{ g_2(\zeta)}}g_2(\zeta)=1-g_1(\zeta)^2. 
\eneq

%{\red{We should write an inverse which will be needed later}}
%\end{df}

Note that the map 
\beq \label{eqPhiS2}
\Phi^{S^2}:=(g_1, g_2): \mathbb C\rightarrow {\red{\{(r,z)\in [-1,1)}}\times\mathbb C: r^2+|z|^2=1\}
\eneq
 is injective whose image can be identified with the unit sphere {\red{minus the north pole.}}
 % $r^2+|z|^2=1$ for $(r, z)\in {\red{[-1, 1)}}\times\C.$ 
 This continuous bijection is in fact inverse to the stereographic projection 
\[
\R\times\C \supset S^2\setminus \{(1, 0)\}\rightarrow \C \qquad (g_1, g_2)\mapsto \frac{g_2}{1-g_1},
\]
{\red{denoted by ${\Phi^{S^2}}^{-1}$}}.
%%%%%%%%%%%%%%%%%%%%%%%%
%
\fi
%
%
%%%%%%%%%%%%%%%%%%%%%

\begin{df}\label{DR2toS2}
Similarly, define $g_0: \R^n \to \R$ and $g_{j}: \R^n\to  \R$ by
\beq\label{eqPhiS2}
g_0((t_1, t_2,...,t_n))={\sum_{k=1}^n|t_k|^2-1\over{\sum_{k=1}^n|t_k|^2+1}}\andeqn
g_{j}((t_1, t_2,...,t_n))={2t_j\over{\sum_{k=1}^n|t_k|^2+1}},
\eneq	
for $j=1,2,...,n.$ Then $-1\le g_j\le 1$ for $j=0,1,2,...,n.$  Moreover
\beq
\sum_{j=1}^ng_{j}(\zeta)^2=1-g_0^2(\zeta)  \tforal \zeta\in \R^n.
\eneq
\end{df}

We also have the map $\Phi_{S^n}: \R^n\to  S^n\setminus \{\zeta^{np}\}=S^n\setminus \{(1,0,\ldots, 0)\}\subset [-1,1)\times \R^n$ defined by, for $\zeta\in \R^n,$
\beq
\hspace{-0.15in}\Phi_{S^n}(\zeta)=(g_0(\zeta), g_1(\zeta),,...,g_{n}(\zeta)): \R^n\to \{(r, \xi)\in [-1,1)\times \R^n: r^2+\|\xi\|_2^2=1\}.
\eneq
Note that $\Phi_{S^n}$ is a homeomorphism from  $\R^n$  onto 
the unit $n$-sphere minus the north pole.  The inverse is the stereographic projection 
\[
\R^{n+1} \supset S^n\setminus \{\zeta^{np}\}\rightarrow \R^n \qquad (r, x_1, \ldots, x_n)\mapsto \left(\frac{x_1}{1-r}, \ldots, \frac{x_n}{1-r}\right),
\]
which is denoted by ${\Phi_{S^n}}^{-1}$.

\begin{df}\label{Mdelta}

Let $M>1$ and $\dt_M=1-{M^2-1\over{M^2+1}}={2\over{M^2+1}}<1.$  
Conversely, for $0<\dt<1,$ put $M_\dt=\left({2\over{\dt}}-1\right)^{1/2}.$
Then the image of the closed Euclidean ball of radius 
$M$ under $\Phi_{S^n}$ is the closed spherical bowl
$\{r\le 1-\dt_M\}$:
\beq
\Phi_{S^n}(\{\zeta\in \R^n: \|\zeta\|_2\le M\})=
\{(r, t_1, t_2,..., t_n)\in S^n:  1-r\ge \dt_M\}.
% {\rm dist}(\zeta, \zeta^{np})\ge \dt_M\},
\eneq
%\red{and ${\rm dist}$ measures the vertical distance}. 

For $M>1,$ there exists a continuous function $F_M: [0,2]\to \R_+$ 
with $F_M(0)=0$ 
such that, if $x, y\in \Phi_{S^n}(\{\zeta\in \R^n: \|\zeta\|_2\le M\})$
%\{\zeta\in S^2: {\rm dist}(\zeta, \zeta^{np})\ge \dt_M\}$ 
and 
%the 
%geodesic 
%distance
$\|x-y\|_2=t\ge 0,$ then 
\beq\label{DMS}
\|\Phi_{S^{n}}^{-1}(x)-\Phi_{S^{n}}^{-1}(y)\|_2\le F_M(t)\,\,\,\rforal t\in [0,2].
\eneq

There exists also a continuous function 
$G_M: [0, 2M]\to [0,2]$ such that $G_M(0)=0$ and  if $\|\xi-\zeta\|_2=t$ for 
$\xi, \zeta\in \{z\in \R^n: \|z\|_2\le M\},$
%\mathbb R^n$,
 then, for $t\in [0, 2M],$
\beq
\|\Phi_{S^n}(\xi)-\Phi_{S^n}(\zeta)\|_2^2\le G_M(t).
\eneq
%$F_M$ exists because on the region \{1−r\ge\dt_M\} the stereographic projection is bi‐Lipschitz on compact spherical caps bounded away from the north pole.

%{\blue{What is {\rm dist}? From (e3.7), this should be the vertical distance, but from the domain of $F_M$, {\rm dist} seems to be the distance given by following the great circle passing the two points.}}---  {\red{We did not say what $F_M$ is---so it should OK.}}
These  functions  $F_M$  and $G_M$ are fixed and will be used later.
\end{df}

\begin{df}\label{dfMR}
Let  $M>1.$  
Fix   $e_M\in C_0(\R)$   such  that 
$e_M(s)=1,$ if {$|s|\le  M,$} $e_M(s)=0$
 if $|s|> M+1$ and $e_M$ is linear when $|s|\in [M, M+1].$ In particular, $0\le e_M\le  1.$
 Moreover $e_{M+1}e_M=e_Me_{M+1}=e_M.$ 
 
%For  $\zeta\in \C,$ write $\zeta={\rm Re}(\zeta)+i {\rm Im}(\zeta),$ where ${\rm Re}(\zeta),
%{\rm Im}(\zeta)\in \R.$
%Choose two continuous functions   $f_{M,re}$
%and $f_{M,im}$  on $\C$ such that 
%$f_{M,re}(\zeta)={\rm Re}(\zeta)e_M((\zeta^*\zeta)^{1/2})$ and  
%$f_{M,im}(\zeta)={\rm Im}(\zeta)e_M((\zeta^*\zeta)^{1/2}\red{)}.$ 
%Define, for each $\xi\in S^2\setminus \zeta^{np},$  
%\beq\label{dfMR-2}
%g_{M, re}(\xi)&=&f_{M,re}\circ {\Phi^{S^2}}^{-1}(\xi), \label{eq:gMre}\\
%g_{M,im}(\xi)&=& f_{M, im}\circ {\Phi^{S^2}}^{-1}(\xi),\\
%g_{M,re}(\zeta^{np})&=&g_{M,im}(\zeta^{np})=0,\\
%e_M^{S^2}(\xi)&=&e_M(| {\Phi^{S^2}}^{-1}(\xi)|), \andeqn\\
%e_M^{S^2}(\zeta^{np})&=&0.
%\eneq

Put $\zeta=(t_1, t_2, ...,t_n)\in \R^n.$ 
Define $f_{M,j}(\zeta)=t_j e_M(\|\zeta\|_{_2}),$ $j=1,2,...,n.$
Recall that $\zeta^{np}=(1, 0,\ldots, 0)\in S^{n}$ is the north pole.
Define, for each $\xi\in S^n\setminus \{\zeta^{np}\},$ 
\beq\label{dfMrn-2}
g_{M,j}(\xi)&=& f_{M,j} \circ {\Phi_{S^n}}^{-1}(\xi),\label{dfMrn-3}\\
g_{M,j}(\zeta^{np})&=& 0,\,\,\, j=1,2,...,n,\\
e_M^{S^n}(\xi)&=& e_M(\|{\Phi_{S^n}}^{-1}(\xi)\|_2),\andeqn\\
e_M^{S^n}(\zeta^{np})&=&0.
\eneq
They are continuous functions on $S^n$ with compact supports in $S^n\setminus \{\zeta^{np}\}$.
\end{df}

\begin{df}\label{Opdr2tos2}
Let $H$ be an infinite dimensional Hilbert space
and let $T$ be a densely defined positive operator on $H$ with domain $D.$
By the spectral theorem for (unbounded) self-adjoint operators, we write 
\beq
T=\int_{{\rm sp}(T)} \lambda\, dE_\lambda,
\eneq
where $\rm sp(T)$ denotes the spectrum of $T$ and 
$dE_{\lambda}$ is the spectral measure associated with $T$.
Define the operators on $D$ by
\beq
(1+T)^{-1}|_D=\int_{{\rm sp}(T)} {1\over{1+|\lambda|}} dE_\lambda\andeqn 
(1+T)^{-1/2}|_D=\int_{{\rm sp}(T)} {1\over{\sqrt{1+|\lambda|}}} dE_\lambda.
\eneq 
It follows that $\|(1+T)^{-1}|_D\|\le 1$ and 
$\|(1+T)^{-1/2}|_D\|\le 1.$ We extend these operators to bounded operators on $H$, and denote them by $(1+T)^{-1}$  and $(1+T)^{-1/2},$ respectively.
%{\blue{Question: Is $T$ positive here?}}

In general, if $A$ is a linear operator defined on a dense linear subspace $D$ of $H$ and 
$\|A|_D\|\le M,$ then we may use $A$ for the unique extension of $A$ on $H.$ In particular, we will write 
$\|A\|\le M.$
\end{df}

\begin{rem}\label{RfH}
If $T$ is a densely defined self-adjoint operator and $f\in C_0(\R),$ then 
\beq
f(T)=\int_{{\rm sp}(T)}f({{\lambda}}) dE_\lambda.
\eneq
Let $P_n$ be the spectral projection of $T$ associated with $(n-1, n],$ $n\in \N.$
Then  $P_iP_j=0$ if $i\not=j.$ Hence 
\beq
\|f(T)\|=\|\sum_{n=1}^\infty f(T)P_n\|\le \sup\{|f(t)|: t\in \R\}.
\eneq
In particular, $f(T)$ is a bounded operator. 
\end{rem}

\section{Unbounded $n$-tuples of self-adjoint operators and almost $n$-spheres}

Let us begin with the following folklore lemma, which will be convenient for our purposes.

%Let us begin with the following folklore which is quite convenient for us.
%The following lemma is a folklore  {\red{but  convenient for us to  include here.}}
%may not be standard.
\begin{lem}\label{3unb}
Let $T$ be a  linear operator  defined on a dense subspace $D$ of an infinite separable Hilbert space $H.$
Suppose that $\|T^*T|_D\|\le M$ for some $M>0.$ 
Then $\|T|_D\|\le M^{1/2}.$
\end{lem}

\begin{proof}
Let $q_1, q_2,...,q_m,...$ be an increasing sequence of finite rank projections 
such that $q_nH\subset D$ and $\bigcup_{n=1}^\infty q_nH$ is dense in $D$ and hence in $H.$
Then
\beq
\|q_kT^*Tq_k\|\le M.
\eneq
Since $Tq_k$ is bounded, we must have $\|Tq_k\|\le M^{1/2}.$
Since $\bigcup_{n=1}^\infty q_{n}H$ is dense in $H,$ This implies 
$\|T|_D\|\le M^{1/2}.$
\end{proof}

\begin{lem}\label{Linqx}
Let $\{h_1,h_2,...,h_n\}$ be an $n$-tuple of densely defined self-adjoint operators on $H.$ 
Set $D=D(h_1,h_2,...,h_n)$ and $d=\sum_{j=1}^n h_j^2.$ 
Then the following hold:

{\rm (i)} $0\le  d(1+d)^{-1}|_D\le 1;$

{\rm (ii)} $\|h_j(1+d)^{-1/2}|_D\|\le 1,$ $\|(1+d)^{-1/2}h_j|_D\|\le 1$ and 
$\||h_j|(1+d)^{-1/2}\|\le 1,$ $j=1,2,...,n.$

{\rm (iii)} $\|h_j(1+d)^{-1}|_D\|\le 1$ and $ \|(1+d)^{-1}h_j|_D\|\le 1,$ $j=1,2,...,n.$

%(iii) 
%{\blue $\|h_j(1+d)^{-1/2}|_D\|\le 1$ and $\|(1+d)^{-1/2}h_j\|\le 1,$ $j=1,2,...,n.$}

{\rm (iv)}
For any $M>1, $ $\|e_M(d^{1/2})h_j|_D\|\le M.$ 

{\rm (v)} $\|2 (1+d)^{-1/2} h_j(1+d)^{-1/2}\|\le 1.$

\end{lem}

\begin{proof}
Note that we have established that $0\le (1+d)^{-1}\le 1.$ 
Therefore 
\beq
\|d(1+d)^{-1}|_D\|\le 
%\|\|(1+d)^{-1}\|\cdot 
\|1-(1+d)^{-1}\| \le 1.
\eneq
This proves (i).

For (ii), for each $m\in \N,$ let $p_m$ be the spectral projection of $d$ associated with 
$[0, m].$ 
Then 
\beq
0\le p_mh_j^2 p_m\le p_m(\sum_{i=1}^n h_j^2)p_m=p_m dp_m\le m p_m.
\eneq
Hence we may view $p_m h_j^2 p_m$ as a bounded operator on $H.$ 
By Lemma \ref{3unb}, $\|h_jp_m\|\le m^{1/2}$ and 
%{\red{Moreover, for any finite rank projection $q,$
%\beq
%\%|qp_mh_j^2p_mq\|\le m.
%\eneq
%It follows that $\|h_jp_mq\|\le m^{1/2}$ and 
$\||h_j|p_m\|\le m^{1/2}.$
%for all finite rank projection $q.$ Hence
%$\|h_jp_m\|\le m^{1/2}$ and $\||h_j|p_m\|\le m^{1/2}.$
We also have that
\beq
\|p_m(1+d)^{-1/2} h_j^2(1+d)^{-1/2}p_m\|&=& \|(1+d)^{-1/2} p_m h_j^2p_m (1+d)^{-1/2}\|\\
&\le &\|(1+d)^{-1/2} p_mdp_m(1+d)^{-1/2}\|\\
&\le & \|(1+d)^{-1/2} d (1+d)^{-1/2}p_m\|\\
&=&\|d(1+d)^{-1}p_m\|\le 1.
%&=&\|(1+d)^{-1} d p_m\|\le 1.
\eneq
Let $D_0$ be the dense set as in Definition \ref{Dntuple}. Then, by Lemma \ref{3unb}, 
$h_j(1+d)^{-1/2}p_m$ is defined on $D_0$ and 
\beq
\|h_j(1+d)^{-1/2}p_m|_{D_0}\|\le 1.
\eneq 
Note that $p_m(D_0)\subset D$ and $\bigcup_{m=1}^\infty p_m(D_0)$ is dense in $H.$
This implies that $h_j(1+d)^{-1/2}$ can be (uniquely) extended to a bounded operator on $H$ with
\beq
\|h_j(1+d)^{-1/2} \|\le 1,\,\,\, j=1, 2, ..., n.
\eneq
Similarly, 
\beq
\|(1+d)^{-1/2}h_j\|\le 1,\,\,\, j=1, 2, ..., n.
\eneq
The same argument also shows that
\beq\label{31+1}
\||h_j|(1+d)^{-1/2}\|\le 1.
\eneq

For (iii), we note that $0\le (1+d)^{-1/2}\le 1.$ Hence
\beq
\|h_j(1+d)^{-1}|_D\|\le \|h_j(1+d)^{-1/2}|_D\|\|(1+d)^{-1/2}\|\le 1.
\eneq

For (iv), let $M>1$ and $q_{m,j}$ be the spectral projection of $h_j$ associated with 
the subset $[-m, m].$ Then $h_jq_{m,j}$ is a bounded operator on $H$ and equals $q_{m,j}h_j$.
We have 
\beq
\|e_M(d^{1/2}) h_jq_{m,j}\|^2&=&\|q_{m,j}h_je_M(d^{1/2})\|^2\\
&\le & \|e_M(d^{1/2})q_{m,j}h_j ^2q_{m,j}e_M(d^{1/2})\|\\
&\le &\|e_M(d^{1/2})h_j^2e_M(d^{1/2})\|\\
&\le& \|e_M(d^{1/2})de_M(d^{1/2})\|\le  M^2.
\eneq
Since this holds for all $m,$ and $D_0\subset \bigcup_{m=1}^\infty q_{m,j}H,$
we conclude that $\|e_M(d^{1/2}) h_j\|\le M.$

For (v), we note that
  \beq\label{31-8n}
 (2p_m|h_j|p_m)^2=4p_m|h_j|p_m|h_j|p_m\le 4p_mh_j^2p_m
 \le  4p_m(\sum_{i=1}^n h_i^2)p_m=4p_mdp_m.
 \eneq
 On the other hand, {{by the spectral theory (for unbounded self-adjoint operators), we know that $(d^{1/2}-1)^2$ is positive.}}
 %by  the spectral theory (for unbounded self-adjoint operators) that $(d^{1/2}-1)^2$ is positive. 
 In particular,  $2d^{1/2}\le (1+d).$ 
 Thus, by \eqref{31-8n},
 \beq
 2p_m|h_j|p_m\le 2(p_mdp_m)^{1/2}= p_m2d^{1/2}p_m\le p_m(1+d)p_m. 
 \eneq
Then  
 \beq
 \hspace{-0.3in}\|p_m(1+d)^{-1/2} 2|h_j|(1+d)^{-1/2}p_m\|&=&
 \|(1+d)^{-1/2} 2p_m|h_j|p_m(1+d)^{-1/2}\|\\
 &\le& \|(1+d)^{-1/2}p_m(1+d)p_m(1+d)^{-1/2}\|\le 1.
 \eneq
 %Let $q$ be a finite rank projection. 
 %Then 
 %\beq
 %\|qp_m(1+d)^{-1/2} 2|h_j|(1+d)^{-1/2}p_mq\|\le 1.
 %\eneq
 %Therefore $\|\sqrt{2}|h_j|^{1/2}(1+d)^{-1/2}p_mq\|\le 1$ for all finite rank projection $q.$
  It follows  from Lemma \ref{3unb} that
 \beq
 \sqrt{2}\||h_j|^{1/2}(1+d)^{-1/2}p_m\|\le 1.
 \eneq
 Since this holds for all $m\in \N,$ we conclude that
 \beq
 \sqrt{2}\|h_j|^{1/2}(1+d)^{-1/2}\|\le 1.
 \eneq
 It follows that
 \beq
\|2 (1+d)^{-1/2}|h_j|(1+d)^{-1/2}\|\le 1.
 \eneq
 %So we now know that both $2 (1+d)^{-1/2}h_j(1+d)^{-1/2}$ and $2 (1+d)^{-1/2}|h_j|(1+d)^{-1/2}$
 %are bounded operators. 
 By the spectral theory, $|h_j|\pm h_j$ is a positive operator. 
 Hence, $p_m(|h_j|\pm h_j)p_m\ge 0$ as a bounded operator. 
 We have that
 \beq\label{31-n2}
 p_m(1+d)^{-1/2} (|h_j|\pm h_j)(1+d)^{-1/2}p_m=(1+d)^{-1/2} p_m(|h_j|\pm h_j)p_m(1+d)^{-1/2}\ge 0.
 \eneq
 By (ii), $(1+d)^{-1/2} h_j(1+d)^{-1/2}$ is a bounded operator. We also proved  that \\
 $2 (1+d)^{-1/2}|h_j|(1+d)^{-1/2}$ is a  bounded 
 operator. Since  $\bigcup_{m=1}^\infty p_m(D_0)$ is dense in $H,$ \eqref{31-n2} implies that 
 \beq
(1+d)^{-1/2} (|h_j|\pm h_j)(1+d)^{-1/2}\ge 0.
 \eneq
 Hence 
\beq\nonumber
- (1+d)^{-1/2} |h_j|(1+d)^{-1/2} \le (1+d)^{-1/2} h_j(1+d)^{-1/2}
\le (1+d)^{-1/2} |h_j|(1+d)^{-1/2}.
\eneq
It follows that
\beq
\|(1+d)^{-1/2} 2h_j(1+d)^{-1/2}\|\le \|(1+d)^{-1/2} 2|h_j|(1+d)^{-1/2}\|\le 1.
\eneq
This proves (v).
\end{proof}

%
%%%%%%%%%%%%%%%%%%%%%
%
\iffalse
\begin{lem}
$b$ is bounded and $\left\|\left(z^* z+1\right)^{-1} z^* z\left(z^* z+1\right)^{-1}\right\| \leq 1$. 
\end{lem}

\begin{proof}
Let $Z=\begin{bmatrix}0 & z^* \\ 0 & z\end{bmatrix}$ and observe that it is self-adjoint as a densely defined operator. 

Recall that a densely defined operator $T$ on a Hilbert space is  symmetric (or Hermitian, or formally self-adjoint) if 
\[
(T \varphi, \psi)=(\varphi, T \psi) \quad \text { for all } \quad \varphi, \psi \in D(T)
.\] 
$\quad T$ is self-adjoint if $T$ is symmetric and $D(T)=D\left(T^*\right)$.

Apply function  $f(x)=\frac{x}{1+x^2}$ via unbounded functional calculus. See (Reed-Simon, Theorem VIII 5). Then 
\begin{align*}
f(Z)&=Z\left(1+Z^2\right)^{-1}=\left(1+Z^2\right)^{-1} Z \\
& =\begin{bmatrix}
 0 & z^*\left(1+z z^*\right)^{-1} \\
z\left(1+z^* z\right)^{-1} & 0   
\end{bmatrix}
=\begin{bmatrix}
    0 & \left(1+z^* z\right)^{-1} z^* \\
\left(1+z z^*\right)^{-1} z & 0
\end{bmatrix}
\end{align*}
 is bounded, with $\|f(Z)\| \leqslant 1$.
Hence, $z\left(1+z^* z\right)^{-1}$ is bonded with $\left\|z\left(1+z^* z\right)^{-1}\right\| \leq 1$.

The other statement follows from 
\[
f(Z)^2=\left(\begin{array}{cc}\left(1+z^* z\right)^{-1} z^* z\left(1+z^* z\right)^{-1} & 0 \\ 0 & \left(1+z z^*\right)^{-1} z z^*\left(1+z z^*\right)^{-1}\end{array}\right)
\]
is bounded, with $\|f(z)\| \leqslant 1$.
\end{proof}
%%%%%%%%%%
\fi
%
%%%%%%%%%%%%%%%%%%%%%%%%

\begin{df}\label{Dabj}
Let $\{h_1, h_2,...,h_n\}$ be an $n$-tuple of densely defined self-adjoint operators (see Definition~\ref{Dntuple}). 
Put $d=\sum_{j=1}^n h_j^2.$ 
Define $a=(-1+d)(1+d)^{-1},$ $b_j=2h_j(1+d)^{-1},$ $j=1,2,\ldots, n.$ Then
$a, b_j\in B(H)$ (see Lemma \ref{Linqx}). Moreover
\beq\label{eqahj}
\sum_{j=1}^n b_j^*b_j+a^2=1.
\eneq
\end{df}

\begin{lem}\label{Lappcom-1n}
Let $\eta>0.$ 
%There exists $\dt>0$ satisfying the following:
Suppose that
$\{h_1,h_2, ...,h_n\}$ is an $n$-tuple of densely defined (unbounded) self-adjoint operators satisfying
\beq
\|h_ih_j-h_{j}h_{i}\|<\eta/6n
\eneq
for all $i,\,j\in\{1,2,...,n\}.$
Then the associated bounded transforms in Definition~\ref{Dabj} satisfy
\beq
\|[a, b_j]\|<\eta, \quad \|b_j-b_j^*\|<\eta  \andeqn \|[b_i, b_j]\|<\eta.
\eneq
Moreover, 
\[
\|[h_i, (1+d)^{-1}]\|<\eta/3, \quad \|[h_i, (1+d)^{-1/2}]\|<\eta/3.
\]
In other words, by setting $\widetilde h_j=h_j(1+d)^{-1}$ and $\bar h_j=(1+d)^{-1/2}h_j(1+d)^{-1/2}$,one has 
\[
\|\widetilde h_j-\bar h_j\|<\eta/3 \qquad j=1,\ldots, n.
\]

\end{lem}

\begin{proof}
First, by (iii) of Lemma~\ref{Linqx}, for any $i,\,j\in \{1,2,...,n\},$ 
\beq\label{Lappcom-1n-1}
\|h_jh_i^2(1+d)^{-1}-h_ih_jh_i(1+d)^{-1}\|\le \|[h_j, h_i]\|\|h_i(1+d)^{-1}\|<\eta/6n.
\eneq
Also, applying \eqref{Lappcom-1n-1} and  (iii) of Lemma \ref{Linqx}, we have that 
\beq
\hspace{-0.6in}\|(1+d)^{-1}(h_i^2h_j-h_j h_i^2)(1+d)^{-1}\| &\le &
\|(1+d)^{-1}(h_i^2h_j-h_ih_jh_i)(1+d)^{-1}\|\\
&&\hspace{0.4in}+\|(1+d)^{-1}(h_ih_jh_i-h_jh_i^2)(1+d)^{-1}\|\\
&<& \|(1+d)^{-1}h_i[h_i,h_j](1+d)^{-1}\|+\eta/6n\\
&< &\eta/6n+\eta/6n=\eta/3n.
\eneq
It follows that 
\beq\label{Lappcom-1n-3}
\|(1+d)^{-1}(h_jd-dh_j)(1+d)^{-1}\|<n\eta/3n=\eta/3.
\eneq
Hence
\beq\nonumber
\hspace{-0.4in}\|h_j(1+d)^{-1}-(1+d)^{-1}h_j\|&=&
\|(1+d)^{-1}((1+d)h_j -h_j(1+d))(1+d)^{-1}\|\\\label{[hjd]}
&=&\|(1+d)^{-1}(dh_j-h_jd)(1+d)^{-1}\|<\eta/3.
\eneq
In other words, observing $b_j-b_j^*=2h_j(1+d)^{-1}-2(1+d)^{-1}h_j$, and we have
\beq
\|b_j-b_j^*\|<2\eta/3,\,\,\, j=1,2,...,n.
\eneq
Now 
\beq\nonumber
\|[a, b_j]\|&=&\|(d-1)(1+d)^{-1}h_j(1+d)^{-1}-h_j(1+d)^{-1} (d-1)(1+d)^{-1}\|\\\nonumber
&\le & \|(1+d)^{-1}(d-1)h_j(1+d)^{-1}-(1+d)^{-1}h_j(d-1)(1+d)^{-1}\|\\\nonumber
&&\hspace{0.3in}+\|(1+d)^{-1}h_j(d-1)(1+d)^{-1}-h_j(1+d)^{-1}(d-1)(1+d)^{-1}\|\\\label{eqabj0}
&=&\|(d+1)^{-1} (dh_j-h_jd)(1+d)^{-1}\|
+\|[h_j, (1+d)^{-1}]a\|\\
&<& \eta/3+\eta/3<\eta. \label{eq:abj}
\eneq
Hence, by \eqref{[hjd]} and (ii) of Lemma \ref{Linqx},  for $i,\,j\in \{1,2,...,n\},$ 
%\beq
%b_j-b_j^*&=&h_j(1+d)^{-2}h_j\approx_{\eta/3} (1+d)^{-1}h_j(1+d)^{-1}h_j\\
%&\approx_{\eta/3}&(1+d)^{-1}h_j^2(1+d)^{-1}=b_j^*b_j.
%\eneq}}
%{\blue{Also, for any $i,\, j\in \{1,2,...,n\},$
\beq
b_ib_j&=&h_i(1+d)^{-1}h_j(1+d)^{-1}\approx_{\eta/3}(1+d)^{-1}h_i h_j(1+d)^{-1}\\
&\approx_{\eta/6n}&(1+d)^{-1}h_jh_i(1+d)^{-1}\approx_{\eta/3} h_j(1+d)^{-1}h_i(1+d)^{-1}=b_jb_i.
\eneq
%\red{
%Similarly, 
%\begin{align*}
%\|[b_j, b_j^*]\|=&4\|h_j(1+d)^{-2}h_i-4(1+d)^{-1}h_i^2\| \\
%=&4\|[h_j, (1+d)^{-1}](1+d)^{-1}h_i-(1+d)^{-1}h_i[h_i, (1+d)^{-1}]\|\\
%<& 4\eta
%\end{align*}
%and 
%\begin{align*}
%&\|[b_i, b_j]\|\\
%=&4\|h_i(1+d)^{-1}h_j(1+d)^{-1}-h_j(1+d)^{-1}h_i(1+d)^{-1}\| \\
%=&4\|[h_i, (1+d)^{-1}]h_i(1+d)^{-1}+(1+d)^{-1}[h_i, h_j](1+d)^{-1}-[h_j, (1+d)^{-1}]h_i(1+d)^{-1}\|\\
%<& 4\eta(1+\frac{1}{n})
%\end{align*}
%}
To show the next statement, recall 
\[
dh_j-h_jd=\sum_{i=1}^n(h_i^2 h_j-h_j h_i^2)=\sum_{i=1}^n(h_i[h_i, h_j]+[h_i, h_j]h_i).
\]
By applying functional calculus to $(1+x)^{-1/2}=\frac{2}{\pi}\int_0^{\infty}(1+x+\lambda^2)^{-1}d\lambda$ with a uniformly convergent integral over $[0,\infty)$, we obtain that
\beq
(1+d)^{-1/2}=\frac{2}{\pi}\int_0^{\infty}(1+d+\lambda^2)^{-1}d\lambda
\eneq
with the norm convergence. 

Let $p_m$ be the spectral projection of $d$ associated with $[0, m]$ as in
the proof of Lemma \ref{Linqx}.
Then
\beq
(1+d)^{-1/2}p_m=\frac{2}{\pi}\int_0^{\infty}(1+d+\lambda^2)^{-1}p_m d\lambda.
\eneq
%
%
%%%%%%
\iffalse
 {\blue (in fact, one can alternatively take the integral formula $\left(1+D^2\right)^{-1/2}=\frac{1}{\pi} \int_0^{\infty} \lambda^{-1/2}\left(1+\lambda+D^2\right)^{-1} d \lambda$ in  \cite{Baaj-Julg} where $D$ is an unbounded self-adjoint operator and the integrand is continuous in operator norm and the integral converges absolutely in operator norm and then, replace $D^2$ by $d$ and do the change of variable $\lambda\mapsto\lambda^2$),} we obtain
 \fi
 %%%%%%%%%%%%%%%%%%%
 Note that, for each $M>0,$ we have that
 \beq
 h_j\int_0^M (1+d+\lambda^2)^{-1}p_m d\lambda=\int_0^M h_j(1+d+\lambda^2)^{-1}p_m d\lambda.
 \eneq
 On the other hand, for $M>0,$ 
 \beq\nonumber
 \|\int_M^\infty h_j(1+d+\lambda^2)^{-1} p_m d\lambda\|&\le &
 \int_M^\infty\|(1+d+\lambda^2)^{-1}p_mh_j^2p_m(1+d+\lambda^2)^{-1}\|^{1/2}d\lambda\\\nonumber
& \le &\int_M^\infty\|(1+d+\lambda^2)^{-1}p_md^2p_m(1+d+\lambda^2)^{-1}\|^{1/2}d\lambda\\\nonumber
& \le& \int_M^\infty\|d(1+d+\lambda^2)^{-1} p_m\| d\lambda\\\nonumber
 &\le& m\int_M^\infty \|(1+d+\lambda^2)^{-1} \|d\lambda\\\nonumber
& \le& m\int_M^\infty {1\over{1+\lambda^2}} d\lambda\\\nonumber
&=&m({\pi\over{2}}-\arctan( M))\to 0, \,\,\,{\rm as}\,\,\, M\to\infty
 \eneq
 (converges  in norm). 
 Note that in the second last inequality, we applied the spectral theory,
\beq \label{eq:spectral}
\|(1+d+\lambda^2)^{-1}\|\le (1+\lambda^2)^{-1}.
\eneq
 It follows that, for any $m\in \N,$
 \beq
 h_j\int_0^\infty (1+d+\lambda^2)^{-1} p_m d\lambda=\int_0^\infty h_j (1+d+\lambda^2)^{-1}p_m d\lambda.
 \eneq
 Similarly,
 \beq
  (\int_0^\infty (1+d+\lambda^2)^{-1} d\lambda) h_j p_m=\int_0^\infty (1+d+\lambda^2)^{-1}h_jp_m d\lambda.
\eneq
 Thus  
\beq\nonumber
& &[h_j, (1+d)^{-1/2}]p_m\\\nonumber
&=&\frac{2}{\pi}\int_0^{\infty}{{(}}h_j(1+d+\lambda^2)^{-1}-(1+d+\lambda^2)^{-1}h_j{{)}}p_md\lambda \\\nonumber
&=& \frac{2}{\pi}\int_0^{\infty}(1+d+\lambda^2)^{-1}[(1+d+\lambda^2)h_j-h_j(1+d+\lambda^2)](1+d+\lambda^2)^{-1} p_md\lambda \\\nonumber
&=& \frac{2}{\pi}\int_0^{\infty}(1+d+\lambda^2)^{-1}[dh_j-h_jd](1+d+\lambda^2)^{-1}p_m d\lambda \\\nonumber
&=& \sum_{i=1}^n\frac{2}{\pi}\int_0^{\infty}(1+d+\lambda^2)^{-1}[(h_i[h_i, h_j]+[h_i, h_j]h_i)](1+d+\lambda^2)^{-1} p_md\lambda.
\eneq
{Note that  $\|(1+d+\lambda^2)^{-1}\|\le \|(1+d)^{-1}\|.$ Hence, 
 by   (iii) of Lemma \ref{Linqx},
\beq
&&\|(1+d+\lambda^2)^{-1}h_i\|^2\le \|h_i(1+d)^{-2}h_i\|=\|(1+d)^{-1}h_i\|^2\le 1\andeqn\\
&&\|h_i(1+d+\lambda^2)^{-1}\|\le 1.
\eneq
Moreover, in view of (\ref{eq:spectral}), following from the spectral theory,
we estimate that 
\beq
&&\hspace{-0.8in}\|(1+d+\lambda^2)^{-1}[(h_i[h_i, h_j]+[h_i, h_j]h_i)](1+d+\lambda^2)^{-1}\|\\
&&\le 
\|(1+d+\lambda^2)^{-1}h_i\|\cdot \|[h_i, h_j]\|
\cdot \|(1+d+\lambda^2)^{-1}\|\\
&&\hspace{0.6in}+\|(1+d+\lambda^2)^{-1}\|\cdot \|[h_i, h_j]\|
\cdot \|h_i(1+d+\lambda^2)^{-1}\|\\
&&< 2 {\eta\over{6n}}(1+\lambda^2)^{-1}.
\eneq
}
{Then we have  
\beq\nonumber
\|[h_j, (1+d)^{-1/2}]p_m\|&< &\sum_{i=1}^n\frac{2}{\pi}\int_0^{\infty}2 {\eta\over{6n}}(1+\lambda^2)^{-1} d\lambda\\ 
%|(1+d+\lambda^2)^{-1}h_i\|\cdot \|[h_i, h_j]\|\cdot \|(1+d+\lambda^2)^{-1}\|d\lambda \\\nonumber
&\le & \frac{2\eta}{3\pi}\int_0^{\infty}(1+\lambda^2)^{-1}d\lambda=\eta/3.
\eneq}
%where the second last inequality follows from $\|(1+d+\lambda^2)^{-1}h_i\|\le 1$ and $\|(1+d+\lambda^2)^{-1}2\|\le (1+\lambda^2)^{-1}.$ 
Since $\bigcup_{m=1}^\infty p_m(D)$ is dense in $H,$
we actually have that
\beq
\|[h_j, (1+d)^{-1/2}]\|<\eta/3.
\eneq
Consequently, 
\[
\|\tilde h_j-\bar h_j\|=\|[h_j, (1+d)^{-1/2}](1+d)^{-1/2}\|< \eta/3.
\]
%({\blue I think $2\eta/3$ above could be $\eta/3.$})
\end{proof}

\begin{rem}
It should be  noted that  the integral 
\beq
\int_0^\infty {x^{1/2}\over{1+x+\lambda^2}} d\lambda
\eneq
does not converge uniformly on $[0,\infty).$   Since $h_j$ are unbounded operators, some extra efforts are required in the proof 
of the last part of Lemma \ref{Lappcom-1n}. 
\end{rem}

\begin{lem}\label{Lappcom-1nbd}
Let $\eta>0.$ 
%There exists $\dt>0$ satisfying the following:
Suppose that
$\{h_1,h_2, ...,h_n\}$ is an $n$-tuple of densely defined self-adjoint operators satisfying
\beq
\|\tilde h_i\tilde h_j-\tilde h_{j}\tilde h_{i}\|<\eta/4\quad {and}\quad \|\tilde h_j-\tilde h_j^*\|<\eta/2, 
\eneq
for $i,\,j\in\{1,2,...,n\}$, where $\tilde h_j=h_j(1+d)^{-1}.$
Then the associated bounded transforms of $h_j$ in Definition~\ref{Dabj} satisfy
\beq \label{eq:htoab}
\|[a, b_j]\|<\eta, \quad \|b_j-b_j^*\|<\eta  \andeqn \|[b_i, b_j]\|<\eta.
\eneq
\end{lem}

\begin{proof}
    By \eqref{eqabj0}, we have 
    $$
    \|[a, b_j]\|\le \|(d+1)^{-1}(dh_j-h_jd)(1+d)^{-1}\|+\|[h_j, (1+d)^{-1}] a\|.
    $$ 
    Note that $\|a\|\le 1$ and 
    \begin{align*}
    (d+1)^{-1}(dh_j-h_jd)(1+d)^{-1}=&(1-(d+1)^{-1})h_j-h_j(1-(1+d)^{-1})\\
    =&[h_j, (1+d)^{-1}]=\tilde h_j-\tilde h_j^*.
    \end{align*}
    So $\|[a, b_j]\|\le 2\|\tilde h_j-\tilde h_j^*\|\le \eta.$
    The other inequalities in (\ref{eq:htoab}) follow from $b_j=2\tilde h_j.$ 
\end{proof}

\begin{lem}\label{Lemtildehjbarh}
Let $\widetilde h_j=h_j(1+d)^{-1}$ and $\bar h_j=(1+d)^{-1/2}h_j(1+d)^{-1/2}$.
For any $\epsilon>0$, there exists $\delta>0$, whenever
\[
\|\widetilde h_j-\tilde h_j^*\|<\delta, \quad \text{or equivalently} \quad \|[h_j, (1+d)^{-1}]\|<\delta,
\]
one has 
\[
\|\widetilde h_j-\bar h_j\|<\epsilon.
\]
\end{lem}

\begin{proof}
Let $\epsilon>0.$ 
Denote $f(t)=\sqrt t$ for $t\in [0, 1].$ Choose a polynomial $p$ having the form $p(t)=\sum_{m=0}^N a_m t^m$ so that 
\[
\sup_{t\in[0,1]}|f(t)-p(t)|<\epsilon/4.
\] 
Note that $\|[h_j, (1+d)^{-1}]\|=\|\tilde h_j-\tilde h_j^*\|$.
Let $\delta=\epsilon/(2\sum_{m=1}^N|a_m|m),$ and $A:=h_j(1+d)^{-1/2}$, $B:=(1+d)^{-1}$. 

{Suppose that  $\|[h_j, (1+d)^{-1}]\|<\dt.$} 
 {Then, by Lemma \ref{Linqx},} $\|A\|\le 1$ and $ \|B\|\le 1.$ 
Hence, for every $m\in \N,$ 
\beq
AB^m\approx_{\dt} BAB^{m-1} \andeqn \|[A, B^m]\|<m\dt.
\eneq
%%%%%%%%%%%%
\iffalse
Then for every integer $m>0,$ from 
\[
[h_j(1+d)^{-1/2}, (1+d)^{-m}]=\sum_{p=0}^m (1+d)^{-p}[h, (1+d)^{-1}](1+d)^{p-m}(1+d)^{-1/2}, 
\]
we have \[
\|[h_j(1+d)^{-1/2}, (1+d)^{-m}]\|\le m\|[h_j, (1+d)^{-1}]\|.
\]
\fi
%%%%%%%%%%%%%%%%%%
Thus, 
\[
\|[A, p(B)]\|< \left(\sum_{m=1}^N|a_m|m\right)\dt=\ep/2.
%\|[h_j, (1+d)^{-1}]\|.
\]
%%%%%%%%%%%%%
\iffalse
\[
\|[h_j(1+d)^{-1/2}, p((1+d)^{-1})]\|\le (\sum_{m=1}^N|a_m|m)\|[h_j, (1+d)^{-1}]\|.
\]
Let $\delta=\epsilon/(2\sum_{m=1}^N|a_m|m)$.
Then, whenever $\|[h_j, (1+d)^{-1}]\|<\delta$, we have \[
\|[h_j(1+d)^{-1/2}, p((1+d)^{-1})]\|<(\sum_{m=1}^N|a_m|m)\delta=\epsilon/2.
\]
Let $A:=h_j(1+d)^{-1/2}$, $B:=(1+d)^{-1}$. Then $\|A\|<1, \|B\|\le 1$ and
\fi
%%%%%%%%%%%%%%%
It follows that
\beq\nonumber
\|[A, f(B)]\| & \le & \|[A, p(B)]\|+\|A(f(B)-p(B))\|+\|(f(B)-p(B))A\|< \epsilon/2+\epsilon/4+\epsilon/4=\epsilon.
\eneq
Therefore, 
\[
\|{\widetilde h_j}-\bar h_j\|=\|[h_j(1+d)^{-1/2}, (1+d)^{-1/2}]\|=\|[A, f(B)]\|<\epsilon.
\]
\end{proof}

Fixing the generators $r, x_1,\ldots, x_n$ for $C(S^n)$ satisfying $\sum_{j=1}^n x_j^*x_j=1-r^2$ as in (\ref{eq:generatorSn}) where $-1\le r\le 1,$ one has the following lemma.

\begin{lem}\label{Ls2app}
Let $\ep>0.$ There exists $\dt>0$ satisfying the following:
For any elements 
$a, b_1, b_2,...,b_n$ 
 in any unital \CA\, $A$ with 
$-1\le a, \, b_j^*b_j\le 1$ such that
\beq
&&\|[a, b_j]\|<\dt,\,\,\|b_j-b_j^*\|<\dt,\,\,\,
 \sum_{j=1}^n b_j^*b_j=1-a^2,\\
&&\tand \|[b_j, b_i]\|<\dt,,\,\, 1\le i,j\le n,
\eneq
there exists a { {c.p.c.}} map  $\phi: C(S^n)\to A$ 
such that
\beq
\|\phi(r)-a\|<\ep\tand \|\phi(x_j)-b_j\|<\ep,
\eneq
for $j=1,2,..., n.$
\end{lem}

\begin{proof}
Suppose that such a c.p.c. map does not exist. Then there exists $\varepsilon>0$ such that, for $\delta_m=\frac{1}{m}$, there exist unital \CA s $A_m,$ $a_m, b_{j,m}\in A_m$ with $-1\le a_m,\,\, \, b_{j,m}^*b_{j,m}\le 1$ satisfying 
\beq \label{eqreln}
\hspace{-0.2in}\left|\left[a_m, b_{j,m}\right]\right\|<\frac{1}{m},~\sum_{j=1}^n b_{j,m}^*b_{j,m}=1-a_m^2,
~ \|b_{j,m}-b_{j,m}^*\|<\frac1{m},~ \text{ 
and }~\|[b_{i,m}, b_{j, m}]\|<\frac1{m}
\eneq
such that for every c.p.c. map 
$\varphi: C\left(S^n\right) \rightarrow A_m$, 
we have 
\beq \label{eqcontra}
\left\|\varphi\left(r\right)-a_m\right\|\ge \varepsilon\,\, \text{ 
or } \left\|\varphi\left(x_j\right)-b_{j,m}\right\|\ge \varepsilon\,\,\,\text{for\,\,\,some}\, \, 1\le j\le n.
\eneq
Let $\prod_m A_m$ be the $C^*$-algebra consisting of sequences $\{c_m\}$ with $c_m\in A_m$ and $\sup_m\|c_m\|<\infty$ and 
\[
\bigoplus_m A_m:=\left\{\{c_m\}\in\prod_m A_m : \lim_{m\to\infty}\|c_m\|=0\right\} .
\]
Then $\bigoplus_m A_m$ is a (closed) ideal of $\prod_m A_m.$
Regard $\{a_m\}, \{b_{j,m}\},$ $1\le j\le n,$ as elements in $\prod_m A_m.$ 
Then $\lim_{m\to\infty}\|[a_m, b_{j,m}]\|=0$ implies that $\{[a_m,\, b_{j,m}]\}$ belongs to the ideal $\bigoplus_m A_m$. 
Hence in the quotient ${\prod A_m/}{\bigoplus A_m}$, {\rm the images of} $\{a_m\}$ and $\{b_{j,m}\}$ commute.
Similarly, in the quotient, $\{b_{j,m}\}$ and $\{b_{i,m}\}$ commute, and $\{b_{j,m}\}$ is self-adjoint i.e., 
\beq \label{eqc1}
\Pi(\{a_m\})\Pi(\{b_{j,m}\})&=&\Pi(\{b_{j,m}\})\Pi(\{a_m\})  \\
\Pi(\{b_{j,m}\})&=&\Pi(\{b_{j,m}^*\})
\andeqn  \\
\Pi(\{b_{j,m}\})\Pi(\{b_{i,m}\})&=&\Pi(\{b_{i,m}\})\Pi(\{b_{j,m}\}) \label{eqc2}
\eneq
in  ${\prod A_m}/{\bigoplus A_m},$ for
$i,j=1,2,...,n,$
where $\Pi: \prod A_m\rightarrow {\prod A_m}/{\bigoplus A_m}$ is the quotient map. 

Define a map $\phi: C\left(S^n\right) \rightarrow {\prod A_{m}}/{\bigoplus A_{m}}$ by
\beq\label{eqphi}
\phi(r)=\Pi(\left\{a_m\right\}) \andeqn \phi(x_j)=\Pi(\left\{b_{j,m}\right\}).
% \quad \phi(\bar z)=\pi(\left\{b_n^*\right\}).
\eneq
By the universal property of the $C^*$-algebra $C(S^n)$ discussed in Remark~\ref{RemSn}, and in view of $\sum_{j=1}^n b_{j,m}^*b_{j,m}=1-a_m^2$, (\ref{eqc1})-- (\ref{eqc2}), we find that $\phi$ is a \hm. 

By the Choi-Effros lifting theorem \cite{CE},  the map $\phi$ lifts to a c.p.c map, denoted by $\widetilde\phi$, 
\[
\widetilde{\phi}: C(S^n) \longrightarrow \prod_m A_m.
\]
Let $\widetilde\phi_m$ be the restriction to the $m$-th component in the image of  $\widetilde{\phi}$.
Then $\widetilde\phi_m$ is a c.p.c. map for every $m.$

%Since the homomorphism $\phi$ is multiplicative and $\widetilde\phi$ is its lift, for every $\eta>0$, there exists a large $N$, such that for all $n>N$, the c.p.c. map $\widetilde\phi_n$ is $\eta$-multiplicative, i.e.,
%\[
%\|\widetilde\phi_n(x)\widetilde\phi_n(y)-\widetilde\phi_n(xy)\|<{\eta} \qquad \forall x,y\in \red{C(}S^2\red{)}.
%\]
%Note that this observation is not relevant to the proof but useful to the proof Lemma~\ref{Ls2app2}.

Because $\widetilde\phi$ is a lift of $\phi$ and in view of (\ref{eqphi}), we have 
\[
\Pi(\widetilde\phi(r)-\{a_m\})=0\quad \Pi(\widetilde\phi(x_j)-\{b_{j,m}\})=0.
% \quad \pi(\widetilde\phi(\bar z)-\{b_n^*\})=0.
\]
Hence, there exists $N>0$ such that for all $m>N$ we have
\beq \label{eqdif}
\|\widetilde\phi_m(r)-a_m\|<\epsilon \quad \|\widetilde\phi_m(x_j)-b_{j,m}\|<\epsilon.
%\quad \|\widetilde\phi_n(\bar z)-b_n^*\|<\epsilon.
\eneq
Therefore, we have (\ref{eqreln}) and (\ref{eqdif}) for the c.p.c. map $\widetilde\phi_m.$
This contradicts to (\ref{eqcontra}). The proof is then complete. 
\end{proof}

\begin{lem}\label{Ls2app2}
Let $n\in\N, \ep>0,$ $\eta>0,$ and ${\cal F}\subset C(S^n)$ be a finite subset.
Suppose that $M\ge 1$ is also given. Then there exists $\dt>0$ satisfying the following:
For any $n$-tuple  of densely defined self-adjoint operators $h_1,h_2,...,h_n$ 
on a Hilbert space $H$ 
such that their commutators $h_jh_i-h_ih_j$ 
are bounded, and 
\beq
\|h_jh_i-h_ih_j\|<\dt,\,\,\, i,j=1,2,...,n,
\eneq
there exists a c.p.c. map  $\phi: C(S^n)\to B(H)$ 
such that  (with $d=\sum_{j=1}^n h_j^2$)
\beq
&&\|\phi(r)-(-1+d)(1+d)^{-1}\|<\ep,\,\,\|\phi(x_j)-2h_j(1+d)^{-1}\|<\ep\\\label{eq1/2-0}
&& \|\phi(x_j)-2(1+d)^{-1/2}h_j(1+d)^{-1/2}\|<\ep\\\label{eq1/2}
&&\|\phi(g_{M, j})-h_je_M(d^{1/2})\|<\ep,\\
&&\|\phi(e_M^{S^n})-e_M(d^{1/2})\|<\ep\tand\\\label{Ls2app2-5}
&&\|e_M(d^{1/2})h_j-h_je_M(d^{1/2})\|<\ep, \label{eq:h&e}
\eneq
for all $j=1,2,...,n.$
Moreover, we may require that 
\beq
\|\phi(f)\phi(g)-\phi(fg)\|<\eta\tforal f,\, g\in {\cal F}.
\eneq
\end{lem}

\begin{proof}
We first observe that it suffices to prove the lemma without the conclusion \eqref{eq1/2-0}.
Indeed, suppose we can find $\dt'>0$ for $\ep/2,$ $\eta,$ ${\cal F}$ and $M$ 
such that the lemma holds without \eqref{eq1/2-0}.
Then, by choosing $\dt=\min\{\dt', \ep/16n\}>0,$ we claim that the full lemma including \eqref{eq1/2-0} follows.

To see this, assume that $\|h_jh_i-h_ih_j\|<\dt$ for all $i,j=1,2,...,n.$
By %by
Lemma~\ref{Lappcom-1n},  we have 
\[
\|h_j(1+d)^{-1}-(1+d)^{-1/2}h_j(1+d)^{-1/2}\|<\epsilon/4.
\]
Now, if $\phi$ is the c.p.c. map 
obtained without \eqref{eq1/2-0}, then for each $j,$
\beq
&&\|\phi(x_j)-2(1+d)^{-1/2}h_j(1+d)^{-1/2}\|\\
&\le&\|\phi(x_j)-2h_j(1+d)^{-1}\|+2\|h_j(1+d)^{-1}-(1+d)^{-1/2}h_j(1+d)^{-1/2}\|<\epsilon
\eneq
which establishes \eqref{eq1/2-0}. Hence, 
for the remainder of the proof, we may focus on proving the lemma {\it without} \eqref{eq1/2-0}.
%%%%%%%%%%%
\iffalse
Observe first that we need only to prove the lemma without {{\eqref{eq1/2-0}.}}
% (\ref{eq1/2}). 
In fact, suppose $\delta'>0$ is found for $\epsilon/2$ so that the lemma is proved without {{\eqref{eq1/2-0}.}}
%(\ref{eq1/2}). 
Then set $\delta=\min\{\delta', \epsilon/24n\}>0.$
When $\|h_ih_j-h_jh_i\|<\delta$ for $i,j=1,2,\ldots, n$, we have \[
\|h_j(1+d)^{-1}-(1+d)^{-1/2}h_j(1+d)^{-1/2}\|<\epsilon/4
\]
by Lemma~\ref{Lappcom-1n}.  Hence, the lemma is concluded:
\beq
&&\|\phi(x_j)-2(1+d)^{-1/2}h_j(1+d)^{-1/2}\|\\
&\le&\|\phi(x_j)-2h_j(1+d)^{-1}\|+2\|h_j(1+d)^{-1}-(1+d)^{-1/2}h_j(1+d)^{-1/2}\|<\epsilon.
\eneq
}
\fi

%%%%%%%%%%%%
%{\red{The number of equation does seem fit.}}
Now, we shall prove the lemma without {\eqref{eq1/2-0}}
%(\ref{eq1/2}) 
by contradiction. Assume that there exist $\epsilon > 0$, $\eta > 0$, and a finite subset ${\cal F} \subset C(S^n)$ such that, for every $\delta_m := \frac{1}{m}$, the following conditions are satisfied:

There exist an $n$-tuple of densely defined self-adjoint operators $\{h_{1,m},h_{2,m},...,h_{n,m}\}$ on a Hilbert space $H$ with bounded commutators and $\|[h_{1,m}, h_{2,m}]\| < \frac{1}{m}$. Denote
$d_m=\sum_{j=1}^n h_{j,m}^2.$ 
%$z_m := h_{1,m} + i h_{2,m}$, 
Then for any c.p.c. map $\varphi: C(S^n) \to B(H)$ on some Hilbert space $H$, at least one of the following holds:
\beq
&&\|\phi(r)-(-1+d)(1+d)^{-1}\|\ge\ep,\,\,\|\phi(x_j)-2h_j(1+d)^{-1}\|\ge\ep \label{eq1},\\
&&{{\|\phi(g_{M, j})-h_je_M(d_m^{1/2})\|\ge\ep,}}\\
&&{{\|\phi(e_M^{S^n})-e_M(d_m^{1/2})\|\ge\ep}}\\
&&{{\|e_M(d_m^{1/2})h_j-h_je_M(d_m^{1/2})\|\ge\ep,}}
\eneq
$j=1,2,...,n,$ 
or, there exist $f,\, g\in {\cal F}$ such that
\beq \label{eq2}
\|\phi(f)\phi(g)-\phi(fg)\|\ge \eta .
\eneq

We shall construct a specific c.p.c. map $C(S^n) \to B(H)$ for which none of the previously stated inequalities hold.  

Similarly to Definition~\ref{Dabj}, set  
\beq \label{eqanbn}
a_m = (-1+d_m)(1+d_m)^{-1} = 1 - 2 (1+d_m)^{-1}, \quad b_{j,m} = 2 h_{j,m} (1+d_m)^{-1}.
\eneq
Then,  $-1\le a_m\le 1$, and $\sum_{j=1}^n b_{j,m}^*b_{j,m}=1-a_m^2$.
Given that $\|[h_{i,m}, h_{j,m}]\| < \frac{1}{m}$ and by applying Lemma~\ref{Lappcom-1n}, we obtain 
\[
\|[a_m,b_{j,m}]\|<\frac{6n}{m}, ~~\|b_{j,m}-b_{j,m}^*\|<\frac{6n}{m}~~\text{and}~~\|[b_{i,m},b_{j,m}]\|<\frac{6n}{m}.
\] 

Following the approach used in the proof of Lemma~\ref{Ls2app}, we now define a $*$-homomorphism  
\beq \label{eq:spphi}
\psi: C(S^n) \to \prod B(H)}/{\bigoplus B(H)
\eneq
by setting  
\beq \label{eqphi+1}
\psi(r) = \Pi(\{a_m\}), \quad \psi(x_j) = \Pi(\{b_{j,m}\}).
% \quad \text{and} \quad \psi(\bar{z}) = \pi(\{b_n^*\}).
\eneq
In particular $\Pi(\{b_{j,m}\})$ is a self-adjoint operator for $j=1,2,...,n.$

Using the Choi--Effros lifting theorem, we lift $\psi$ to a c.p.c. map  
\[
\widetilde{\psi}: C(S^n) \to \prod B(H),
\]
and by restricting $\widetilde{\psi}$ to the $m$-th component, we obtain a sequence of c.p.c. maps $\widetilde{\psi}_m$.  

We claim that for sufficiently large $m$, the map $\widetilde{\psi}_m$ fails all inequalities (\ref{eq1})-(\ref{eq2}), thereby providing the desired c.p.c. map.

As argued in (\ref{eqdif}), for the given $\epsilon,$ there exists $N_1>0$ such that for all $m>N_1$ we have
\beq \label{eqdif2}
\|\widetilde\psi_m(r)-a_m\|<\epsilon \quad \|\widetilde\psi_m(x_j)-b_{j,m}\|<
\epsilon\quad \|\widetilde\psi_m(\bar x_j)-b_{j,m}^*\|<\epsilon
\eneq
%and equivalently, 
%\[
%\|\widetilde\psi_m(r)-a_m\|<\epsilon~~\text{and}~~\|\widetilde\psi_m(z)-g_2(z)\|<\epsilon.
%\]

From (\ref{eqPhiS2}) and Definition~\ref{dfMR}, one observes that for $\zeta\in\R^n$ and $r=g_0(\zeta), x_j=g_{j}(\zeta)$, 
\[
g_{M,j}(r,x_1,\ldots, x_n)=f_{M,j}\circ{\Phi_{S^n}}^{-1}(g_1(\zeta), g_2(\zeta),...,g_{1+n}(\zeta))=f_{M, j}(\zeta).
\]
This is implemented by the induced map 
\[
{\Phi_{S^n}}_{\#}: C_0(\R^n)\rightarrow C(S^n)
\]
where $f_{M, j}\mapsto g_{M, j}$.
Thus, 
\beq \label{eq:gf-1}
\psi(g_{M, j})=(\psi\circ{\Phi_{S^n}}_{\#})(f_{M, j}).
\eneq
By the definition (\ref{eqphi}) and the continuity of $\psi$, one has for $g\in C(S^n),$ 
\[
\psi(g(r,x_1,\ldots, x_n))=g(\Pi(\{a_m\}), \Pi(\{b_{1,m}\}, ..., \Pi(\{b_{n,m}\}))=\Pi(\{g(a_m, b_{1,m}, ...,b_{n,m})\}).
\]
%Note that in the last term, $g(a_m, b_{1,m},...,b_{n,m})$ is not functional calculus, and the order of $r, x_j$ matters, but $\Pi(\{g(a_m, b_{1,m},...,b_{n,m})\})$ is {\blue independent of such orders. (To check again, do we need to state it? This seems using Section 6 already.)}
Similarly, through a change of variables, the induced \hm\,
%$*$-morphism 
\[
{\psi\circ\Phi_{S^n}}_{\#}: C_0(\R^n)\rightarrow C(S^n)\to \prod B(H)/\bigoplus B(H)
\]
satisfies
\beq \label{eq:gf}
[{\psi\circ\Phi_{S^n}}_{\#}](f(\zeta))=f(\Pi(\{h_{1,m}\},...,\Pi(h_{n,m})\}))=\Pi(\{f(h_{1,m},...,h_{n,m})\})
\eneq
for $f\in C_0(\R^n).$ 
Applying this to $f_{M, j},$ $t_je_{M+1}(\|\zeta\|_2), e_{M}(\|\zeta\|_2)$ in $C_c(\R^n)$, 
we have
%\beq \label{eq:gf}
%\psi(g_{M, re})&=&(\psi\circ\Phi^{S^2}_{\#})(f_{M, re}),\\
\beq\label{eq:mid}
(\psi\circ{\Phi_{S^n}}_{\#})(f_{M, j})&=&\Pi(\{h_{j,m}\} \{e_{M+1}(d_m^{1/2})\}\{e_{M}(d_m^{1/2})\})\\
&=&\Pi(\{h_{j,m} e_{M}(d_m^{1/2})\}),  
%\label{eq:mid}
\eneq
for $1\le j\le n.$
Combining (\ref{eq:gf}) and (\ref{eq:mid}), we obtain
\beq \label{eq:psig}
 \psi(g_{M, j})=({\psi\circ\Phi_{S^n}}_{\#})(f_{M, j})
 =\Pi(\{h_{j,m} e_{M}(d_m^{1/2})\})
\eneq
and then
\[
\lim_{m\to\infty}\|\widetilde\psi_m(g_{M,j})-h_{j,m}e_{M}(d_m^{1/2})\|=0,
\]
that is, there exists $N_2>0$, so that for all $m>N_2$
\[
\|\widetilde\psi_m(g_{M,j})-h_{j,m}e_{M}(d_m^{1/2})\|<\epsilon,\,\,\, j=1,2,...,n.
\]

Similarly as in (\ref{eq:psig}), we have 
\beq
%&& \psi(g_{M, im})=(\psi\circ\Phi^{S^2}_{\#})(f_{M, im})
 %=\pi(\{h_{2,n} e_{M}((z_n^*z_n)^{1/2})\});\\
 && \psi(e_{M}^{S^n})=({\psi\circ\Phi_{S^n}}_{\#})(e_{M})
 ={\Pi}(\{e_{M}(d_m^{1/2})\}).
\eneq
Hence, there exists $N_3>0,$ for all $m>N_3,$ 
\beq
%&&\|\widetilde\psi_n(g_{M,im})-h_{2,n}e_{M}((z_n^*z_n)^{1/2})\|<\epsilon; \\
&& \|\widetilde\psi_m(e_{M}^{S^n})-e_{M}(d_m^{1/2})\|<\epsilon.
\eneq

Since $C(S^n)$ is commutative, 
%Next, note that in (\ref{eq:psig}), one can use a different decomposition $f_{M,re}=f_1'f_2',$ where $f_1'=e_{M}((\zeta^*\zeta)^{1/2})$ and $f'=e_{M+1}((\zeta^*\zeta)^{1/2}){\rm Re}(\zeta)$
%and obtain 
%\beq 
% \psi(g_{M, j})=(\psi\circ\Phi_{S^n}_{\#})(f_{M, j})
 %=\pi(\{e_{M}(d_m^{1/2})h_{j,m} \})
%\eneq
%Equating the right hand side with that of (\ref{eq:psig}) one has
we must have 
\[
\lim_{n\to\infty}\| e_{M}(d_m^{1/2})h_{j,m}-h_{j,m}e_{M}(d_m^{1/2}) \|=0
\]
%and one has similar result replacing $h_{1,n}$ by $h_{2,n}$. 
Thus, there exists $N_4>0$  such that, for all $m>N_4$, 
\[
\| e_{M}(d_m^{1/2})h_{j,m}-h_{j,m}e_{M}(d_m^{1/2}) \|<\epsilon
\]
for $j=1,
\ldots, n$.

Finally, since the $*$-homomorphism $\psi$ is multiplicative and $\widetilde\psi$ is its lift, for the given $\eta>0$, there exists $N_5$, such that for all $m>N_5$, the c.p.c. map $\widetilde\psi_m$ is $\eta$-multiplicative, i.e.,
\[
\|\widetilde\psi_m(x)\widetilde\psi_m(y)-\widetilde\psi_m(xy)\|<{\eta} { \tforal} x,y\in \mathcal{F}\subset C(S^n).
\]

To conclude, if we choose $m>\max\{N_1,\ldots, N_5\}$, then for the c.p.c. map $\widetilde\psi_m$, none of the inequalities (\ref{eq1})-(\ref{eq2}) holds. 
The lemma is then proved.
\end{proof}

%%%%%%%%%%%%%%%%%%%
\iffalse
\begin{df}
For any $n$-tuple  of (unbounded) self-adjoint operators $h_1,h_2,...,h_n$ 
(densely defined) on a Hilbert space $H$ and $d=\sum_{j=1}^n h_j^2.$
Fix $M\ge 1.$ In what follows denote by 
$P$ the spectral projection of $d^{1/2}$ associated with $[0, M].$

Denote by $e_j\in C(M^n)$  such that 
$e_j((t_1,t_2,...,t_n))=t_j,$ $1\le j\le n.$ 
\end{df}

\begin{cor}\label{Cadd1}
In Lemma \ref{Ls2app2}, define 
$\phi^M: C(M^n)\to B(H)$ by 
\beq
\phi^M(e_j)=P_M\phi(g_{M,j})P,\,\,1\le j\le n.
\eneq
For any finite subset ${\cal F}_M\subset C(M^n)$ 
(with possibly  smaller $\dt$),  we may also require that
\beq
\|\phi^M(fg)-\phi^M(f)\phi^M(g)\|<\eta\andeqn
\|\phi^M(e_j)-Ph_jP\|<\eta,
\eneq 
$j=1,2,...,n.$ 
\end{cor}

 \begin{proof}
 Fix $\ep>0,$ $\eta>0,$ $M\ge 1$ and  a finite subset ${\cal F}_M\subset C(M^n).$
 Choose $\sigma>0$ such that, for any c.p.c. map
 $\Phi: C(M^n)$ to any unital \CA,   that 
 \beq
 \|\Phi(e_ie_j)-\Phi(e_i)\Phi(e_j)\|<\sigma,\,\,\, 1\le i,j\le n,
 \eneq
 implies that 
 \beq
 \|\Phi(fg)-\Phi(f)\Phi(g)\|<\eta\rforal f,g\in {\cal F}_M.
 \eneq
 Choose $\ep_1=\min\{\sigma/2, \ep/4\}.$ 
 For this $\ep_1>0$ (instead of $\ep$), $\eta$ and $M,$ 
 choose $\dt$ be required by Lemma \ref{Ls2app2}.
 
 Let $\phi: C(S^n)\to B(H)$ be the c.p.c. map given by Lemma \ref{Ls2app2}
 (for $\ep_1>0$ instead of $\ep$ though).
 By \eqref{Lsapp2-5},
 \beq
 Ph_j=Pe_M(d^{1/2})h_j
 \eneq

 \end{proof}
\fi
%%%%%%%%%%%%%%%%%%%%%%%%%%
%

%%%%%%%%%%%%%%%%%%%%%%%%%%%%%%%
%{\red{One may also choose to prove the following first.}}

\begin{cor}\label{Cs2app2}
Let $\ep>0,$ $\eta>0$ and $n\in \N.$  Let ${\cal F}$ be a finite subset of $C(S^n).$
Let $M\ge 1.$
Then there exists $\dt>0$ satisfying the following:
For any $n$-tuple of self-adjoint operators $h_1,h_2,...,h_n$ densely defined on an infinite dimensional 
separable Hilbert space $H$ such that
\beq
&&\|[\tilde h_i,\,\tilde h_j]\|<\dt,\,\,\,\|\tilde h_i-\tilde h_i^*\|<\dt, \quad i,j=1,2,...,n
%\tand\\
%&&e_M(d^{1/2})\not=0,
\eneq
where $\tilde h_j=h_j(1+d)^{-1}$ and $d=\sum_{j=1}^n h_j^2,$ 
there exists a c.p.c. map  $\phi: C(S^n)\to B(H)$ 
such that  (with $d=\sum_{j=1}^n h_j^2$)
\beq\label{Cs2app2-1}
&&\|\phi(r)-(-1+d)(1+d)^{-1}\|<\ep,\,\,\|\phi(x_j)-2h_j(1+d)^{-1}\|<\ep\\\label{Cs2app2-2} 
&&\|\phi(x_j)-2(1+d)^{-1/2}h_j(1+d)^{-1/2}\|<\ep  \\\label{Cs2app2-3}
&&\|\phi(g_{M, j})-h_je_M(d^{1/2})\|<\ep,\\\label{Cs2app2-4}
&&\|\phi(e_M^{S^n})-e_M(d^{1/2})\|<\ep\tand\\\label{Cs2app2-5}
&&\|e_M(d^{1/2})h_j-h_je_M(d^{1/2})\|<\ep,
\eneq
$j=1,2,...,n.$
Moreover, we may require that 
\beq
\|\phi(f)\phi(g)-\phi(fg)\|<\eta\tforal f,\, g\in {\cal F}.
\eneq
\end{cor}

\begin{proof}
%{\red{The corollary follows from the same argument as in Lemma~\ref{Ls2app2}, with the following modifications.}}
%The corollary follows the same argument as that of Lemma~\ref{Ls2app2}, except for the following modifications. 
First, we claim that we need only to prove the corollary without \eqref{Cs2app2-2}.
% (\ref{Cs2app2-2.5}). 
In fact, suppose $\delta'>0$ is found for $\epsilon/2$ so that the lemma is proved without 
\eqref{Cs2app2-2}.
%(\ref{Cs2app2-2.5}).
Then applying Lemma~\ref{Lemtildehjbarh}, for $\epsilon/2$, one chooses $\delta''>0$, so that whenever $\|\tilde h_j-\tilde h_j^*\|<\delta''$ for $j=1,\ldots, n$, we have 
\[
\|h_j(1+d)^{-1}-(1+d)^{-1/2}h_j(1+d)^{-1/2}\|<\epsilon/4 \quad j=1, \ldots, n.
\]
Then set $\delta=\min\{\delta', \delta''\}>0.$
Whenever $\|[\tilde h_i, \tilde h_j]\|<\delta$ and $\|\tilde h_j-\tilde h_j^*\|<\delta$ for $i,j=1,2,\ldots, n$, we additionally obtain 
{{\begin{multline*}
\|\phi(x_j)-2(1+d)^{-1/2}h_j(1+d)^{-1/2}\| \\
\le \|\phi(x_j)-2h_j(1+d)^{-1}\|+2\|h_j(1+d)^{-1}-(1+d)^{-1/2}h_j(1+d)^{-1/2}\|<\epsilon
\end{multline*} }}
and the claim is proved.

The corollary  then follows from the same argument as in Lemma~\ref{Ls2app2}, with the following modifications.

In the paragraph below (\ref{eqanbn}), one needs to apply Lemma~\ref{Lappcom-1nbd} instead: 

Given that $\|[\tilde h_{i,m}, \tilde h_{j,m}]\| < \frac{1}{m}$, $\|\tilde h_{i,m}-\tilde h_{i,m}^*\|<\frac1{m}$ and applying Lemma~\ref{Lappcom-1nbd} we obtain
 $-1\le a_m\le 1$, $\sum_{j=1}^n b_{j,m}^*b_{j,m}=1-a_m^2$ and additionally 
\[
\|[a_m,b_{j,m}]\|<\frac{2}{m}, ~~\|b_{j,m}-b_{j,m}^*\|<\frac{2}{m}~~\text{and}~~\|[b_{i,m},b_{j,m}]\|<\frac{2}{m}.
\] 
\end{proof}

\section{Synthetic spectrum}

\begin{df}\label{Dnsp}
Let $A$ be a  unital  \CA\, and 
$\Omega$ a compact metric space.
Fix a finite subset ${\cal F}\subset C(\Omega)^{\bf 1}.$ 
%Suppose that $C(\Omega)$ is generated by $e_1,e_2,...,e_n$
%with $\|e_i\|\le 1,$ $i=1,2,...,n.$ 
Suppose that $L: C(\Omega)\to A$ is a c.p.c. map.
Fix $0<\eta<1/2.$ 
Suppose that $X\subset \Omega$ is a compact subset and 
%$a_1,a_2,...,a_n\in A_{s.a.}$ for some $n\in \N.$
%Let us assume that $\|a_i\|\le 1,$ $i=1,2,...,n.$ 
%Fix $0<\eta<1/2.$ Suppose  that $X\subset \I^n$ is a compact subset and 
$\phi: C(X)\to A$ is a unital c.p.c. map
such that 
\beq
%\hspace{-2.1in}{\rm (i)} && \{e_j|_X: 1\le j\le n\}\,\,\, {\rm generates}\,\,\,  C(X);\\
%
\hspace{-2.1in}{\rm (i)} &&\|L(f|_X)-\phi(f|_X)\|<\eta \rforal f\in {\cal F}, \\
\hspace{-2.1in}{\rm (ii)}&&\|\phi((fg)|_X)-\phi(f|_X)\phi(g|_X)\|<\eta\rforal f, g\in {\cal F},\andeqn\\
\hspace{-2.1in}{\rm (iii)} &&  \|\phi(f)\|\ge 1-\eta
\eneq
for any $f\in C(X)_+$  which has 
value $1$ on an open ball with the center $x\in X$ (for every $x\in X$) and the radius $\eta.$
Then we say that $X$ is an $\eta$-near-spectrum of $L$  associated with ${\cal F}.$
%the $n$-tuple 
%$(a_1, a_2,...,a_n).$ 
We write 
$$
n{\rm Sp}^\eta(L, {\cal F}):=X.
$$

If, moreover,  $\phi$ is a unital \hm,  then  we say $X$ is 
an $\eta$-spectrum of $L.$   

Suppose that $C(\Omega)$ is  generated by $e_1, e_2,...,e_n\in C(\Omega)^{\bf 1}.$
We may choose ${\cal F}=\{e_1,e_2,...,e_n\}.$ In that case (when $e_1,e_2,...,e_n$ are understood), we may write 
\beq
n{\rm Sp}^\eta(L):=n{\rm Sp}^\eta(L, {\cal F}).
\eneq

Although $n{\rm Sp}^\eta(L, {\cal F})$ is not uniquely defined, the notation is convenient to employ. Its definition is unique up to an $\eta$-neighborhood of $X$, and for our purposes, we will only need the special case established in Proposition \ref{Pspsub}.
\end{df}

\begin{lem}\label{Lnsp}
Let $\Omega$ be a compact metric space
and ${\cal F}\subset C(\Omega)$ 
be a fixed finite subset.
% (remove? also stated below)}.  
Let $\eta>0.$ 
%and ${\cal F}\subset C(\Omega)$
  Then there is $\dt>0$ and a finite subset ${\cal G}\subset C(\Omega)$
satisfying the following:

Suppose that $L: C(\Omega)\to A$ is a positive linear map, (for any unital \CA\, $A$)  such that
$\|L(1)\|\ge 1/2$ and 
\beq
\|L(fg)-L(f)L(g)\|<\dt\tforal f,g\in {\cal G}.
\eneq
Then $n{{\rm Sp}}^\eta(L, {\cal F})$ exists, i.e., there exists  a compact subset 
$X\subset \Omega$ and 
a c.p.c{.} map $\phi: C(X)\to A$ such that
\beq
%\hspace{-2.1in}{\rm (i)} && \{e_j|_X: 1\le j\le n\}\,\,\, {\rm generates}\,\,\,  C(X);\\
%
\hspace{-2.1in}{\rm (i)} &&\|L(f)-\phi(f|_X)\|<\eta \rforal f\in {\cal F}, \\
\hspace{-2.1in}{\rm (ii)}&&\|\phi((fg)|_X)-\phi(f|_X)\phi(g|_X)\|<\eta\rforal f, g\in {\cal F},\andeqn\\
\hspace{-2.1in}{\rm (iii)} &&  \|\phi(g|_X)\|\ge 1-\eta
\eneq
for any $g\in C(X)_+$  which has 
value $1$ on an open ball with the center $x\in X$ (for some $x$) and the radius $\eta.$

Moreover, if $C(\Omega)$ has a finite generating subset ${\cal F}_g,$ 
we may choose ${\cal G}={\cal F}_g.$
\end{lem}

%%%%%%%%%%%%%
\iffalse
Let $X$ be a compact metric space and $0<\eta<1.$ 
Fix a finite $\eta$-dense net $\{x_1, x_2,..., x_m\}\subset X,$ i.e.,
for any$x\in X,$  there exists $i\in \{1,2,...,m\}$ such that
\beq
{\rm dist}(x, x_i)<\eta.
\eneq 
For each $x_i\in \{x_1, x_2,...,x_m\},$ define 
$\theta_{x_i, \eta}\in C(X)_+$ such that $\theta_{i, \eta}(x)=1,$
if ${\rm dist}(x, x_i)<\eta/2,$ $\theta_{i, \eta}(x)=0$ if 
${\rm dist}(x, x_i)\ge \eta.$ 

Let $A$ be a unital \CA\,  and $L: C(X)\to A$ be a contractive completely positive linear map.
Set 
\beq
X_{L, \eta}=\cup_{\|L(\theta_{x_i,\eta})\|\ge 1-\eta} B(x_i, \eta),
\eneq
where $B(x_i, \eta)=\{x\in  X: {\rm dist}(x, x_i)\le \eta\},$ $i=1,2,...,m.$

The set $X_{L, \eta}$ is called the $\eta$-synthetic spectrum of $L.$ 
\fi
%%%%%%%%%%%%%%%
%\end{df}

\begin{proof}
Suppose that the lemma is false.  Then there is
a finite subset $\mathcal{F}_0$ and $\eta_0>0,$ 
%, {\red{for all finite $\mathcal G\subset C(\Omega)$, the following is satisfied:}} %satisfying the following:
%There exists 
a sequence of unital \CA s $\{A_m\}$ and  a sequence 
of positive linear maps $L_m: C(\Omega)\to A_m$ with  $\|L_m(1)\|\ge 1/2$ and 
\beq\label{Lnsp-5}
\lim_{m\to\infty}\|L_m(fg)-L_m(f)L_m(g)\|=0\rforal f, g\in
% {\red{\mathcal G}} %
C(\Omega)
\eneq
such that there is no compact subset $X\subset \Omega$ and 
c.p.c. map $\phi_m: C(X)\to A_m$ satisfying
 (i), (ii) or (iii) above, i.e.,  one of the following does not hold:
 \beq
%\hspace{-2.1in}{\rm (i)} && \{e_j|_X: 1\le j\le n\}\,\,\, {\rm generates}\,\,\,  C(X);\\
%
\hspace{-1.3in}{\rm (i)} &&\|L_m(f)-\phi_m(f|_X)\|< \eta_0 \rforal f\in {\cal F}_0, \\
\hspace{-1.3in}{\rm (ii)}&&\|\phi_m((fg)|_X)-\phi_m(f|_X)\phi_m(g|_X)\|<\eta_0\rforal f, g\in {\cal F}_0,\andeqn\\
\hspace{-1.3in}{\rm (iii)} &&  \|\phi_m(g|_X)\|\ge 1-\eta_0
\eneq
for any $g\in C(X)_+$  which has 
value $1$ on an open ball with the center $x\in X$ (for some $x$) and the radius $\eta_0.$

Define $\Lambda: C(\Omega)\to \prod_{m=1}^\infty A_m$ by
$\Lambda(f)=\{L_n(f)\}$ for all $f\in C(\Omega).$
Set  $Q=\prod_{m=1}^\infty A_m/\bigoplus_{m=1}^\infty A_m$ and let $\Pi: \prod_{m=1}^\infty A_m\to Q$ be the 
quotient map.

Consider the map $\Psi=\Pi\circ \Lambda: C(\Omega)\to Q.$   By \eqref{Lnsp-5},
$\Psi$ is a \hm. 
%{\red{(\eqref{Lnsp-5} has $\mathcal G$ instead of $C(\Omega)$, still true? Maybe $\mathcal G$ should not appear in the statement?)}} 
Since $\|L_m(1)\|\ge 1/2,$ $\Lambda(1)\not\in \bigoplus_{m=1}^\infty A_m.$
In other words, $\Psi(1)\not=0.$  Denote $J={\rm ker}\Psi.$ 
Hence there exists a compact subset $X\subset \Omega$ such 
that $J=\{g\in C(\Omega): g|_X=0\}.$ Hence, $C(\Omega)/J\cong C(X).$ 
Then $\Psi$ induces an injective \hm\, $\Phi: C(X)\to Q$ by
$\Psi=\Phi\circ q,$ where $q: C(\Omega)\to C(X)$ is the quotient map. 

By the Choi-Effros lifting theorem (\cite{CE}), there is a c.p.c.map
$F: C(X)\to \prod_{m=1}^\infty A_m$ such that
\beq
\Pi\circ F=\Phi.
\eneq
We may write $F(f)=\{F_m(f)\}$ for all $f\in C(X),$ where each $F_m: C(X)\to A_m$ 
is a c.p.c. map. 
Since $\Psi=\Phi\circ q,$ 
\beq
\lim_{m\to\infty}\|F_m(q(f))-L_m(f)\|=0\rforal f\in C(\Omega),
\eneq
or equivalently,
\beq
\lim_{m\to\infty}\|F_m(f|_X)-L_m(f)\|=0\rforal f\in C(\Omega).
\eneq
Moreover,  by \eqref{Lnsp-5},
\beq
\lim_{m\to\infty}\|F_m(fg|_X)-F_m(f|_X)F_m(g|_X)\|=0\rforal f, g\in 
%\red{\mathcal G}%
C(\Omega).
\eneq
Since ${\cal F}_0$ is a finite subset, there is $n_1\in \N$ such that, for all $m\ge n_1,$
\beq\label{Lnsp-12}
&&\|L_m(f)-F_m(f|_X)\|<\eta_0\rforal f\in {\cal F}_0\andeqn\\\label{Lnsp-13}
&&\|F_m((fg)|_X)-F_m(f|_X)F_m(g|_X)\|<\eta_0\rforal f\in {\cal F}_0.
\eneq
Since 
$X$ is compact, there are $x_1, x_2,..., x_k\in X$ 
such that
\beq
X\subset \bigcup_{i=1}^k B(x_i, \eta_0/4).
\eneq
Let $g_1, g_2,...,g_k\in C(X)_+$ such that $0\le g_i\le 1,$ 
$g_i(x)=1$ when $x\in B(x_i, \eta_0/8),$ and $g_i(x)=0$ when $x\not\in B(x_i, \eta_0/4),$
$i=1,2,...,k.$ 
Then, for any  $x\in X,$ there is $i\in \{1,2,...,k\}$ such that $x\in B(x_i, \eta_0/4).$ 
In other words,  for any $x\in X,$ there is $i\in \{1,2,...,k\}$ such that
\beq
B(x_i,\eta_0/4)\subset B(x, \eta_0).
\eneq
Therefore, for any $g\in C(X)_+$  which has 
value $1$ on an open ball with the center $x\in X$ (for some $x$) and radius $\eta_0,$
there is $i\in \{1,2,...,k\}$ such that
\beq\label{Lnsp-15}
g\ge g_i.
\eneq

Since $\Phi$ is injective,  $\|\Phi(g_i)\|=1,$ $i=1,2,...,k.$
There is $n_2\in \N$ such that, for all $m\ge n_2,$ 
\beq\label{Lnsp-16}
\|F_m(g_i)\|\ge 1-\eta_0,\,\,\, i=1,2,...,k.
\eneq
Since $F_m$ is positive, by \eqref{Lnsp-15}, when $m\ge n_2,$
\beq\label{Lnsp-17}
\|F_m(g)\|\ge 1-\eta_0
\eneq
for any $g\in C(X)_+$  which has 
value $1$ on an open ball with the center $x\in X$ (for some $x$) and radius $\eta_0.$

Now choose $n_3=\max\{n_1, n_2\}.$  For $m\ge n_3,$ by  \eqref {Lnsp-12}, \eqref{Lnsp-13} and 
\eqref{Lnsp-17}, we obtain a contradiction.
\end{proof}

\begin{df}\label{Dsynsp2}
Suppose that $\{h_1, h_2,...,h_n\}$ is an $n$-tuple of densely defined self-adjoint operators on an infinite dimensional separable Hilbert space
$H$ as in Definition~\ref{Dntuple}. 
Fix $0<\eta<1$ and a countable subset $D^\eta\subset \R^n$ such that, for any 
$\xi=(\lambda_1, \lambda_2,..., \lambda_n)\in \R^n,$ 
there is $y\in {{D^\eta}}$ such that
\beq
\|\xi-y\|_2<\eta/2,
\eneq
and, for any $k\in \N,$ 
\beq
D^\eta\cap \{\zeta\in \R^n: \|\zeta\|_2\le k\}
\eneq
is a finite subset. Such $D^{\eta}$ is called $\eta/2$-dense.

%Let $A$ be a unital \CA\, and $(a_1,a_2,...,a_n)$ be an $n$-tuple  of self-adjoint elements in 
%$A$ with $\|a_i\|\le M$ ($1\le i\le n$). \
Fix $\xi=(\lambda_1, \lambda_2,...,\lambda_n)\in \R^n.$
Let 
$\theta_{\lambda_i,\eta}\in C(\R)$ be such that 
$0\le \theta_{\lambda_i, \eta}\le 1,$ $\theta_{\lambda_i, \eta}(t)=1,$ if $|t-\lambda_i|\le 3\eta/4,$ 
$\theta_{\lambda_i,\eta}(t)=0$
if $|t-\lambda_i|\ge  \eta,$  and $\theta_{\lambda_i, \eta}$ is linear in $(\lambda_i-\eta, \lambda_i-3\eta/4)$
and in $(\lambda_i+3\eta/4, \lambda_i+\eta),$
%$3\eta/4\le 
%|t-\lambda_i|\le \eta,$
for $i=1,2,...,n.$ 
 
 Define, for $(t_1, t_2,...,t_n)\in \R^n,$
 % for $j\in \{1,2,...,m\},$  
\beq
\Theta_{\xi, \eta}(t_1,t_2,...,t_n)&=&\prod_{i=1}^n\theta_{\lambda_i, \eta}(t_i)\andeqn \label{eq:ThetaO}
\\ \label{Dpsp-5}
\Theta_{\xi, \eta}(h_1,h_2,...,h_n)&=&\theta_{\lambda_1,\eta}(h_1)\theta_{\lambda_2,\eta}(h_2)\cdots 
\theta_{\lambda_n,\eta}(h_n).
\eneq
Note that  we do not assume that $h_1, h_2,...,h_n$ mutually commute and  the product in \eqref{Dpsp-5} has 
a fixed order.

Set 
% ($B(x_j, \eta)=\{x\in \I^n: {\rm dist}(x, x_j)<\eta\}$)
\beq
s{\rm Sp}^\eta((h_1,h_2,..., h_n))=
%\bigcup\{ \overline{B(x_j, \eta)}: \|\Theta_{x_j, \eta}(a_1,a_2,...,a_n)\|\ge 1-\eta\},
\bigcup{}_{_{\xi \in {D}^\eta, \|\Theta_{\xi, \eta}(h_1,h_2,...,h_n)\|\ge 1-\eta}} \overline{B(\xi, \eta)}. 
\eneq

Fix $M\ge 1$ and a finite subset $D_M^\eta\subset \{\zeta\in \R^n: \|\zeta\|_2\le M\}$ 
which is {$\eta$}-dense in  $\{\zeta\in \R^n: \|\zeta\|_2\le M\}.$
Write $D^\eta_M=\{\xi_1, \xi_2,...,\xi_N\}.$

For each $M>0,$ write 
\beq
s{\rm Sp}_M^\eta((h_1,h_2,..., h_n))=\bigcup{_{\xi_j\in {D_M^\eta},
%_{x\in {\cal F}_\eta, \|x\|_2\le M,
 \|\Theta_{
 \xi_j, \eta}(h_1,h_2,...,h_n)\|\ge 1-\eta}} \overline{B(\xi_j, \eta)}. 
\eneq
Note that $s{\rm Sp}_M^\eta((h_1,h_2,..., h_n))$ is a union of finitely 
many closed balls, and therefore a compact subset of $\R^n.$ 
%and $j=1,2,...,m.$

%Let $\mathbb P_k=\{m/2^k: m\in \Z\}$ and 
For $k\in\mathbb N$, let 
\beq
{\mathbb P}_k^M=\bigcup_{l=1}^k\{\xi=(x_1, x_2,...,x_n): x_j=m_j/l,|m_j|\le Ml, \, m_j\in \Z, 1\le j\le n\}.
\eneq
For $1>\eta>0,$ choose 
$k=\inf\{l\in \N: (M+1)/l<{\eta\over{2\sqrt{n}}}\}.$  Then ${\mathbb P}_k^M$ is $\eta/2$ dense 
in $\{\zeta\in \R^n: \|\zeta\|_2\le M\}.$
%{\mathbb P}_k\cap \{\zeta\in \R^n: \|\zeta\|_2\le M\}.$
%Note that ${\mathbb P}_{k,M}$ is $1/2^k$-dense in $\{\zeta\in \R^n: \|\zeta\|_2\le M\}.$
For convenience, in what follows, we may always assume 
that $D^\eta_M={\mathbb P}_k^M.$  In particular $D^\eta\subset D^\dt,$ if $\eta>\dt.$
\end{df}

\begin{rem}\label{Rmspectrum}

In general, let $A$ be a unital \CA\, and $a_1, a_2,...,a_n\in A_{s.a.}$ with 
$\|a_i\|\le M,$ $1\le i\le n.$  One defines $s{\rm Sp}^\eta((a_1,a_2,...,a_n))$ as in Definition~2.7 of \cite{Lincmp2025}.  Then $s{\rm Sp}^\eta(( a_1,a_2,...,a_n))=s{\rm Sp}^\eta_{\sqrt{n}M}((a_1,a_2,...,a_n)).$ 
%{\red{(Would $s{\rm Sp}^\eta(( a_1,a_2,...,a_n))$ be cubical? Maybe $M$ should be replaced by a larger one say $\sqrt n M$?)}}
One may notice that, unlike $n{\rm Sp}^\eta(L, {\cal F}),$ %{\red{The synthetic spectrum}} 
$s{\rm Sp}^\eta_M((a_1,a_2,...,a_n))$ is uniquely defined.

\begin{df}\label{DfSSP}
{{Fix $0<\eta<1/4.$  Fix 
$$
\xi=(t_0, t_1,t_2,...,t_n)\in S^n=\{(r, \zeta)\in [-1,1]\times [-1,1]^n:   
-1 \le r \le 1,\, \|\zeta\|_2 \le 1,\, \|\zeta\|_2^2 = 1 - r^2 \}.
$$
Define, for each $x=(s_0, s_1, s_2,...,s_n)\in S^n\subset \R^{n+1},$  
\beq \label{eq:ThetaS}
\Theta^S_{\xi, \eta}(x)=\theta_{t_0, \eta}(s_0)\theta_{t_1,\eta}(s_1)\cdot\cdots\cdot \theta_{t_n,\eta}(s_n).
\eneq     
}}
{{Suppose that $h_1, h_2,...,h_n$ form an $n$-tuple densely defined self-adjoint operators on an infinite dimensional separable Hilbert space
$H$ as in Definition~\ref{Dntuple}. 
Define (recall that $d=\sum_{j=1}^n h_j^2$ and $\bar h_j=(1+d)^{-1/2}h_j(1+d)^{-1/2},$ $1\le j\le n$)
\beq \label{eq:ThetaSh}
\Theta^S_{\xi, \eta}(h_1,h_2,...,h_n)&=&
\theta_{t_0,\eta}((-1+d)(1+d)^{-1})\theta_{t_1,\eta}(2\bar h_1)\cdots 
\theta_{t_n,\eta}(2\bar h_n).
\eneq
Put $\Sigma^{\eta}=\Phi_{S^n}(D^\eta).$ 
Set
\beq
s{\rm Sp}_{S^n}^\eta((h_1, h_2,...,h_n))=\bigcup{}_{_{x_j\in \Sigma^\eta, \|\Theta^S_{x_j, \eta}(h_1,h_2,...,h_n)\|\ge 1-\eta}} \overline{D(x_j, \eta)},
\eneq
where $D(x_j, \eta)=\{\xi\in S^n\setminus \zeta^{np}: \|\xi-x_j\|_2<\eta\}$
and $\|\cdot\|_2$ is the $(n+1)$-dimensional Euclidean norm. }}
{{Define 
\beq
S{\rm Sp}^\eta=\Phi_{S^n}^{-1}(s{\rm Sp}_{S^n}^\eta). 
\eneq
Fix $M\ge 1,$ and let $D_M^\eta$ be as in Definition \ref{Dsynsp2}. 
Put $\Sigma_M=\Phi_{S^n}(D_M^\eta).$
Define
\beq
&&s{\rm Sp}_{S^n, M}^\eta((h_1, h_2,...,h_n))=\bigcup{}_{_{x_j\in {\Sigma_M}, \|\Theta^S_{x_j, \eta}(h_1,h_2,...,h_n)\|\ge 1-\eta}} \overline{D(x_j, \eta)}\andeqn\\
%
%Put $\Sigma^M_=\Phi_{S^n}(D^\eta).$
%Define 
\label{DSSPM}
&&S{\rm Sp}_M^\eta((h_1,h_2,...,h_n))=\Phi_{S^n}^{-1}\left(\bigcup{}_{_{x_j\in {\Sigma_M}, \|\Theta_{x_j, \eta}^{S}(h_1,h_2,...,h_n)\|\ge 1-\eta}} \overline{D(x_j, \eta)}\right).
\eneq
%It  should be noted that, if $0<{\red{4\eta<\dt}}<1,$ 
%$0<\eta<{\red{\dt/4}}<1,$ 
%then 
%$S{\rm Sp}_M^\eta((h_1,h_2,..., h_n))\subset S{\rm Sp}_M^\dt((h_1,h_2,..., h_n)).$
}}
\end{df}
%
%%%%%%%%%%%%%%%%%%%%%%%%%%%%
%
%
%%%%%%%%%%%%
\iffalse
Now back to the Hilbert space $H$ and $h_1,h_2,...,h_n$ be as in Definition \ref{Dsynsp2}
Let $a_i=h_i e_M(d^{1/2}),$ $i=1,2,...,n.$
Choose $0<\eta<1/4.$ 
Then, with $D^\eta$ as above. 
\beq
s{\rm Sp}^\eta((a_1, a_2,...,a_n))=a{\rm Sp}^\eta_M((h_1,h_2,...,h_n)).
\eneq
%
To see this, write $D^\eta=\{\xi_1, \xi_2,...,\xi_N\}$ as in Definition \ref{Dsynsp2}. 
%
\fi
%%%%%%%%%%%%%%%
\end{rem}

\begin{lem}[{\cite[Lemma 2.12]{Linself}}]\label{Lpurb}
Let $X$ be a compact metric space, ${\cal G}\subset C(X)$ be 
a finite 
%generating 
subset, and $C\subset C(X)$ be  the \SCA\, generated by ${\cal G}.$

Then, for any $\ep>0$ and any finite subset ${\cal F}\subset  C\subset C(X),$ 
there exists $\dt({\cal F}, \ep)>0$ satisfying the following:
Suppose that $A$ is a unital \CA\, and $\phi_1, \phi_2: C(X)\to A$ are 
unital c.p.c. maps such that
\beq
\|\phi_1(g)-\phi_2(g)\|<\dt \andeqn 
\|\phi_i(gh)-\phi_i(g)\phi_i(h)\|<\dt \tforal g, h\in {\cal G},\,\,\, i=1,2.
\eneq
Then, for all $f\in {\cal F},$ 
\beq
\|\phi_1(f)-\phi_2(f)\|<\ep.
%\andeqn \|\phi_1(f(e_j))-f(\phi_1(e_j))\|<\ep
\eneq

Moreover, if $X\subset \{\zeta\in \R^n: \|\zeta\|_2\le M\},$ and  ${\cal G}=\{1, e_1|_X,e_2|_X,...,e_n|_X\},$
then, for any $\ep>0$ and 
any finite subset  ${\cal H}\subset C_0( [-M,M]),$ 
we may also require
that, for all $h\in {\cal H},$ 
\beq
\|\phi_1(h(e_j|_X))-h(\phi_1(e_j|_X))\|<\ep,\,\,\, 1\le j\le n.
\eneq
\end{lem}

The following is an easy fact and follows from  the proof of \cite[Corollary 2.13]{Linself}.

\begin{cor}{\rm (\cite[Corollary 2.13]{Linself})}\label{C213}
Fix $n\in \N,$  $0<\eta<1,$   
%$M>1$ 
 and $\Sigma=\{\xi_1, \xi_2,...,\xi_m\}\subset S^n\setminus \{\zeta^{np}\}$ which is 
 $\eta$-dense in $S^n.$
 %$\Phi_{S^n}(D^\eta)=\{\xi_0, \xi_1, \xi_2,...,\xi_m\}\subset  S^n$). 
 Let $\{r,x_1, \dots, x_n\}$ be the generating set of $C(S^n)$ as in Remark~\ref{RemSn}.
There exists $\dt(n, \eta)>0$ satisfying the following:
Suppose that $h_1, h_2,...,h_n$ form an $n$-tuple of densely defined 
(unbounded) self-adjoint operators
%Suppose that $A$ is a unital \CA,
%$a_1, a_2,...,a_n\in A_{s.a.}$ with $\|a_i\|\le M$ ($1\le i\le n$),
% $X\subset M^n$ is a 
%compact subset,   $I=\{f\in C_0(\R^n): f((t_1, t_2,...,t_n))=0\,\,\,{\rm if}\,\,(t_1,t_2,...,t_n)\not\in (M+1)^n\},$
and $L: 
%{\widetilde{C_0(\R^n)}}
C(S^n)\to B(H)$ is a unital c.p.c. map
such that, for $i,j=1,2,...,n,$ 
\beq\label{Ctheta-1}
&&\hspace{-0.5in}\|L(x_ix_j)-L(x_i)L(x_j)\|<\dt, \quad  \|L(x_ir)-L(x_i)L(r)\|<\dt,  \\
&& \|L(x_j)-2(1+d)^{-1/2}h_j(1+d)^{-1/2}\|<\dt \,
\\
&&
\andeqn \|L(r)-(-1+d)(1+d)^{-1}\|<\dt.
%\|L(g_{M,j})-h_je_M(d^{1/2})\|<\dt
%\|L(f_{M,i}f_{M,j})-L(f_{M,i})L(f_{M,j})\|<\dt\andeqn \|L(f_{M,j})-h_je_{M}(d^{1/2})\|<\dt
\eneq
%(see Definition \ref{dfMR} for $f_{M,i}$ and $e_M(d^{1/2})$).
%where $\tilde e_j\in I_+$ such that
%$\tilde e_j((t_1,t_2,...,t_n))=t_je_M(\sqrt{\sum_{i=1}^n t_i^2}),$ $1\le j\le n.$ 
Then, for all $\xi_k\in  \Sigma,$ 
%for  $D^\eta=\{x_1, x_2,...,x_m\}\subset  \{x\in \R^n: \|x\|_2\le M\},$
\beq
\|L( \Theta^S_{\xi_k, \eta})-\Theta^S_{\xi_k, \eta}(h_1,h_2,...,h_n)\|<\eta, \,\,\, 1\le k\le m.
\eneq

%{\blue Previously, it was written $\Theta^S_{\xi_k, \eta}|_X$, which does not make sense?}
\end{cor}

%{\red{The above requires a proof, or perhaps, a sketch of proof?}}

\begin{proof}
By \eqref{eq:ThetaSh}, we have $\Theta^S_{\xi_k, \eta}(h_1, \ldots, h_n)=\Theta_{\xi_k, \eta}((-1+d)(1+d)^{-1}, 2\bar h_1, \ldots, 2\bar h_n).$
Note that $\|(-1+d)(1+d)^{-1}\|\le 1$ and $\|2\bar h_j\|\le 1$. 
In fact, the first inequality follows from 
$(-1+d)(1+d)^{-1}=1-2(d+1)^{-1}$. The second inequality follows from (v) of Lemma \ref{Linqx}. 

\iffalse
%%%%%%%%%%
For the second, 
\beq
\|2\bar h_j\|
\eneq
from $(1+d)^{-1/2}(\sqrt d-1)^2(1+d)^{-1/2}\ge 0$, we have 
\[
1=(1+d)^{-1/2}(1+d)(1+d)^{-1/2}\ge (1+d)^{-1/2}(2\sqrt d)(1+d)^{-1/2}\ge 2(1+d)^{-1/2}|h_j|(1+d)^{-1/2}.
\]
\fi
%%%%%%%%%%%
Then the result follows by applying Corollary 2.13 of \cite{Linself}, for $Y=S^n\subset \I^{n+1}$ and $a_1=(-1+d)(1+d)^{-1}, a_2=2\bar h_1, \ldots, a_{n+1}=2\bar h_n$. 
\end{proof}

\begin{df}\label{Dspherical}
Suppose that $h_1, h_2,...,h_n$ form  an $n$-tuple of  
self-adjoint  densely defined operators on an infinite dimensional separable Hilbert space
$H$ as in Definition~\ref{Dntuple}.   
Define $d=\sum_{j=1}^n h_j^2.$
%Define $z=h_1+i h_2.$  
%
Let ${\cal F}_g=\{r, x_1, x_2,...,x_{n}\}\subset C(S^n)$
be as in Remark~\ref{RemSn}. 
%Let ${\cal F}_g\subset C(S^n)$ be a finite subset which also contains 
%$r:(r, \zeta)\to r,$ $x_j: (r,x_1, x_2,...,x_n)\to x_j$ ($\zeta=(x_1,x_2,...,x_n)$)
%and $g_1$ (with 
Note that $r^2+\sum_{j=1}^n x_j^2=1.$

Fix $\eta>0$ and $M\ge 1.$ 
Suppose that $X\subset S^n$ is a compact subset
%$a_1,a_2,...,a_n\in A_{s.a.}$ for some $n\in \N.$
%Let us assume that $\|a_i\|\le 1,$ $i=1,2,...,n.$ 
%Fix $0<\eta<1/2.$ Suppose  that $X\subset \I^n$ is a compact subset 
and $\phi: C(X)\to B(H)$ is a unital c.p.c. map
such that 
\beq
%\hspace{-2.1in}{\rm (i)} && \{e_j|_X: 1\le j\le n\}\,\,\, {\rm generates}\,\,\,  C(X);\\
%
\hspace{-2.1in}{\rm (i_1)} &&\|\phi(r|_X)-(-1+d)(1+d)^{-1}\|<\eta,\,\,\, \\
\hspace{-2.1in}{\rm (i_2)}&& \|\phi(x_j|_X)- 2(1+d)^{-1/2}h_j(1+d)^{-1/2}\|<\eta,\\
\hspace{-2.1in}{\rm (i_3)}&&\|\phi(g_{M,j}|_X)-h_je_M(d^{1/2})\|<\eta,\\
\hspace{-2.1in}{\rm (ii)}&&\|\phi((fg)|_X)-\phi(f|_X)\phi(g|_X)\|<\eta \rforal f, g\in {\cal F}_g
\\
%\hspace{-2.1in}{\rm (ii_1)}&&\|\phi(z)^*\phi(z)-\phi(z)\phi(z)^*\|<\eta,\,\,
\andeqn\\
\hspace{-2.1in}{\rm (iii)} &&  \|\phi(f)\|\ge 1-\eta
\eneq
for any $f\in C(X)_+$  which has 
value $1$ on an open ball with the center $x\in X$ (for some $x$) and the radius $\eta.$
Then we say that $X$ is an $M$-$\eta$-near-spherical spectrum of the $n$-tuple  $h_1, h_2,..., h_n.$ 
%the $n$-tuple 
%$(a_1, a_2,...,a_n).$ 
We write 
$$
n{\rm Sp}_{S^n, M}^\eta(h_1, h_2,...,h_n; {\cal F}_g):=X.
$$
%{\blue $X$ does not seem to have anything to do with $M$?}
%{\red{where $S_M=\Phi_{S^n}(\{\zeta\in \R^n: \|\zeta\|_2\le M\}.$}}

%%%%%%%%%%
\iffalse
One checks that, if $0<\eta'<\eta,$ 
then 
\beq
n{\rm Sp}_{S^n}^{\eta'}(h_1, h_2,...,h_n; {\cal F}_g)
\eneq
\fi
%%%%%%%
%{\red{Do I miss a factor 2?}}

\end{df}

The following is similar to \cite[Proposition 2.11]{Linself}.

\begin{prop}\label{PextspMs}
Let $0<\eta<1/4,$ $n\in \N$ and $M>1.$ Let ${\cal F}_g$ be a finite generating subset of $C(S^n)$
and ${\cal F}\subset C(S^n)$ be any finite subset containing ${\cal F}_g.$
There exists $\dt>0$ satisfying the following:
For any $n$-tuple self-adjoint operators $h_1,h_2,...,h_n$ densely defined on an infinite dimensional 
separable Hilbert space $H$ such that
\beq\label{Pextsp-1}
&&\|[\tilde h_i,\,\tilde h_j]\|<\dt,\,\,\,\|\tilde h_i-\tilde h_i^*\|<\dt\tand\\\label{Pextsp-2}
&&{\rm sp}(d^{1/2})\cap [0, M]\not=\emptyset,
\eneq
where $\tilde h_j=h_j(1+d)^{-1}$ and $d=\sum_{i=1}^n h_j^2,$ $i,j=1,2,...,n,$ 
the following two properties are satisfied:

{\rm (i)} $n{\rm Sp}^\eta_{S^n, M}(h_1, h_2,...,h_n; {\cal F})\cap S_M
%\setminus \{\zeta_{np}\}
\not=\emptyset,$ where $S_M=\Phi_{S^n}(\{\zeta\in \R^n: \|\zeta\|_2\le M\});$

{\rm (ii)} $S{\rm Sp}_M^\eta((h_1, h_2,...,h_n))\not=\emptyset.$
\end{prop}

\begin{proof}
Let us prove part (i) first. 

Fix $0<\eta<1/4,$ $n\in \N$ and $M>1.$
Put $\Omega=S^n$ and 
choose $\eta_0=\eta/3.$  
Let ${\cal F}_1={\cal F}\cup \{r, e^{S^n}_M,  x_j, e_{M,j}, 1\le j\le n\}$ (where 
$r, e^{S^n}_M, e_{M,j}$ are as in  Definition          \ref{dfMR}).

 Choose $\dt_1>0$ (as $\dt$)  and a finite subset 
${\cal G}\subset C(S^n)$ be as given by Lemma~\ref{Lnsp} corresponding to 
$\eta_0$ (in stead of $\eta$) and ${\cal F}_1$ (instead of ${\cal F}$).   
As in the ``Moreover" part of Lemma~\ref{Lnsp}, we may 
choose ${\cal G}={\cal F}_g$ (with possibility of smaller $\dt_1$).

Let $\eta_1=\min\{\eta_0, \dt_1,1/4\}> 0.$ 

To apply Corollary \ref{Cs2app2}, choose $\dt$ be as given by Corollary \ref{Cs2app2}
for $\ep=\eta_1$  and $\eta_1$ (in place of $\eta$).

Now suppose that $h_1, h_2,...,h_n$ are given which satisfy condition \eqref{Pextsp-1} and \eqref{Pextsp-2}.
Keep the notation $a=(-1+d)(1+d)^{-1},$ $\bar h_j=(1+d)^{-1/2}h_j(1+d)^{-1/2},$ $j=1,2,\ldots, n,$ as in Definition~\ref{Dabj} and Lemma~\ref{Lemtildehjbarh}.
By applying Corollary \ref{Cs2app2}, we obtain a c.p.c. map
$\phi: C(S^n)\to A$ such that
\beq
\|\phi(fg)-\phi(f)\phi(g)\|<\eta_1\rforal f,g\in {\cal F}_g
\eneq
together with \eqref{Cs2app2-1}, \eqref{Cs2app2-2}, \eqref{Cs2app2-3},  \eqref{Cs2app2-4}, \and
\eqref{Cs2app2-5}
but replacing $\eta$ by $\eta_1.$ 
Applying Lemma~\ref{Lnsp}, we obtain 
a compact subset $X\subset S^n$ and a c.p.c. map
$\psi: C(X)\to A$ such that
\beq\label{Pextsp-15}
%\hspace{-2.1in}{\rm (i)} && \{e_j|_X: 1\le j\le n\}\,\,\, {\rm generates}\,\,\,  C(X);\\
%
\hspace{-2.1in}{\rm (i)} &&\|\phi(f)-\psi(f|_X)\|<\eta_0<\eta \rforal f\in {\cal F}_1, \\
\hspace{-2.1in}{\rm (ii)}&&\|\psi((fg)|_X)-\psi(f|_X)\psi(g|_X)\|<\eta_0<\eta\rforal f, g\in {\cal F}_1,\andeqn\\
\hspace{-2.1in}{\rm (iii)} &&  \|\psi(g|_X)\|\ge 1-\eta_0\ge 1-\eta
\eneq
for any $g\in C(X)_+$  which has 
value $1$ on an open ball with the center $x\in X$ (for some $x$) and the radius $\eta_0.$

By \eqref{Cs2app2-1}, \eqref{Cs2app2-3}  and \eqref{Pextsp-15},
\beq
&&\|\psi(r|_X)-(-1+d)(1+d)^{-1}\|<\eta_1+\eta_0<\eta,\\
&&\|\psi(x_j|_X)-2h_j(1+d)^{-1}\|<\eta_1+\eta_0<\eta,\,\,1\le j\le n,\\
&&\|\psi(g_{M,j}|_X)-h_je_M(d^{1/2})\|<\eta_1+\eta_0<\eta,\,\, 1\le j\le n.
\eneq

Also, by \eqref{Cs2app2-4} and \eqref{Pextsp-15},
\beq
&&\hspace{-1.2in}\|\psi(e_M^{S^n}|_X)-e_M(d^{1/2})\|\le \|\psi(e_M^{S^n}|_X)-\phi(e_M^{S^n})\|\\
&&\hspace{0.4in}+\|\phi(e_M^{S^n})-e_M(d^{1/2})\|<\eta_1+\eta_0<1/2.
\eneq
Set $S_M=\{\zeta\in S^n:  \|\Phi_{S^n}^{-1}(\zeta)\|_2\le M\}.$
By \eqref{Pextsp-2}, which implies $\|e_M(d^{1/2})\|=1$, we have that $\psi(e_M^{S^n}|_X)\not=0.$ 
Hence 
\beq
(X\setminus \zeta^{np})\cap S_M\not=\emptyset.
\eneq
In other words,
$$n{\rm Sp}^\eta(h_1,h_2,...,h_n, {\cal F}_1)\setminus \{\zeta^{np}\}\not=\emptyset.$$
This proves (i).

Before we prove part (ii), we also note that
%Before we prove part (ii), we note that, we also have 
\beq\label{Pextsp-25}
\|\psi(f_{M,j})-h_j e_M(d^{1/2})\|<\eta_1+\eta,\,\,\, j=1,2,...,n.
\eneq

To prove (ii), consider $\Phi_{S^n}: \R^n\to S^n\setminus \{\zeta^{np}\}$ 
and its inverse $\Phi_{S^n}^{-1}$ in Definition~\ref{DR2toS2}.

Recall $S_M=\{\zeta\in S^n:  \|\Phi_{S^n}^{-1}(\zeta)\|_2\le M\}$ and
set $P_M=\{\xi\in \R^n: \|\xi\|_2\le M\}.$ So $\Phi_{S^n}(P_M)=S_M.$
%$\Phi_{S^n}(S_M)=P_M.$

For any $0<\eta<1/4,$  there is $0<\eta_2<\eta/64$ such that
every closed ball with center in $\{\xi\in \R^n: \|\xi\|_2\le M\}$ and radius $\eta/16$ contains 
$\Phi_{S^n}^{-1}(B)$ for some closed ball $B$ with center in $S_M$ and radius $\eta_2.$ 

%%%%%%%
\iffalse
{{Let $\eta_3>0$ satisfy the following:
for any $\xi_1, \xi_2\in S_M,$ if $\|\xi_1-\xi_2\|_2<\eta_3,$ then 
\beq
\|\Phi_{S^n}^{-1} (\xi_1)-\Phi_{S^n}^{-1}(\xi_2)\|_2<\eta_2.
\eneq}}
\fi
%%%%%%%%%%%%%%%%%

Let $D_M^\eta$ be a finite $\eta$-dense subset of $P_M$ and 
$\Sigma_M:=\Phi_{S^n}(D_M^\eta)=\{\xi_1, \xi_2,...,\xi_m\}\subset  S_M$ 
%P_M$ 
be 
%a subset 
such that it is $\eta_2/16$ dense in $S_M.$
For each $j\in \{1,2,...,m\},$ define $f_j\in C(S^n)$ with 
$0\le f_j\le 1,$ ${f_j}|_{B_j}=1,$ and $f_j(\xi)=0$ if 
$\|\xi-\xi_j\|_2\ge \eta_2/8,$ where $B_j$ is the ball (of $S^n$) with center at $\xi_j$ and radius $\eta_2/16,$
$j=1,2,...,m.$ 

Choose a finite subset ${\cal F}$ of $C(S^n)$ which contains ${\cal F}_g$ and 
$\{f_j: 1\le j\le m\}.$ 
Choose an $\eta_3/2$-dense subset  $\Sigma$ of $S^m$ such that $\Sigma_M\subset \Sigma.$ 

Choose $\eta_3>0$ as $\dt$ in Corollary \ref{C213} for $\eta/4$ (instead of $\eta$)  and 
$\Sigma.$
%and $M,$
%as well as  $D^\eta$ mentioned above.

Let $\eta_4=\min\{\eta_2/4, \eta_3/4\}.$ 

For this $0<\eta_4,$ by the first part of the proof, we obtain $\dt>0,$ 
when \eqref{Pextsp-1} and \eqref{Pextsp-2} hold for $\dt,$ there is a compact subset 
$X\subset S^n$ and a c.p.c. map
$\psi: C(X)\to A$ such that
\beq
&&\|\psi(r|_X)-(-1+d)(1+d)^{-1}\|<\eta_4,\\
&&\|\psi(x_j|_X)-2(1+d)^{-1/2}h_j(1+d)^{-1/2}\|<\eta_4,\,\,1\le j\le n,\\
&&\|\psi(g_{M,j}|_X)-h_je_M(d^{1/2})\|<\eta_4,\,\,1\le j\le n,\\
%&&\|\psi(f_j|_X)-
&&\|\psi(e_M^{S^n}|_X)-e_M(d^{1/2})\|<\eta_4\\\label{Pextsp-35}
&&\|\psi(fg|_X)-\psi(f|_X)\psi(g|_X)\|<\eta_4\rforal f,g\in {\cal F}_1\andeqn\\\label{Pextsp-36}
&&\|\psi(g)\|\ge 1-\eta_4
\eneq
for any $g\in C(X)_+$  which has 
value $1$ on an open ball with the center $x\in X$ (for some $x$) and the radius $\eta_4.$

\iffalse
Let $I\subset C_0(\R^n)$ be defined by
\beq
I=\{f\in C_0(\R^n):  f|_{\R^n\setminus (M+1)^n}=0\}.
\eneq
%On $M^n=\{\xi=(t_1,t_2,...,t_n)\in \R^n: |t_j|\le M,\,\, 1\le j\le n\},$ d
Define
$\tilde e_j: M^n\to \R$ by $e_j((t_1, t_2,...,t_n))=t_j,$ $1\le j\le n.$
\fi
%Define $L: C_0(\R^n)\to A$ by 
%$L(f)=\psi(f\circ \Phi_{S^n}),$ $j=1,2,...,n.$ 
Let $L: C(S^n)\rightarrow A$ be the composition $\psi\circ q,$ where $q: C(S^n)\rightarrow C(X)$ is the quotient map defined by restriction to $X$.
By applying  Corollary \ref{C213} and \eqref{Pextsp-35}, 
we obtain that 
\beq\label{Pextsp-40}
\|L(\Theta^S_{\xi_j, \eta/4})-\Theta^S_{\xi_j, \eta/4}(h_1,h_2,...,h_n)\|<\eta/4
\eneq
for any $1\le j\le m.$   
Therefore, by the choice of $\eta_2,$ $\Theta_{\xi_j, \eta/4}\ge f_j$ 
for some $j\in \{1,2,...,m\},$ where $f_j^{-1}(1)$ contains an open ball of radius $\eta_4$ in $X.$
By \eqref{Pextsp-36},
\beq
\|L(\Theta^S_{\xi_j, \eta/4})\|=\|\psi(\Theta^S_{\xi_j, \eta/4}|_X)\|\ge \|\psi(f_j)\|\ge 1-\eta_4.
\eneq
Hence, by \eqref{Pextsp-40},
\beq
\|\Theta_{\xi_j, \eta/4}^S(h_1,h_2,...,h_n)\|\ge 1-\eta_4-\eta/4>1-\eta.
\eneq
Hence, $\|\Theta^S_{\xi_j, \eta}(h_1,h_2,...,h_n)\|\ge 1-\eta$
for some $1\le j\le m.$
%$\xi\in \Phi^{S^n-1}(X)\cap P_M\not=\emptyset.$
This implies  $s{\rm Sp}_{S^n, M}^\eta((h_1, h_2,...,h_n))\not=\emptyset, $ whence
\beq
S{\rm Sp}_M^{\eta}((h_1, h_2,...,h_n))\not=\emptyset.
\eneq
\end{proof}

\begin{cor}\label{Pextsp}
Let $\eta>0,$ 
$n\in \N$ 
and $M>1.$ 
Let ${\cal F}_g$ be a finite subset of $C(S^n)$ which contains $r,$ $x_1,x_2,..., x_n$
as in Remark~\ref{RemSn}.
There exists $\dt>0$ satisfying the following:
For any $n$-tuple self-adjoint operators $h_1,h_2,...,h_n$ densely defined on an infinite dimensional 
separable Hilbert space $H$ such that
\beq\label{Pextsp-1'}
\|h_ih_j-h_jh_i\|<\dt,\,\,\, i,j=1,2,...,n\tand {\rm sp}(d^{1/2})\cap [0, M]\not=\emptyset,
\eneq
(where $d=\sum_{i=1}^n h_i^2$), then 

{{\rm(i)}} $S{\rm Sp}_M^\eta((h_1, h_2,...,h_n))\not=\emptyset,$

{{\rm(ii)}} $n{\rm Sp}^\eta_{S^n,M}(h_1, h_2,...,h_n; {\cal F}_g)\cap S_M
%\setminus \{\zeta_{np}\}
\not=\emptyset.$

\end{cor}

\begin{proof}
Fix $\eta>0,$ $n\in \N$ and $M>1.$
Choose $\dt_0:=\dt(\eta, n, M)>0$ as given by Lemma \ref{PextspMs}.
Choose $\dt=\dt_0/12n.$  Suppose that \eqref{Pextsp-1'} holds.
By Lemma \ref{Lappcom-1n}, we have
\beq
&&\|[\tilde h_i,\,\tilde h_j]\|<\dt_0,\,\,\,\|\tilde h_i-\tilde h_i^*\|<\dt_0,\,\,\, i,j=1,2,...,n.
\eneq
We also have  ${\rm sp}(d^{1/2})\cap [0, M]\not=\emptyset.$
%$e_M(d^{1/2})\not=0.$
The corollary then follows by applying Lemma \ref{PextspMs}.
\end{proof}

In the following statement, define
\beq\label{DfbM}
S_M:=\Phi_{S^m}(\{\zeta\in \R^m: \|\zeta\|_2\le M\}).
\eneq
%\beq\label{DfbM}
%S_M:=\Phi^{S^m}(\{\zeta\in \R^m: \|\zeta\|_2\le M\}).
%\eneq

\begin{prop}{\rm (\cite[Proposition 2.15]{Linself})}\label{Pspsub}
Fix $m\in \N.$ For any $\eta>0,$ $M>1$ such that $S_M=\Phi_{S^{m}}(\{\xi\in \R^{m}: \|\xi\|_2\le M\})$ is $\eta/4$-dense in $S^m,$  and 
a finite generating subset ${\cal F}_g\subset C(S^m)$
which contains $r,$ $x_1,x_2,..., x_m,$ there exists $\dt(\eta, M)>0$ satisfying the following:

Suppose that  $h_1, h_2,...,h_m$ form an $m$-tuple of (densely defined) self-adjoint operators 
on $H$ 
%$A$ is a unital \CA\, and $a_1, a_2,...,a_k\in A_{s.a.}$ with $\|a_i\|\le 1$ ($1\le i\le k$)
such that $(h_1, h_2,...,h_m)$ has 
an $M$-$\dt$-near-spherical spectrum $X=n{\rm Sp}_{S^m, M}^\dt((h_1, h_2,...,h_m; {\cal F}_g))$ 
such that $X\setminus \{\zeta^{np}\}\not=\emptyset.$
%\beq
%\|a_ia_j-a_ja_i\|<\dt,\,\,\, 1\le i,j\le k.
%\eneq
%Then $(a_1,a_2,...,a_k)$ has an $\eta$-pseudo-spectrum $X$ Moreover,
Then 

{\rm 1)}. $Z=S{\rm Sp}^\eta_{M}((h_1,h_2,...,h_m))\not=\emptyset.$ 

{\rm 2)}.
If 
$Y$ is also a non-empty  $M$-$\dt$-near-spherical spectrum of $(h_1, h_2,...,h_m),$
then 
\beq
d_H(X, Y)<\eta\andeqn  X, Y \subset \Phi_{S^m}(Z) \subset X_{3\eta}.
\eneq
\end{prop}

\begin{proof}
Let $\eta$ and $M$ be as given.
Choose $0<\eta_0<\eta/2$ such that 
$\Sigma_M:=\Phi_{S^m}(D^{\eta_0}_{M})=
%Put $B^K=\Phi^{S^m}(\{\zeta\in \R^n: \|\zeta\|_2\le K\}$ for $K\ge 1.$ 
\{\xi_1, \xi_2,..., \xi_N\}$  be an $\eta/16$-dense subset of $S_M,$
where $D^{\eta_0}_{M}\subset \{\zeta\in \R^n: \|\zeta\|_2\le M\}.$
\Wlog, we may choose $\xi_j$  
%$\|\xi_j\|_2<M.$}} 
such that $\min_{1\le j\le N}{\rm dist}(\xi_j, S_M^c)>0.$
Note, since $S_M$ is $\eta/4$-dense in $S^m,$ $\{\xi_1,\xi_2,...,\xi_{ N}\}$ is also 
$5\eta/16$-dense in $S^m.$ Here $S_M^c$ is the complement of $S_M$ in $S^m$.

Put $\dt_0=\min\{\eta/16, \min_{1\le j\le N}{\rm dist}(\xi_j, S_M^c)\}/2\}>0.$
Let ${\cal G}=\{g_{M,j}:1\le j\le m\}$ (see Definition~\ref{dfMR}) and $C$ be \SCA\, generated by 
${\cal G}.$ 
%$S^m$ and $x_j\not=\zeta^{np},$$1\le j\le n.$
Choose $f_j\in C(S^m)_+$ with $0\le f_j\le 1,$
$f_j(x)=1$ if ${\rm dist}(x, {\xi_j})\le \dt_0,$ and $f_j(x)=0,$ if 
${\rm dist}(x, {\xi_j})\ge {2\dt_0},$ $j=1,2,...,{N}.$
Put ${\cal F}=\{f_j: 1\le j\le {N}\}.$ 
Since $f(\xi)=0$ for all $\xi\not\in S_{M}$ for all $f\in {\cal F},$
${\cal F}\subset C.$ 
%
%Fix ${\cal F}_g$ as in Definition \ref{Dspherical}.  
We also fix $D^\eta_{M}$ as in Definition~\ref{Dsynsp2}.

Let  $\dt_1:=\dt({\cal F}, {\dt_0})>0$ be  given  by Lemma \ref{Lpurb}
%{C2.12
% of \cite{Linself}
%\ref{Lpurb} 
(for  $X=S^m,$ ${\cal G}$ above,
$\eta/32,$ and ${\cal F}$). 
Let 
$\dt_2:=\dt(m,\eta/2)>0$  be given by  Corollary \ref{C213} for  $\Sigma:=\Sigma_M.$
%$M$ and $D^\eta$ mentioned above.
%Corollary \ref{Ctheta}. 
Choose $\dt=\min\{\dt_1/2, \dt_2/2, \eta/33, \dt_0\}.$

Now suppose that $X$  and $Y$ are non-empty $M$-$\dt$-near-spherical spectra of the ${m}$-tuple
$(h_1,h_2,...,h_{m}).$  Denote $X_1=X$ and $X_2=Y.$ 
Let $L_i: C(X_i)\to B(H), i=1,2$  be  unital c.p.c maps such that
\beq
%\hspace{-2.1in}{\rm (i)} && \{e_j|_X: 1\le j\le n\}\,\,\, {\rm generates}\,\,\,  C(X);\\
%
\hspace{-2.1in}{\rm (i_1)} &&\|L_i(r|_{X_i})-(-1+d)(1+d)^{-1}\|<\dt,\,\,\, \\
\hspace{-2.1in}{\rm (i_2)}&& \|L_i(x_j|_{X_i})-2(1+d)^{-1/2}h_j(1+d)^{-1/2}\|<\dt,\\\label{XY-7}
\hspace{-2.1in}{\rm (i_3)}&&\|L_i(g_{M,j}|_{X_i})-h_je_M(d^{1/2})\|<\dt,\\
\hspace{-2.1in}{\rm (ii)}&&\|L_i((fg)|_{X_i})-L_i(f|_{X_i})L_i(g|_{X_i})\|<\dt \rforal f, g\in {\cal F}_g
\\
%\hspace{-2.1in}{\rm (ii_1)}&&\|\phi(z)^*\phi(z)-\phi(z)\phi(z)^*\|<\eta,\,\,
\andeqn\\\label{XY-8}
\hspace{-2.1in}{\rm (iii)} &&  \|L_i(f)\|\ge 1-\dt
\eneq
for any $f\in C(X_i)_+$  which has 
value $1$ on an open ball with the center $x\in X_i$ (for some $x$) and the radius $\dt.$

Let $\xi\in X=X_1.$  We may assume that $\xi\in {D}(x_j, \eta/8).$
Then ${ D}(x_j, \dt)\subset {D}(\xi, \eta/4).$
By \eqref{XY-8},
\beq\label{Pspuniq-10}
\|L_1(f_j|_X)\|\ge 1-\dt.
\eneq
Define  $L'_l: C(S^m)\to A$ ($l=1,2$) by 
$L'_{l}(f)=L_1(f|_{X_{l}})$
% and $L_2'(f)=L_2(f|_Y)$ 
for all $f\in C(S^m).$
By \eqref{XY-7} and applying   Lemma \ref{Lpurb}, we have, for $1\le j\le m,$ 
\beq
\|L_1'(f_j)-L_2'(f_j)\|<\eta/32,\,\,\, {\rm or}\,\,\,
\|L_1(f_j|_X)-L_2(f_j|_Y)\|<\eta/32.
\eneq
Then,  by \eqref{Pspuniq-10},  we obtain, for $1\le j\le m,$ 
\beq
\|L_2(f_j|_Y)\|\ge \|L_1(f_j|_X)\|-\eta/32\ge 1-\dt-\eta/32>3/4.
\eneq
Since $f_j(x)=0,$ if ${\rm dist}(x, x_j)\ge 2\dt_0,$  we must have  ${\rm dist}(x_j, Y)<2\dt_0.$
Therefore 
\beq
{\rm dist}(\xi, Y)\le {\rm dist}(\xi, x_j)+{\rm dist}(x_j, Y)<\eta/8+2\dt_0<\eta/4.
%5\eta/8.
\eneq
This  holds for all $\xi\in X.$ 
Exchanging the role of $X$ and $Y,$ we obtain
\beq
d_H(X, Y)<\eta.
\eneq
This proves the first part of 2).

%Fix  an $\eta/2$-dense subset $D^\eta$ of $P_M=\{\zeta\in \R^m: \|\zeta\|_2\le M\}.$ 
%Now write  $\Sigma^\eta=\Phi_{S^m}(D^\eta)=\{\xi_1,\xi_2,...,\xi_N\}.$ 
%and 
% $Z=\bigcup_{\|\Theta_{\xi_j, \eta}(h_1,h_2,...,h_m)\|\ge 1-\eta} \overline{B(\xi_j, \eta)}.$ 

By the choice of  $\dt$ and applying  Corollary \ref{C213},
%\ref{Ctheta}, 
we have that (for $\xi_j\in \Sigma_M$),
\beq\label{XY-15-}
\|L_1(\Theta^S_{\xi_j,\eta}|_X)-\Theta^S_{\xi_j,\eta}(h_1,h_2,...,h_n)\|<\eta/2,\,\,\, j=1,2,...,N.
\eneq
Pick $\xi\in X.$ There is $\xi_j\in \Sigma_M$ such that
${\rm dist}(\xi, \xi_j)<\eta/2.$  Recall that $\Theta^{S}_{\xi_j,\eta}(y)=1$
if ${\rm dist}(y, \xi_j)\le 3\eta/4.$  As $\dt<\eta/32,$  we have $\eta/2+\dt<3\eta/4.$ Hence 
 $\Theta^{ S}_{\xi_j, \eta}(y)=1$ for all $y\in {D}(\xi, \dt)\subset {D}(\xi_j, 3\eta/4).$ 
It follows from  (iii) above that 
\beq
\|L_1(\Theta^S_{\xi_j,\eta}|_X)\|\ge 1-\dt.
\eneq
Thus,  by \eqref{XY-15-},
\beq
\|\Theta^S_{\xi_j,\eta}(h_1,h_2,...,h_n)\|\ge 1-\dt-\eta/2\ge 1-\eta.
\eneq
This implies that 
\beq
X\subset \bigcup_{\|\Theta^{S}_{\xi_j,\eta}(h_1,h_2,...,{{h_m}})\|\ge 1-\eta}\overline{{D}(\xi_j, \eta)}=s{\rm Sp}^\eta_{{S^m, M}}((h_1,h_2,...,h_m)).
\eneq
Since $X\not=\emptyset,$ we conclude that $Z_S:=s{\rm Sp}_{S^m, M}^\eta((h_1,h_2,...,h_{m}))\not=\emptyset.$  Consequently $Z=\Phi_{S^m}^{-1}(Z_S)\not=\emptyset.$ 
This proves 1). 
The same argument shows that $Y\subset  Z_S.$
%s{\rm Sp}_{S^m, M}^\eta((a_1,a_2,...,a_k)).$ 

Next, let $\|\Theta^{S}_{\xi_j,\eta}({h_1,h_2,...,h_m})\|\ge 1-\eta$ {{for $\xi_j\in\Sigma_M$.}} 
By \eqref{XY-15-},
\beq
\|L_1(\Theta^{S}_{\xi_j, \eta}|_X)\|\ge 1-\eta-\eta/2>1/4.
\eneq
Hence, $\Theta^{S}_{\xi_j, \eta}|_X\not=0.$ 
Thus 
\beq
X\cap {D}(\xi_j, \eta)\not=\emptyset. 
\eneq
It follows that
\beq
{D}(\xi_j, \eta)\subset X_{2\eta}. 
\eneq
Hence 
\beq
Z_S\subset X_{2\eta}.
\eneq

\end{proof}

\begin{df}\label{Dpsi}
Let $\{h_{1,k},h_{2,k},...,h_{n,k}\}_{k\in \N}$ be a sequence of $n$-tuple densely defined 
selfadjoint operators in an infinite dimensional Hilbert space $H$.  Suppose 
that
\beq
\lim_{k\to\infty} \|[h_{i,k}, h_{j,k}]\|=0,\,\,\, 1\le i,j\le n.
\eneq
Let us write $l^\infty (B(H))=\prod_{k=1}^\infty B(H)$ and 
$c_0(B(H))=\bigoplus_{k=1}^\infty B(H).$ 
Let $\Pi: l^\infty(B(H))\to l^\infty(B(H))/c_0(B(H))$ be the quotient map.
Define $\psi: C_0(S^n\setminus \zeta^{np})\to l^\infty(B(H))/c_0(B(H))$
by
\beq
\psi(r)&=&\Pi(\{(\sum_{j=1}^n h_{j,k}^2-1)(\sum_{j=1}^n h_{j,k}^2+1)^{-1}\}),\\
\psi(x_j)&=&\Pi(\{{\bar h}_{j,k}\}),\,\,\, 1\le j\le n,
\eneq
where ${\bar h}_{j,k}=(1+\sum_{i=1}^n h_{i,k}^2)^{-1/2}h_{j,k}(1+\sum_{i=1}^nh_{{i},k}^2)^{-1/2},$ $1\le j\le n.$
In the previous  two sections we 
repeatedly used this \hm\, $\psi.$ In what follows we will call
the \hm\, 
\[
\psi: C_0(S^n\setminus \zeta^{np})\to l^\infty(B(H))/c_0(B(H))
\]
the \hm \,
generated by the sequences
\[
(\{(\sum_{j=1}^n h_{j,k}^2-1)(\sum_{j=1}^n h_{j,k}^2+1)^{-1}\}, \{\bar{h}_{1, k}\}, \{\bar{h}_{2, k}\},..., \{\bar{h}_{n, k}\}).
\]
\end{df}

\section{AMU States}

\begin{lem}\label{Ldecomp}
Let $\Omega$ be a compact metric space such that $C(\Omega)$
has a finite generating subset ${\cal F}_g\subset C(\Omega)^{\bf 1}.$ 
Let $\ep>0.$ 
%and 
%${\cal F}\subset C(\Omega)$ be a finite subset.
Then there exists  $\dt>0$ 
%and a finite subset ${\cal G}\subset C(\Omega)$ 
satisfying the following:

For any c.p.c. map $L: C(\Omega)\to A,$ where $A$ is a unital \CA\, of real rank zero, such that
\beq\label{Ldecomp-1}
\|L(fg)-L(f)L(g)\|<\dt\tforal f, g\in {\cal F}_g,
\eneq 
% for any finite (distinct) points $\xi_1, \xi_2,...,\xi_m\in X,$
there exist  $\xi_1, \xi_2,..., \xi_m\in X:=nSp^{\ep/4}(L)$ which are $\ep$-dense in 
$X,$
 and 
mutually orthogonal non-zero projections 
$p_1,p_2,..., p_m\in A$ such that, for all $f\in {\cal F}_g,$
\beq
\|\sum_{k=1}^m f(\xi_k) p_k+(1-p)L(f)(1-p)-L(f)\|<\ep,
%,\,\,\, 1\le j\le n.
\eneq 
where $p=\sum_{i=1}^m p_i.$

\end{lem}

\begin{proof}
Let $\dt_1>0$ be $\dt$ and ${\cal F}_g$ be ${\cal G}$ in \cite[Lemma 3.3]{Lincmp2025}
for $\ep/4$ (instead of $\ep$), respectively. We may assume 
that $0<\dt_1<\ep/4.$ 

Choose $\dt$ in Lemma \ref{Lnsp} associated with $\dt_1$ (instead of $\ep$)
and ${\cal F}_g$ (see the last line of Lemma~\ref{Lnsp}). 

Now suppose that $L: C(\Omega)\to A$  is a c.p.c map satisfying \eqref{Ldecomp-1}.
By applying Lemma \ref{Lnsp}, we obtain  a compact subset $X\subset \Omega$ and 
a c.p.c. map
$L_1:  C(X)\to A$ such that
\beq\label{Ldecomp-3}
&&\|L_1(f|_X)-L(f)\|<\dt_1\rforal f\in {\cal F}_g,\\\label{Ldecomp-4}
&&\|L_1(fg|_X)-L_1(f|_X)L_1(g|_X)\|<\dt_1\rforal f,g\in {\cal F}_g\andeqn\\\label{Ldecomp-5}
&&\|L_1(f)\|\ge 1-\dt_1
\eneq
for any $f\in C(X)_+$  which has 
value $1$ on an open ball with the center $x\in X$ (for some $x$) and the radius $\eta.$

We now apply \cite[Lemma 3.3]{Lincmp2025}.  By the choice of $\dt_1$ and  by applying 
\cite[Lemma 3.3]{Lincmp2025} to $L_1$ (instead of $L$), we obtain
$\xi_1, \xi_2,...,\xi_m\in X$ which is $\ep/2$ dense in $X,$  and
mutually orthogonal projections $p_1,p_2,...,p_m\in A$ such that
\beq\label{Ldecomp-7}
\|\sum_{j=1}^m f(\xi_j)p_j+(1-p)L_1(f)(1-p)-L_1(f)\|<\ep/4\rforal f\in {\cal F}_g,
\eneq
where $p=\sum_{j=1}^m p_j.$
We now estimate that, by \eqref{Ldecomp-3} and  \eqref{Ldecomp-7}, for all $f\in {\cal F}_g,$ 
\beq
L(f)&\approx_{\dt_1}& L_1(f)\approx_{\ep/4} 
\sum_{j=1}^m f(\xi_j)p_j+(1-p)L_1(f)(1-p)\\
&\approx_{\dt_1} & \sum_{j=1}^m f(\xi_j)p_j+(1-p)L(f)(1-p).
\eneq

The lemma then follows from the fact that $\dt_1+\ep/4+\dt_1<\ep.$
\end{proof}

\begin{thm}\label{TAMU1}
Let $m\in\N, \ep>0$ and $M>1.$  There exist $0<\eta<\ep/4$ (depending on 
$\ep$ and $M$), $\dt(m, \eta, M, \ep)>0$ satisfying the following:

Let $\{h_1,h_2,...,h_m\}$ be an $m$-tuple of (unbounded) self-adjoint operators densely defined 
on an infinite dimensional separable Hilbert space $H$ such that
\beq\label{TAMU1-1}
&&\|[\tilde h_i,\, \tilde h_j]\|<\dt, \,\, \|\tilde h_i-\tilde h_i^*\|<\dt\andeqn\\
%\|h_ih_j-h_ih_j\|<\dt,\,\,\, i,j=1,2,...,m.
&&{\rm sp}(d^{1/2})\cap [0, M]\not=\emptyset,
\eneq
where $d=\sum_{j=1}^m h_j^2$ and $\tilde h_j=h_j(1+d)^{-1}.$
Then

{\rm (i)} $S{\rm Sp}_M^\eta((h_1, h_2,...,h_m))\not=\emptyset;$

{\rm (ii)} for any $\lambda=(\lambda_1, \lambda_2,...,\lambda_m)\in S{\rm Sp}_M^{\eta}((h_1, h_2,...,h_m)),$
there is a  unit vector $v\in H$ such that
\beq
%\max_{1\le i\le 2}|\la h_i(v), v\ra -\lambda_i|<\ep\tand
\max_{1\le i\le m}\|(h_i-\lambda_i)(v)\|<\ep.
\eneq

\end{thm}

\begin{proof}
Fix $M\ge 1.$ 
%such that $S_M=\{\Phi_{S^m}(\{\zeta\in \R^m: \|\zeta\|_2\le M\}$ 
%is $\eta/16$-dense in $
Let $\Omega=S^m$ and ${\cal F}_g\subset C(S^m)^{\bf 1}$ be a finite subset which contains a 
generating set, as well as $e_{ M}^{S^2}$  and $g_{M, j}$ ($1\le j\le m$). 
Let $1>\ep>0.$    We may assume that $\ep<1/2{ M}.$ 
%{\red{$\dt_0?$}}
%%%%%%%%%%%%%%
%\iffalse
%%%%%%%%%%%%%%%%%%%%%%%%%%%%%%%%%
Let $F_{2M}$ be as defined in Definition \ref{Mdelta}.
Since $F_{2M}$ is continuous on $[0, 2]$ and $F_{2M}(0)=0,$
we may choose $\ep_1>0$ such that
\beq
F_M(r)<\ep/16\rforal 0<r<\ep_1.
\eneq
Put ${{\ep_2=\min\{\ep/2^8, \ep_1\}/2(M+{ 1})>0}}.$
%\fi
%%%%%%%%%%%%%%%%%%%%%%%%%%
%
%%%%%%%%%%%%%%%%%%%%%%%%%%%%

Choose $\dt_1>0$ (in place of $\dt$)  associated with $\ep_2/4$ (in place of $\ep$), ${\cal F}_g$
(and with $\Omega=S^m$)
in Lemma \ref{Ldecomp}.  We may assume that $\dt_1<\ep_2/8.$ 
Let $\dt_1'>0$ (as $\dt$)  in Proposition~\ref{PextspMs}
%{Pextsp} 
associated with $\eta:=\ep_2/16$ and ${ M}.$
We may also assume that 
$$
\dt_1<\min\{\dt_1', \dt(\ep_2/16, { M})\},
$$ where $\delta(\ep_2/16, M)$ is defined (with $\eta=\ep_2/16,$ %and with  $M+1$ in place of $M$
) as in Proposition~\ref{Pspsub}.

Choose $\dt_2>0$ (in place of $\dt$)  associated with $\dt_1/2$ (in place of $\ep=\eta$)
and ${\cal F}_g$  as well as $M$ in Lemma \ref{Cs2app2}. 

Let $\eta=\ep_2/16$ and $\dt=\min\{\dt_2, \dt_1/2\}>0.$ 

Suppose now that $h_1, h_2, ...,h_m$ form an $m$-tuple  of self-adjoint operators densely defined on $H$
which satisfy \eqref{TAMU1-1}.   Recall $d=\sum_{j=1}^m h_j^2.$
%Set $z=h_1+i h_2.$ 

  It follows from Lemma \ref{Cs2app2} that
there exists a c.p.c. map
$L: C(S^m)\to B(H)$ such that
\beq\label{TAMU-10}
&&\|L(g_{_{{ M}, j}})-e_{ M}(d^{1/2})h_j\|<\dt_1/2,\\\label{TAMU-11}
%&&\|L(g_{M+1,im})-e_{M+1}(z^*z^{1/2})h_2\|<\dt_1/2,\\\label{TAMU-12}
&&\|L(e_{ M}^{S^m})-e_{ M}(d^{1/2})\|<\dt_1/2\andeqn\\\label{TAMU-13}
&&\|e_{ M}(d^{1/2})h_j-h_je_{ M}(d^{1/2})\|<\dt_1/2.\label{TAMU-14}
\eneq
Moreover, we may require that 
\beq
\|L(f)L(g)-L(fg)\|<\dt_1/2\tforal f,\, g\in {\cal F}_g.
\eneq
%Put $z=h_1+ih_2.$ 
By the choice of $\dt_1,$ applying Lemma \ref{Ldecomp}, we obtain 
$\xi_1, \xi_2,..., \xi_N\in X:=nSp^{\ep_2/16}_{{S^m, { M}}}(L)$ which are $\ep_2$-dense in 
$X,$
 and 
mutually orthogonal non-zero projections 
$p_1,p_2,..., p_N\in {{B(H)}}$ such that, for all $f\in {\cal F}_g,$
\beq
\|\sum_{k=1}^N f(\xi_k) p_k+(1-p)L(f)(1-p)-L(f)\|<\ep_2/4,
%,\,\,\, 1\le j\le n.
\eneq 
where $p=\sum_{i=1}^N p_i.$  It follows  (also from \eqref{TAMU-10} -- \eqref{TAMU-14})  that
\beq\label{TAMU-20}
&&\hspace{-1.1in}
\|\sum_{k=1}^m g_{_{{M},j}}(\xi_k)p_k-(1-p)L(g_{_{{ M},j}})(1-p)-h_je_{ M}(d^{1/2})\|<\ep_2/4+2\ep_2/16=3\ep_2/8,\\
%&&\hspace{3.4in}=3\ep_2/8,\\
%&&\hspace{-0.5in}\|\sum_{k=1}^m g_{M+1,im}(\xi_k)p_k-(1-p)L(g_{M+1,im})(1-p)-h_1e_{M+1}(z^*z^{1/2})\|<3\ep_2/8,
\andeqn\\\label{TAMU-22}
&&\hspace{-1.1in}\|\sum_{k=1}^m e_{ M}^{S^{m}}(\xi_k)p_k-(1-p)L(e_{ M}^{S^{m}})(1-p)-e_{ M}({d}^{1/2})\|<\ep_2/4+\ep_2/16.
\eneq

By the choice of $\dt_1,$ applying Lemma \ref{Pspsub}, we may assume that
\beq
X\subset  {s}{\rm Sp}^{\ep_2/16}_{S^m,M}((h_1,h_2,...,h_m))
%\Phi^{S^m}({\rm s}Sp^{\ep_2/16}_{(M+1)_{\ep_2/16}}(h_1, h_2,...,h_m))_{\ep_2/16}
\subset X_{3\ep_2/16}.
\eneq
This implies that $S{\rm Sp}_M^\eta((h_1, h_2,...,h_m))\not=\emptyset.$
%%%%%%%%%%%%%%%%%%
%
%%%%%%%%%%%%%
\iffalse
Note that 
\beq
&&M+1\le (M+1)_{\ep_2/16}\andeqn\\
&& {\rm s}Sp_M^{\eta}(h_1,h_2,...,h_m)
\subset ({\rm s}Sp^{\ep_2/16}_{(M+1)_{\ep_2/16}}(h_1, h_2...,h_m))_{\ep_2/16}.
\eneq
\fi
%%%%%%%%%%%%%%%%

Let $\lambda=(\lambda_1,\lambda_2,...,\lambda_m)\in S{\rm Sp}_M^{\eta}((h_1,h_2,...,h_m)).$
Since $\{\xi_1, \xi_2,...,\xi_N\}$ is $\ep_2$-dense in $X,$ by the choice of $\ep_2,$
there is $k\in \{1,2,...,N\}$ such that
\beq
|\lambda-{\Phi_{S^m}}^{-1}(\xi_k)|<\ep/8.
\eneq
Let $\zeta_k={\Phi_{S^m}}^{-1}(\xi_k)=(t_{k,1},t_{k,2},...,t_{k,m})$  Then 
$g_{M+1, j}(\xi_k)=t_{k,j}.$
By \eqref{TAMU-20} and \eqref{TAMU-22}, we have that
\beq\label{TAMU-26}
&&\|t_{k,j} p_k-h_je_{M}(d^{1/2})p_k\|<3\ep_2/8\andeqn\\\label{TAMU-27}
&&\|p_k-e_{M}(d^{1/2})p_k\|<5\ep_2/16
\eneq
Since $p_k\not=0,$ one may choose  $v_1\in p_k(H)$ with $\|v_1\|=1$
and $v_2=e_{M}(d^{1/2})p_kv_1.$ 
Then 
\beq
\|v_1-v_2\|<5\ep_2/16\andeqn \|v_2\|\ge 1-5\ep_2/16.
\eneq
Choose   $v:=v_2/\|v_2\|.$
Then, by \eqref{TAMU-26} and \eqref{TAMU-27},
\beq
h_jv_2\approx_{3\ep_2/8} t_{k,j}v_1\approx_{5\min\{\ep/2^{9},\ep_1\}} t_{k,j} v_2\approx_{\ep/ 8}  \lambda_k v_2.
\eneq
%{\blue{How to obtain $5\min\{\ep/2^{10},\ep_1\}$? Should be $5M\epsilon_2/16$?}}

Hence
\beq
\|h_jv_2-\lambda_j v_2\|<\ep/4.
\eneq
It follows that, for $j=1,2,...,m,$
\beq
\|h_jv-\lambda_j v\|={\ep\over{4\|v_2\|}}<{\ep\over{4(1-5\ep_2/16)}}<\ep.
\eneq
%Exactly the same argument shows that
%\beq
%\|h_2v-\lambda_2 v\|<\ep.
%\eneq

%\beq
%\|v-v_1\|\le \|v-v_2\|+\|v_2-v_1\|<10\ep_2/16.
%\eneq
%%%%%%%
\iffalse
Note that
\beq
h_1e_{M+2}(z^*z^{1/2}) e_{M+1}(z^*z^{1/2})=h_1e_{M+1}(z^*z^{1/2})
\eneq
and, by ?, 
\beq
\|h_1e_{M+2}(z^*z^{1/2})\|\le M+2.
\eneq
Hence
\beq
\|h_1p_k-h_1e_{M+1}(z^*z^{1/2})p_k\|
\eneq
{\red{To be continued}}
\fi
%%%%%%%%%%%
\end{proof}

\begin{thm}\label{TAMU''}
Let $m\in\N, \ep>0$ and $M>1.$  There exists $0<\eta<\ep/4$ (depending on 
$\ep$ and $M$) and $\dt(m, \eta, M, \ep)>0$ satisfying the following:

Let $\{h_1,h_2,...,h_m \}$ be an $m$-tuple of (unbounded) self-adjoint operators densely defined 
on an infinite dimensional separable Hilbert space $H$ such that
\beq\label{TAMU-1.3}
\|h_ih_j-h_{j}h_{i}\|<\dt,\,\,\, i,j=1,2,...,m \quad and \,\,\, {\rm sp}( d^{1/2})\cap [0, M]\not=\emptyset,
\eneq
where $d=\sum_{j=1}^m h_j^2.$
Then, for any $\lambda=(\lambda_1, \lambda_2,...,\lambda_m)\in S{\rm Sp}_M^{\eta}((h_1, h_2,...,h_m)),$
there is a unit vector $v\in H$ such that
\beq
%\max_{1\le i\le 2}|\la h_i(v), v\ra -\lambda_i|<\ep\tand
\max_{1\le i\le m}\|(h_i-\lambda_i)(v)\|<\ep.
\eneq

\end{thm}

\begin{proof}
Fix $m\in \N.$ 
Let $\ep>0$ and $M>1$ be given.  
%Choose $\ep_1=\ep/6m.$ 
%Let us remind the reader that ${\rm sp}(d^{1/2})\cap [0, M]=\emptyset$ implies 
%that $e_M(d^{1/2})\not=0.$ 
Let $\eta>0$ be given by Theorem \ref{TAMU1} for $\ep$ and $M.$
Then choose $\dt:=\dt(\ep, M, \eta)/6m>0$ that is given by Theorem \ref{TAMU1}.
 Now suppose that $h_1, h_2,...,h_m$ are as given such that
 $\|[h_i, \, h_j]\|<\dt$ ($1\le i,j\le m$) and ${\rm sp}( d^{1/2})\cap [0, M]\not=\emptyset.$
 %$e_M(d^{1/2})\not=0.$ 
 %{\blue{(Note that we do not have $e_M(d^{1/2})\not=0$ in the statement of Thm 5.2. Change back?)}}
 
 It follows from Lemma \ref{Lappcom-1n} that 
 \beq
 &&\|[\tilde h_i, \tilde h_j]\|<6m\dt=\dt(\eta,  M, \ep)\andeqn \\
 &&\|\tilde h_i^*-\tilde h_i\|<\dt(\eta, M, \ep),
 \eneq 
 $1\le i, j\le m.$ 
It follows from Theorem \ref{TAMU1}  that, for any $\lambda=(\lambda_1, \lambda_2,...,\lambda_m)\in S{\rm Sp}_M^{\eta}((h_1, h_2,...,h_m)),$ there exists a unit vector $v\in H$ such that
\beq
\max_{1\le i\le m}\|(h_i-\lambda_i)v\|<\ep.
\eneq

\end{proof}

\section{Synthetic Spectrum revisited}

The purpose of this section is to present the proof of Theorem \ref{TAMU}, which follows from 
the more elaborate Theorem \ref{thm:6main}. 
In the statement of Theorem \ref{thm:6main}, we use the more straightforward version of 
synthetic spectrum,  \( s\mathrm{Sp}_M^{\eta}((h_1, \ldots, h_n)) \subset \mathbb{R}^n \)
(see Definition \ref{Dsynsp2}). 
This set looks like \( S\mathrm{Sp}_M^{\sigma}((h_1, \ldots, h_n)) \). 
Theorem \ref{thm:6main} should then easily follow from 
Theorem \ref{TAMU''}.
However,  
we will need to establish that the synthetic spectrum 
\( s\mathrm{Sp}_M^{\eta}((h_1, \ldots, h_n)) \) 
is nonempty whenever 
% if and only if 
the spherical synthetic spectrum 
\( S\mathrm{Sp}_M^{\sigma}((h_1, \ldots, h_n)) \) is nonempty, and, more importantly, they are close to each other.
This allows us to reformulate Theorem~\ref{TAMU''} 
as Theorem~\ref{thm:6main}, where a more practical version of the synthetic spectrum is employed (see (2) of Remark \ref{Rems6}).

Unfortunately, some technical difficulties arise here due to the unboundedness of the operators. It is worth pointing out again that
% a key challenge: 
if $h_1$ and $h_2$ are unbounded operators such that $[h_1, h_2]$ is a bounded operator with very small norm, this does not imply that $[h_1, h_2^2]$  itself has a small norm. In fact, $[h_1, h_2^2]$  may not only have a large norm but could even be unbounded (see \eqref{S1Sh2}).

%(see the end of Section 7.1). 
% This makes 
This complication disrupts standard functional calculus and necessitates lengthy and careful calculations throughout this and other sections. For example, we need Lemma~\ref{Lfrom3-2} for our computation.

\iffalse
Some technical difficulties reappear here when dealing with unbounded operators. 
It might worth to point out the following fact. Suppose that $h_1$ and $h_2$ are two unbounded 
operators such that $\|[h_1, h_2]\|$ is a bounded operator with very small norm. 
Unlike bounded case, this does not imply that $\|[h_1, h_2^2]\|$ is small. In fact, not only
$[h_1, h_2^2]$ does not have small norm, it could be unbounded (see \ref{pme}). This fact disrupts a number of  functional calculus
and results some lengthy careful calculation in this and other sections.}}
\fi

\begin{df}\label{DV0}
Let $n\in \N.$ Define the following subset of $C_0(\R^n)$:
\beq
V_{00}=\{ (t_1,\ldots, t_n)\mapsto f_1(t_1)f_2(t_2)\cdots f_n(t_n):  f_{i}\in C_0(\R), 1\le i\le n\}.
\eneq
%Here %$(t_1, t_2,...,t_n)\in \R^n$ and  
%$C_c(\R)$ is the set of  continuous functions having compact support.
Note that ${V_{00}}^*=V_{00}.$ Moreover, if $f, g\in V_{00},$ then 
$fg\in V_{00}.$ Put $V_0={\rm span}\{V_{00}\}$ then 
$V_0$ is a self-adjoint subalgebra of $C_0(\R^n).$

%Moreover $f_1(t_1)f_2(t_2)\cdots f_n(t_n)\in V_{00}$ if and only if one of $f_j\in C_0(\R).$ {\blue (Need all $f_j\in C_0(\mathbb R).$}
Recall for $(t_1, t_2,...,t_n)\in \R^n, $ we have the spherical coordinates
\[
r=1-2(1+\sum_{i=1}^n t_i^2)^{-1} \qquad s_j=(1+\sum_{i=1}^n t_i^2)^{-1/2}t_j(1+\sum_{i=1}^n t_i^2)^{-1/2}, j=1,\ldots, n
\] 
Define another subset of $C_0(\R^n)$ by
\beq
&&\hspace{-0.8in}
V_{00}^S=\{ (t_1,\ldots, t_n)\mapsto f_0(r)f_1(2s_1)\cdots f_n(2s_n):  f_{i}\in C_0(\R), \forall i, f_0(1)f_1(0)\cdots f_n(0)=0\}.
\eneq
%Define $V_0={\rm span}\{V_{00}\}.$ Then $V_0$ is a self-adjoint linear subspace of 
%$C_0(\R^n).$ 
We also have $(V_{00}^S)^*=V_{00}^S.$
If $f, g\in V_{00}^S,$ then $fg\in V_{00}^S.$ Define $V_0^S={\rm span}\{V_{00}^S\}.$
Then $V_0^S$ is a self-adjoint subalgebra of $C_0(\R^n).$ 
Note that both subalgebras $V_0$ and $V_0^S$ separate points of $\R^n$ and 
nowhere vanishes. 
%\beq
%V_1={\rm span}\, \{V_{0}, V_{0}^*, V_{0}^S, V_0^{S,*}\}.
%\eneq
Thus both 
%the 
subalgebras  $V_0$ and $V_0^S$ 
%generated by $\{V_0, V_0^*\}$  and by $\{V_0^S, V_0^{S,*}\}$ 
are dense in $C_0(\R^n)$ by the Stone-Weierstrass theorem.

%{\blue In fact, if $(x_1, \ldots, x_n)\neq (y_1, \ldots, y_n)$, then there is a $k$ such that $x_k\neq y_k$. Let $f_1, \ldots, f_n$ be continuous functions with compact support where $f_k(x_k)\neq f_k(y_k)$ and $f_i(x_i)\neq f_i(y_i)$ for $i\ne k$. Then $f_1(t_1)\cdots f_n(t_n)$ takes distinct values at $(x_1, \ldots, x_n)$ and $(y_1, \ldots, y_n)$. Hence, $V_{00}$ separate points and the claim follows by Stone-Weierstrass theorem.}  
\end{df}

Note that $(t_1, \ldots, t_n)\mapsto \Theta_{\xi, \eta}(t_1, \ldots, t_n)$ as in (\ref{eq:ThetaO}) is a function in $V_{00}$ and
$(t_1, \ldots, t_n)\mapsto \Theta_{\xi, \eta}^S(r, s_1, \ldots, s_n)$ as in (\ref{eq:ThetaS}) is a function in $V_{00}^S$. 
Introducing the following definition allows us to relate these functions to $\Theta_{\xi, \eta}(h_1, \ldots, h_n)$ as in (\ref{Dpsp-5}) and $\Theta_{\xi, \eta}^S(h_1, \ldots, h_n)$ as in (\ref{eq:ThetaSh}), respectively.

\begin{df}\label{Dphin}
Let $A$ be a unital \CA\, and $u_1,  u_2, ..., u_n\in A$ be $n$ commuting unitaries.
For each $f\in C(\T),$ define $f^{(j)}\in C(\T^n)$ by
$f^{(j)}((t_1,t_2,...,t_n))=f(t_j),$ where $(t_1, t_2,...,t_n)\in \T^n$ and $t_j\in \T,$
$1\le j\le n.$ 
Then one obtains a \hm\, $\Psi:\, C(\T^n)\to A$ 
such that $\Psi(f^{(j)})=f(u_j)$ for all $f\in C(\T).$

Suppose that $\phi_j: C_0(\R)\to A$ is a \hm, $1\le j\le n.$
Denote $\tilde \phi_j: C(\T)\to A$ by $\tilde \phi_j(f)=\phi(f)$ for all $f\in C_0(\R)$ and 
  $\tilde \phi_j(1_{C(\T)})=1_A,$ $j=1,2,...,n.$   Suppose 
  that $\phi_i(f)\phi_j(g)=\phi_j(g)\phi_i(f)$
  for all $f,g\in C_0(\R),$ $1\le i,j\le n.$
Then we obtain a \hm\, $\tilde \Phi: C(\T^n)\to A$ by $\tilde \Phi(f^{(j)})=\tilde \phi_j(f)$
for $f\in C(\T).$ Note that
\[C_0(\R)\otimes C_0(\R)\otimes \cdots{\otimes} C_0(\R)\subset  C(\T^n).
\]
Let $j: C_0(\R^n)\to C_0(\R)\otimes C_0(\R)\otimes \cdots {\otimes}C_0(\R)\subset C(\T^n)$ be the embedding.
Define $\Phi: C_0(\R^n)\to A$ by $\Phi(f)=\tilde \Phi(j\circ f)$ for all $f\in C_0(\R^n).$

Let $f_j\in C_0(\R),$ $1\le j\le n.$ 
Consider $f(t_1,t_2,...,t_n)=f_1(t_1)f_2(t_2)\cdots f_n(t_n)\in V_{00}.$
Then we have 
\beq
\Phi(f)=\phi_1(f_1)\phi_2(f_2)\cdots \phi_n(f_n).
\eneq
We will call $\Phi: C_0(\R^n)\to A$ the \hm\, induced by
$\phi_1, \phi_2, ..., \phi_n.$
\end{df}

%%%%%%%%%%%%%%
\iffalse
%
%
%%%%%%%%%%%%%%%%%%%%%
\begin{df}\label{DLV0}
Let $h_1, h_2,....,h_n$ be a $n$-tuple of densely defined self-adjoint operators on a separable 
infinite dimensional Hilbert space $H.$
Define 
a linear map $L: V_0\to B(H)$ by
\beq\label{LV1-1}
&&L(f_1(t_1)f_2(t_2)\cdots f_n(t_n))=f_1(h_1)f_2(h_2)\cdots f_n(h_n)\andeqn\\\label{LV1-2}
&&L(g_0(r)g_1(2s_1)\cdots g_n(2s_n))\\
&&\hspace{0.4in}=
g_0(1-2(1+\sum_{i=1}^n h_i^2)^{-1})g_1(2\bar h_1)\cdots g_n(2\bar h_n)
\eneq
for all  $f_1(t_1)f_2(t_2)\cdots f_n(t_n)\in {{V_{00}}}$ and $g_0(r)g_1(2s_1)\cdots g_n(2s_n)\in V_{00}^S$
\end{df}
%%%%%%%%
%
%
\fi
%%%%%%%%%%%%%%%%%%%%

%To establish the main theorem, we shall construct the homomorphism $\Lambda$ in Lemma~\ref{Lsyth1} using the map $L$. So Lemmas~\ref{Lfrom3}--\ref{Lfrom3-2} are provided to ensure that $\Lambda$ preserves multiplication.

\begin{lem}\label{Lfrom3}
Let $n\in \N.$ There exists $\dt>0$ satisfying the following:
Suppose that ${\{}h_1, h_2,...,h_n{\}}$ is an $n$-tuple of densely defined self-adjoint operators 
on an infinite dimensional separable Hilbert space $H$
such that
\beq
\|h_ih_j-h_jh_i\|<\dt,\,\,\, 1\le i,j\le n.
\eneq
Then, for $H_i=h_i(1+t h_i^2)^{-1},$ where $0<t\le 1,$  
\beq
%&&\hspace{-0.8in}(i)\hspace{0.8in} \,\,\,\|[h_i, (1+h_j^2)^{-1}]\|<2\dt,\,\,\,1\le i,\, j\le n,\\
%&&\hspace{-0.8in}(ii)\hspace{0.8in} \,\,\,\|[H_i, h_j]\|<2\dt,\,\,\,1\le i,\, j\le n,\\
&&\hspace{-0.8in}{\rm (i)}\hspace{0.8in}\,\,\,\|[H_i,\, H_j]\|<4\dt,\,\,\,1\le i,\, j\le n\andeqn\\
&&\hspace{-0.8in}{\rm (ii)} \hspace{0.8in}\,\,\, \|[H_i, (1+t h_j^2)^{-1}]\|<4\dt t^{1/2}, \,\,\,1\le i,\, j\le n.
\eneq
\end{lem}

\begin{proof}
%Put $H_j=h_j(1+h_j^2)^{-1},$ $j=1,2,...,n.$ 
%%%%%%%%%%%%%%%%%
\iffalse
(i) We first show that
\beq
\|(1+h_j^2)^{-1}((h_ih_j^3-h_j^3h_i)(1+h_j^2)^{-1}\|<3\dt.
\eneq

In fact, we have
\beq
&&\hspace{-0.4in}\|(1+h_j^2)^{-1}((h_ih_j^3-h_j^3h_i)(1+h_j^2)^{-1}\|\le 
\|(1+h_j^2)^{-1} (h_ih_j^3-h_jh_ih_j^2)(1+h_j^2)^{-1}\|\\
&&+\|(1+h_j^2)^{-1} (h_jh_ih_j^2-h_j^2h_ih_j)(1+h_j^2)^{-1}\|\\
&&+\|(1+h_j^2)^{-1} (h_j^2h_ih_j-h_j^3h_i)(1+h_j^2)^{-1}\|
\le \dt \|(1+h_j^2)^{-1} h_j^2(1+h_j^2)^{-1}\|\\
&&+\dt \|(1+h_j^2)^{-1} h_j^2(1+h_j^2)^{-1}\| +\dt\|(1+h_j^2)^{-1} h_j^2(1+h_j^2)^{-1}\|<3\dt.
\eneq
This proves (1).]
%%%%%%%%%%%%%%%%%%%%%%%%%
\fi

We begin with 
\beq
\hspace{-0.4in}\|(1+h_j^2)^{-1}h_i-h_i(1+h_j^2)^{-1}\|&=&\|(1+h_j^2)^{-1}(h_i(1+h_j^2)-(1+h_j^2)h_i)(1+h_j^2)^{-1}\|\\
&=&\|(1+h_j^2)^{-1}(h_ih_j^2-h_j^2h_i)(1+h_j^2)^{-1}\|\\
&\le& \|(1+h_j^2)^{-1}(h_ih_j^2-h_jh_ih_j)(1+h_j^2)^{-1}\|\\
&&+\|(1+h_j^2)^{-1}(h_jh_ih_j-h_j^2h_i)(1+h_j^2)^{-1}\|\\
&<&\dt\|(1+h_j^2)^{-1}\|\|h_j(1+h_j^2)^{-1}\|\\
&&+\dt\|(1+h_j^2)^{-1}h_j\|\|(1+h_j^2)^{-1}\|
\le 2\dt.
\eneq

We  have 
\beq\nonumber
\|H_jh_i-h_iH_j\|&=&\|(1+th_j^2)^{-1}(h_jh_i(1+th_j^2)-(1+th_j^2)h_ih_j)(1+th_j^2)^{-1}\|\\\nonumber
&=&\|(1+th_j^2)^{-1}((h_jh_i+th_jh_ih_j^2)-(h_ih_j+th_j^2h_ih_j))(1+th_j^2)^{-1}\|\\
&<& \dt\|(1+th_j^2)^{-1}\| \|(1+th_j^2)^{-1}\|\\
  &&+t\|(1+th_j^2)^{-1}(h_jh_ih_j^2-h_j^2h_ih_j)(1+th_j^2)^{-1}\|\\
  &\le & \dt+\dt\|(1+th_j^2)^{-1}t^{1/2}h_j\| \| t^{1/2}h_j (1+th_j^2)^{-1}\|\le 2\dt. \label{ii}
\eneq
%%%%%%%

%%%%%%%%%%%%%%%%%%%%%
We can now prove (ii). By applying the above inequality, we obtain that
\beq\nonumber
\|(1+th_j^2)^{-1}H_i-H_i(1+th_j^2)^{-1}\|&=&
\|(1+th_j^2)^{-1}(H_i(1+th_j^2)-(1+th_j^2)H_i)(1+th_j^2)^{-1}\|\\\nonumber
&=&t\|(1+th_j^2)^{-1}(H_ih_j^2-h_j^2H_i)(1+th_j^2)^{-1}\|\\
&\le & t\|(1+th_j^2)^{-1}(H_ih_j^2-h_jH_ih_j)(1+th_j^2)^{-1}\|\\
&&+t\|(1+th_j^2)^{-1}(h_jH_ih_j -h_j^2H_i)(1+th_j^2)^{-1}\|\\
&<& 2\dt t\|(1+th_j^2)^{-1}h_j(1+th_j^2)^{-1}\|\\
&&+ 2\dt t\|(1+th_j^2)^{-1}h_j(1+th_j^2)^{-1}\|\\
&&\le 4\dt t^{1/2}.
\eneq
%%%%%%

%%%%%%%%%%%%%%%%%%%%
To see (i), again by applying (\ref{ii}), we estimate that
\beq
\hspace{-0.3in}\|H_jH_i-H_i H_j\|&=&\|(1+th_j^2)^{-1}h_jH_i-H_ih_j(1+th_j^2)^{-1}\|\\
&=&\|(1+th_j^2)^{-1}(h_jH_i(1+th_j^2)-(1+th_j^2)H_ih_j)(1+th_j^2)^{-1}\|\\
&\le & \|(1+th_j^2)^{-1}(h_jH_i-H_ih_j)(1+th_j^2)^{-1}\|\\
&&+t\|(1+th_j^2)^{-1}(h_jH_ih_j^2-h_j^2H_ih_j)(1+th_j^2)^{-1}\|\\
&&\hspace{-0.6in}<2\dt\|(1+th_j^2)^{-1}\|^2 +2\dt t\|(1+th_j^2)^{-1}h_j\| \|h_j(1+th_j^2)^{-1}\|\\
&<&4\dt.
\eneq
\end{proof}

\begin{df}
In what follows, we denote 
\beq
C(\overline{\R})=\{f\in C(\R): {\rm both} \lim_{x\to + \infty} f(x)\andeqn \lim_{x\to-\infty} f(x)\,\,\, {\rm exist}\}.
\eneq
\end{df}

\begin{lem}\label{LHlambda}
Let ${\cal F}\subset  C_0(\R)$ be a  finite subset
and $\ep>0.$ 
Then there is $1\ge t>0$  satisfying the following:
%Then there is $1\ge \lambda>0$  satisfying the following:
There exists
%For any $0<t \le \lambda,$ there exists 
$\phi\in C(\overline \R)$ such that
for any densely defined self-adjoint operator $h$ on an infinite dimensional separable Hilbert space
$H,$ 
%for $0<t\le \lambda,$ 
\beq
\|f\circ \phi(h(1+t h^2)^{-1})-f(h)\|<\ep\rforal f\in {\cal F}.
\eneq

\end{lem}

\begin{proof}
First let us assume that ${\cal F}\subset C_c(\R).$ 
Choose $M>1$ such that, if $|x|\ge M,$ 
\beq
|f(x)|=0. 
\eneq

Choose $0<t<\frac{1}{(M+2)^2}$ and
%Choose $0<\lambda<{1\over{(M+2)^2}}.$  For $0<t\le \lambda,$ 
define $g(x)={x\over{1+t x^2}}$
for $x\in \R.$ Then $g\in C_0(\R).$ 
Note that 
\beq
g'(x)={1-tx^2\over{(1+tx^2)^{2}}}>0\rforal x\in [-M-1,M+1].
\eneq
Thus $g$ has an inverse $\tilde g$ on $[g(-M-1), g(M+1)].$
Define $\phi(x)=\tilde g(x)$ if $x\in [g(-M-1), g(M+1)],$ $\phi(x)={\tilde g(M+1)}$ if $x\in (g(M+1), \infty)$
and $\phi(x)={\tilde g(-M-1)}$ if $x\in (-\infty, g(-M-1)).$
Hence, $\phi\in C(\overline \R).$
Then  (recall that $f(x)=0$ if $|x|> M$)
\beq
f\circ \phi(g(x))=f(x)\rforal x\in \R.
\eneq
It follows that 
\beq
f\circ \phi(h(1+t h^2)^{-1})=f(h).
\eneq

In general, let $\ep>0.$ 
For ${\cal F}\subset C_0(\R),$  choose a finite subset 
${\cal F}_0\subset C_c(\R)$ such that, for any $f\in {\cal F},$ there exists $g\in  {\cal F}_0$ such that
\beq
\|f-g\|_{_\infty}<\ep/2.
\eneq

By what has been proved, there exists $t>0$ depending on $\mathcal{F}_0$ and $\phi\in C(\overline \R)$ such that
%By what has been proved, there exist $\lambda>0$ and $\phi\in C(\overline \R)$ such that, for $0<t<\lambda,$
\beq
g\circ \phi(h(1+t h^2)^{-1})=g(h)\rforal g\in {\cal F}_0.
\eneq
Then, if $\|f-g\|_\infty<\ep/2,$ then 
\beq
\|f\circ \phi-g\circ \phi\|_{_\infty}<\ep/2.
\eneq
Hence, by {{Remark}}~\ref{RfH},
\beq
\|f\circ \phi(h(1+t h^2)^{-1})-f(h)\|\le \|f\circ \phi(h(1+t h^2)^{-1})-g\circ \phi(h(1+t h^2)^{-1})\|\\+
\|g\circ \phi(h(1+t h^2)^{-1})-g(h)\|
+\|g(h)-f(h)\|\\
<\ep/2+\ep/2=\ep.
\eneq

\end{proof}

\begin{lem}\label{Lfrom3-2}
Let $\ep>0$ and $n\in \N.$  Let ${\cal F}\subset C_0(\R)$ be a finite subset.
There exists $\dt>0$ satisfying the following:
Suppose that $\{h_1, h_2,...,h_n \}$ is an $n$-tuple of densely defined self-adjoint operators 
on an infinite dimensional separable Hilbert space $H$
such that
\beq
\|h_ih_j-h_jh_i\|<\dt,\,\,\, 1\le i,j\le n.
\eneq
Then 
\beq
\|f(h_i)g(h_j)-g(h_j)f(h_i)\|<\ep \tforal f, g\in {\cal F}.
\eneq
\end{lem}

\begin{proof}
Let $\epsilon>0$ and $\mathcal F\subset C_0(\mathbb R)$ be finite. Set $N=\max_{l\in\mathcal F}\|l\|_{\infty}.$
By Lemma~\ref{LHlambda}, there exist $\phi\in C(\overline \R)$ 
and $0<t\le 1$ such that for any 
%$n$-tuple of 
densely defined selfadjoint operator $h,$
%s $h_1, \ldots, h_n$,
\beq \label{eq:lphil}
\| l\circ\phi(h(1+th^2)^{-1})-l(h)\|< \epsilon/8N \rforal l\in\mathcal F.
\eneq
Since $t>0,$ there exists $K>0$ such that
\beq\label{eqK}
\sup_{s\in \R} |{s\over{1+ts^2}}|\le K.
\eneq

Put ${\cal G}=\{f\circ \phi: f\in {\cal F}\}.$   Since $\phi\in C(\overline{\R}),$
the image of $\phi$ is in a bounded interval of $\R.$ Therefore ${\cal G}\subset C^b(\R).$  In particular, 
$g|_{[-K, K]}\in C([-K,K])$ for all $g\in {\cal G}.$  
Let ${\cal G}_1=\{g|_{[-K, K]}: g\in {\cal G}\}.$ 
Then ${\cal G}_1\subset C([-K, K])$ is a finite subset.

We will use the following well-known claim. 

For $p,q\in {\cal G}_1$ and $\epsilon>0$, there exists $\delta'>0$ such that for any self-adjoint 
elements $X,Y\in A$ (any \CA\, $A$) with $\|X\|, \|Y\|\le K$, whenever $\|[X,Y]\|<\delta'$, we have $\|[p(X), q(Y)]\|<\epsilon/2.$

In fact, suppose the claim does not hold. Then there is $\epsilon_0>0$ and a sequence of 
\CA s\, $A_m$ and self-adjoint elements $X_m, Y_m\in A_m$ with $\|[X_m, Y_m]\|<1/m$ while $\|[p(X_m), q( Y_m)]\|>\epsilon_0,$ $m\in \N.$
Then $\Pi(\{X_m\})$ and $\Pi(\{Y_m\})$ commute in $\prod_mA_m/\bigoplus_m A_m,$ where $\Pi$ denotes the quotient map. 
Thus, 
\[
\Pi(\{[p(X_m), q(Y_m)]\})=[p(\Pi(\{X_m\})), q(\Pi(\{Y_m\}))]=0
\]
in $\prod_mA_m/\bigoplus_m A_m$.
This implies that the existence of $m$ such that $\|[p(X_m), q(Y_m)]\|<\epsilon_0$, which is a contradiction. Claim is proved.

Now choose $\delta=\delta'/4.$  Suppose $\|h_ih_j-h_jh_i\|<\delta$ for $i,j=1, \ldots, n$. 
For the chosen $0<t<1$ above, let $H_i=h_i(1+th_i^2)^{-1},$ $1\le i\le n.$ Then, by \eqref{eqK} and the spectral theorem,
\beq
\|H_i\|\le K,\,\,\, 1\le i\le n.
\eneq
Then, by Lemma~\ref{Lfrom3}, $\|[H_i, H_j]\|<4\delta=\delta',$ $1\le i,j\le n.$
%
%for some $t>0$

For $f,g\in\mathcal F$, denote by 
\[A=f(h_i), \quad A'=f\circ\phi(h_i(1+th_i^2)^{-1}) \quad B=g(h_j), \quad B'=g\circ\phi(h_j(1+th_j^2)^{-1}).
\]
Then (\ref{eq:lphil}) implies that $\|A-A'\|<\epsilon/8N$ and $\|B-B'\|<\epsilon/8N.$ 
We estimate that, applying the claim at the last inequality,  for $1\le i, j\le n,$
%\begin{align}
\beq \label{eq:fhighj}
&&\hspace{-0.6in}\|f(h_i)g(h_j)-g(h_j)f(h_i)\|\\
&=&\|AB-BA\| \\
&=&\|A(B-B')+(A-A')B'-B(A-A')-(B-B')A'+[A', B']\| \\
&\le & 2\|f\|_{\infty}(\epsilon/8N)+2\|g\|_{\infty}(\epsilon/8N)+\|[f\circ\phi(H_i), g\circ\phi(H_j)]\|\\
&<& \ep/4+\ep/4+\ep/2.
\eneq
The lemma is then proved.
\end{proof}

\begin{lem}\label{Lproduct}
%Let $\ep>0.$ There exists $\dt>0$ satisfyinhg the following 
Let $A$  be a \CA\, and $h_1, h_2,...,h_n, s_1, s_2,...,s_n$  in $A_+^{\bf 1}.$
Suppose that $h_i^2\ge s_i^2,$ $i=1,2,...,n.$ 
Then, when $\|[h_i,\, s_j]\|<\dt,\,\, 1\le i, j\le n$  one has that 
\beq
\|h_1h_2\cdots h_n\|\ge \|s_1s_2\cdots s_n\|-\sqrt{n(n-1)\dt}.
\eneq
\end{lem}

\begin{proof}
\iffalse 
First consider $n=2.$
We have
$
s_2s_1s_1s_2\le s_2h_1^2s_2.
$
Hence  
\beq
\|s_2s_1^2s_2\|\le \|s_2h_1^2s_2\|
\eneq
It follows that $\|s_1s_2\|\le \|h_1s_2\|.$ 
Since $s_2^2\le h_2^2,$ by what has just proved,
\beq
\|s_1s_2\|\le \|h_1s_2\|=\|(h_1s_2)^*\|=\|s_2^*h_1^*\|=\|s_2h_1\|\le \|h_2h_1\|=\|h_1h_2\|.
\eneq
\fi
%Choose $0<\dt<\ep/2(n(n-1).$
We estimate that 
\beq\nonumber
\|s_1s_2\cdots s_n\|^2&=&\|s_1s_2\cdots s_ns_n\cdots s_2s_1\|\le  \|s_1s_2\cdots s_{n-1}h_n^2s_{n-1}\cdots s_2s_1\|\\\nonumber
&\le& \|s_1s_2\cdots h_ns_{n-1}^2h_n \cdots s_2s_1\|+2\dt\le \|s_1s_2\cdots h_nh_{n-1}^2h_n\cdots s_2 s_1\|
+2\dt\\\nonumber
&\le& \|s_1s_2\cdots h_nh_{n-1}s_{n-2}^2h_{n-1}h_n\cdots s_2s_1\|+4\dt+2\dt\\
&\le & \|s_1s_2\cdots h_nh_{n-1}h_{n-2}^2h_{n-1}h_n\cdots s_2s_1\|+2(1+2)\dt\\
&\le&  \|h_nh_{n-1}\cdots h_2h_1^2h_2\cdots h_{n-1}h_n\|+2(\sum_{j=1}^{n-1} j)\dt\\
&<&\|h_1h_2\cdots h_{n-1}h_n\|^2+n(n-1)\dt.
\eneq
Hence
$$
\|s_1s_2\cdots s_n\|\le \|h_1h_2\cdots h_{n-1}h_n\|+\sqrt{n(n-1)\dt}.
$$
\end{proof}

\begin{lem}\label{Ladd251112-2}
Let $0<\eta<\dt/4<1/4$ and let $M>1.$ There exists $\af>0$ satisfying the following:
If $h_1, h_2,...,h_n$ are $n$-tuple selfadjoint operators such that $\|[h_i,\, h_j]\|<\af,$
then 
\beq
s{\rm Sp}_M^\eta ((h_1, h_2,...,h_n))\subset s{\rm Sp}_M^\dt((h_1, h_2,...,h_n));\\
S{\rm Sp}_M^\eta((h_1,h_2,..., h_n))\subset S{\rm Sp}_M^\dt((h_1,h_2,..., h_n)).
\eneq
\end{lem}

\begin{proof}
Here we fix finite subsets $D_M^\eta$ and $D_M^\dt$ as in Definition \ref{Dsynsp2}.
 Let ${\cal F}\subset C_0(\R)$ consist of 
 $\theta_{\xi_i, \eta}$ and $\theta_{\zeta_i, \dt},$ where $\xi_i,\zeta_i$ are coordinates for
 $\xi=(\xi_1, \xi_2,...,\xi_n)\in D_M^\eta$ and  $\zeta=(\zeta_1, \zeta_2,...,\zeta_n)\in D_M^\dt.$
 By Lemma \ref{Lfrom3-2}, we choose $\af$ such that, when $\|[h_i,\, h_j]\|<\af,$ 
 \beq\label{eq1116-1}
 \|f(h_i)g(h_j)-g(h_j)f(h_i)\|<(\dt-\eta)^2/4n(n-1),\,\,\, 1\le i,j\le n.
 \eneq
As $D_M^\dt$ is $\dt/2$-dense in $\{\zeta\in \R^n: \|\zeta\|_2\le M\},$ 
for each $\xi=(\xi_1, \xi_2,...,\xi_n)\in D^\eta,$ 
there exists $\zeta=(\zeta_1, \zeta_2,...,\zeta_n)\in D_M^\dt$ 
such that  $|\xi_i-\zeta_i|<\dt/2$ ($1\le i\le n$). 
It follows that $|\xi_i-\eta-\zeta_i|\le \dt/2+\eta<\dt/2+\dt/4=3\dt/4.$
Thus
\beq
(\theta_{\xi_i, \eta})^2\le (\theta_{\zeta_i, \dt})^2,\,\,\, 1\le i\le n.
\eneq
Hence 
\beq
(\theta_{\xi_i, \eta}(h_i))^2\le (\theta_{\zeta_i, \dt}(h_i))^2,\,\,\, 1\le i\le n. 
\eneq
By \eqref{eq1116-1},
\beq
\|[\theta_{\zeta_i, \dt}(h_i),\,\theta_{\xi_j, \eta}(h_j)]\|<(\dt-\eta)^2/4n(n-1),\,\,\, 1\le i,j\le n.
\eneq
Applying Lemma \ref{Lproduct},  we obtain 
%if $\|[h_i,\, h_j]\|<(\dt-\eta)/2(n(n-1)),$
\beq
\|\Theta_{\xi, \eta}(h_1, h_2,...,h_n)\|-(\dt-\eta)/2\le \|\Theta_{\zeta_i, \dt}((h_1, h_2,...,h_n))\|.
\eneq
If $\|\Theta_{\xi, \eta}(h_1, h_2,...,h_n)\|\ge 1-\eta,$ then 
\beq\nonumber
\|\Theta_{\zeta_i, \dt}((h_1, h_2,...,h_n))\|\ge \|\Theta_{\xi, \eta}(h_1, h_2,...,h_n)\|-(\dt-\eta)/2
\ge 1-\eta-(\dt-\eta)/2\ge 1-\dt.
\eneq
It follows that 
\beq
s{\rm Sp}_M^\eta((h_1, h_2, ...,h_n))\subset s{\rm Sp}_M^{\dt}((h_1, h_2,...,h_n)).
\eneq
The other inclusion follows similarly by applying Lemma \ref{Lproduct} for $n+1.$ 
\end{proof}

\begin{lem}\label{Linequality}
Let  $A$ be a unital \CA\, and $a,b_1, b_2,..., b_n\in A_+$ be invertible elements
in a commutative \SCA\, of $A.$
Suppose that $a\le  b_i,$ $i=1,2,...,n.$
Then 
\beq
a\le (b_1b_2\cdots b_n)^{1/n}
\eneq
\end{lem}

\begin{proof}
Since $b_i\ge a,$ $i=1,2,...,n,$ 
one has,  for $1\le i\le n,$  
\beq
b_i^{1/2}ab_i^{1/2}\le  b_i^{1/2}b_1 b_i^{1/2}=b_1b_i.
\eneq
Hence
\beq
a^{2}=a^{1/2} aa^{1/2}\le a^{1/2} b_i a^{1/2}=b_i^{1/2}a b_i^{1/2}\le b_1b_i.
\eneq
By induction, we obtain that
\beq
a^{n}\le b_1b_2\cdots b_n.
\eneq
Hence 
\beq
a\le (b_1b_2\cdots b_n)^{1/n}
\eneq
as desired.
\end{proof}

\begin{lem}\label{Lsyth1}
Let $n\in \N.$  
Suppose that 
$h_{1,k}, h_{2,k},...,h_{n,k}$ are
%For any $\ep>0$ and a finite subset ${\cal F}\subset V_{00}\cup V_0,$
%there exists  $\dt>0$ satisfying the following:
 $n$-tuples of densely defined self-adjoint operators on a separable 
infinite dimensional Hilbert space $H$
such that
\beq \label{eq:seqcon}
\lim_{k\to\infty} \|h_{i, k}h_{j,k}-h_{j,k}h_{i,k}\|=0,\,\,\, 1\le i, j\le n.
\eneq
Define $L_k: V_{00}\to B(H)$ by
\beq
L_k(f)=f_1(h_{1,k})f_2(h_{2,k})\cdots f_n(h_{n,k})
\eneq
if $f((t_1,t_2,...,t_n))=f_1(t_1)f_2(t_2)\cdots f_n(t_n)$ and $f_j\in C_0(\R),$ $1\le j\le n,$ $k\in \N.$ 
%\to B(H)$  as in Definition \ref{DLV0}
%associated with $h_{1,k},h_{2,k},...,h_{n,k},$ $k\in \N.$ 

{\rm (1)} Then there exists a \hm\,
$\Lambda: C_0(\R^n)\to l^\infty(B(H))/c_0(B(H))$
such that
\beq
\Lambda(f)=\Pi(\{L_k(f)\})\rforal f\in V_{00},
\eneq
where $\Pi: l^\infty(B(H))\to l^\infty(B(H))/c_0(B(H))$ is the quotient map;

{\rm (2)} Let $\Psi: C_0(S^n\setminus \{\zeta^{np}\})\to l^\infty(B(H))/c_0(B(H))$ be generated by the sequences  (see Definition~\ref{Dpsi})
\[
(\{(\sum_{j=1}^n h_{j,k}^2-1)(\sum_{j=1}^n h_{j,k}^2+1)^{-1}\}, \{\bar{h}_{1, k}\}, \{\bar{h}_{2, k}\},..., \{\bar{h}_{n, k}\}).
\]
Then 
\beq
\Psi(f)=\Lambda(f\circ \Phi_{S^n})\rforal f\in  C_0(S^n \setminus \{\zeta^{np}\}).
\eneq

\end{lem}

\begin{proof}
(1) For $1\le i\le n$, consider the homomorphism \[
\phi_i: C_0(\R)\rightarrow l^{\infty}(B(H))/c_0(B(H)) \quad f\mapsto\Pi(\{f(h_{i,k})\}_k)
\]
By \eqref{eq:seqcon} and Lemma~\ref{Lfrom3-2}, one has \[
\phi_i(f)\phi_j(g)=\phi_j(g)\phi_i(f)
\]
for all $1\le i,j\le n$ and $f,g\in C_0(\R)$.
Thus, by Definition~\ref{Dphin}, we obtain a \hm\, 
$\Lambda: C_0(\R^n)\to l^\infty(B(H))/c_0(B(H))$  such that
%%%%%%%%%%%%%%%%%%%%%%%%%%%%%%%%%%
%
\iffalse
Thus $\{\Pi(\{L_k(f)\}: f\in V_0\}$ forms a commutative \SCA\, of $l^\infty(B(H))/c_0(B(H)).$
Since ${\rm span} \{V_0, V_0^*\}$ is a self-adjoint  linear subspace and the subalgebra generated by 
${\rm span} \{V_0, V_0^*\}$ is dense in $C_0(\R^n),$ 
one obtains a \hm\, $\Lambda: C_0(\R^n)\to l^\infty(B(H))/c_0(B(H))$ such that
%%%%%%%%%
%
%
\fi
%%%%%%%%%%%%%%%%%%%%%
\beq
\Lambda(f)&=&\Pi(\{L_k(f)\})\rforal f\in V_{00},
\,\,\, {\rm {consequently,}}\\
 \Lambda(f)&=&\Pi(\{L_k(f)\}){\rforal f\in V_{0}.}
\eneq

{(2) By the definition, $\Psi$ sends a general element $(r, 2s_1, \ldots, 2s_n)\mapsto f_0(r)f_1(2s_1)\cdots f_n(2s_n)$ in $C_0(S^n\backslash\{\zeta^{np}\})$ to 
\[
\Pi\left(f_0((1-\sum_{i=1}^n h_{i,k}^2)(1+\sum_{i=1}^n h_{i,k}^2)^{-1})f_1(2\bar h_{1,k})\cdots f_n(2\bar h_{n,k})\right).
\]
%\beq
%\Psi((1-\sum_{i=1}^n t_i^2)(1+\sum_{i=1}^n t_i^2)^{-1})&=&(1-\sum_{i=1}^n h_i^2)(1+\sum_{i=1}^n h_i^2)^{-1},\\
%\Psi(t_j(1+\sum_{i=1}^n t_i^2)^{-1})&=&\bar h_j,\,\,\, j=1,2,...,n.
%\eneq
%Composing $\Psi$ with $(\Phi_{S^n}^{\#})^{-1}$, where the isomorphism 
Recall that the map $\Phi_{S^n}^{\#}: C_0(S^n\backslash\{\zeta^{np}\})\rightarrow C_0(\R^n)$  induced by the homeomorphism  $\Phi_{S^n}: \R^n\rightarrow S^n\backslash\{\zeta^{np}\}$
is an isomorphism. 
Consider the  homomorphism 
\[
\Psi\circ(\Phi_{S^n}^{\#})^{-1}: C_0(\R^n)\to l^\infty(B(H))/c_0(B(H)).
\]
By the definition, we obtain
\beq\label{Lsyn1-3}
\Psi\circ(\Phi_{S^n}^{\#})^{-1}(f)=\Pi(\{L_k^S(f)\})\rforal f\in V_0^S,
\eneq 
where 
\beq
L_k^S(f_0f_1\cdots f_n)=f_0(r_k) f_1(2{\bar h}_{1,k})\cdots f_n(2{\bar h}_{n,k}),
\eneq
$f_j\in C_0(\R),$ $0\le j\le n,$ and $r_k=(1-\sum_{i=1}^n h_{i,k}^2)(1+\sum_{i=1}^n h_{i,k}^2)^{-1},$
$k\in \N.$

(i) Claim: $\Lambda({1\over{1+\sum_{i=1}^n t_i^{2}}})=\Pi((\{1+\sum_{i=1}^n h_{i,k}^2)^{-1}\})$\,\,\,\,
($(t_1, t_2,...,t_n)\in \R^n$).

There are $g_k\in V_0$ such that
\beq\label{Lsyth1-10}
\lim_{k\to\infty}\sup_{_{(t_1,t_2,...,t_n)\in \R^n}} \left|g_k((t_1, t_2,...,t_n))-{1\over{1+\sum_{i=1}^n t_i^2}}\right|=0.
\eneq
Let $\Delta_i(t)={1\over{(1+t_i^2)^2}}$ and $\Delta=\prod_{i=1}^n \Delta_i.$
Note that $\Delta, \Delta^{1/2}\in V_{00}$ and they are strictly positive in $C_0(\R^n).$
Moreover (for $t=(t_1, t_2,...,t_n)$)
\beq
t_i^2\Delta(t)={t_i^2\over{(1+t_i^2)^2}}\prod_{j\not=i}^n \Delta_j \in V_{00}.
\eneq
Hence
\beq
(1+\sum_{i=1}^n t_i^2)\Delta(t)\in V_0.
\eneq
Therefore 
\beq\label{Lsyn1-13}
\Lambda((1+\sum_{i=1}^n t_i^2)\Delta(t))=\Pi(\{(1+\sum_{i=1}^n h_{i,k}^{2})\Delta_1(h_{1,k})\Delta_2(h_{2,k})\cdots \Delta_{n}(h_{n,k})\}).
\eneq
By \eqref{Lsyth1-10}
\beq\label{Lsyth1-15}
\lim_{k\to\infty}\|\Lambda((1+\sum_{i=1}^n t_i^2)\Delta(t))\Lambda(g_k)-\Lambda({(1+\sum_{i=1}^n t_i^2)\Delta(t)\over{1+\sum_{i=1}^n t_i^2}})\|=0.
\eneq 
However
\beq
\Lambda({(1+\sum_{i=1}^n t_i^2)\Delta(t)\over{1+\sum_{i=1}^n t_i^2}})=\Lambda(\Delta).
\eneq
Combining with \eqref{Lsyn1-13} and \eqref{Lsyth1-15}, we obtain 
that
\beq\nonumber
&&\hspace{-0.2in}\Pi(\{(1+\sum_{i=1}^n h_{i,k}^{2})\Delta_1(h_{1,k})\Delta_2(h_{2,k})\cdots \Delta_n(h_{n,k})\})\left(\Lambda({1\over{1+\sum_{i=1}^n t_i^2}})-\Pi(\{(1+\sum_{i=1}^n h_{i,k}^{ 2})^{-1}\})\right)=0,\\\label{Lsyth1-19}
&&{\rm or},\,  \Lambda(\Delta(t)(1+\sum_{i=1}^n { t}_i^2))\left(\Lambda({1\over{1+\sum_{i=1}^n t_i^2}})-\Pi(\{(1+\sum_{i=1}^n h_{i,k}^2)^{-1}\})\right)=0.
\eneq
Let $A=\overline{\Lambda(C_0(\R))(l^\infty(B(H))/c_0(B(H)))\Lambda(C_0(\R^n))}.$
Then $A$ is a hereditary \SCA\, of $l^\infty(B(H))/c_0(B(H)).$
Note that, for all $j=1,2,...,n,$  
\beq
(1+\sum_{i=1}^n h_{i,k}^2)^{-2}\le (1+h_{j,k}^2)^{-2}=\Delta_j(h_{j,k}).
\eneq
By Lemma \ref{Linequality},
\beq
\Pi(\{(1+\sum_{i=1}^n h_{i,k}^{2})^{-2}\})\le (\Pi(\{\Delta_1(h_{1,k})\Delta_2(h_{2, k})\cdots \Delta_n(h_{n,k})\})^{1/n}.
\eneq
It follows that $\Pi(\{(1+\sum_{i=1}^n h_{i,k}^{2})^{-2})\})\in A,$ whence 
\beq\label{Lsyth-1-20}
%\Pi(\{(1+\sum_{i=1}^n h_{i,k}^{2})^{-2})\})\in A,\,\, {\red{hence}\,\,\,
\Pi(\{(1+\sum_{i=1}^n h_{i,k}^{2})^{-1})\})
=(\Pi(\{(1+\sum_{i=1}^n h_{i,k}^{2})^{-2})\}))^{1/2}\in A.
\eneq
On the other hand, $\Delta((t_1, t_2,...,t_n))(1+\sum_{i=1}^n t_i^2)$ is a strictly 
positive element in $C_0(\R^n).$ Hence 
$\Lambda(\Delta(t)(1+\sum_{i=1}^nt_i^2))$ is a strictly positive element of 
$A.$   It follows from \eqref{Lsyth1-19} (also using \eqref{Lsyth-1-20}) that 
\beq
\Lambda({1\over{1+\sum_{i=1}^n t_i^2)}})=\Pi(\{(1+\sum_{i=1}^n h_{i,k}^2)^{-1}\}).
\eneq
This proves the claim (i).

(ii) $\Lambda({t_j\over{1+\sum_{i=1}^n t_i^2}})=\Pi(\{(1+\sum_{i=1}^n h_{i,k}^2)^{-1} h_{j,k}\})$
for  all ${j}=1,2,...,n.$

 Note that  $t_j\Delta\in V_0.$ 
 It follows that
 \beq
 \Lambda(t_j\Delta)=\Pi(\{h_{j,k}\Delta_1(h_{1,k})\Delta_2(h_{2, k})\cdots \Delta_n(h_{n,k})\}).
 \eneq
We then compute that (applying Claim (i))
\beq
\left(\Lambda(t_j(1+\sum_{i=1}^n t_i^2)^{-1})-\Pi(\{h_{j, k}(1+\sum_{i=1}^n h_{i,k}^2)^{-1})\}\right)\Lambda(\Delta)=0.
\eneq
Applying \eqref{Lsyth-1-20},  since $\Lambda(\Delta)$ is a strictly positive element of $A,$ we obtain that
\beq
\Lambda(t_j(1+\sum_{i=1}^n t_i^2)^{-1})=\Pi(\{h_{j, k}(1+\sum_{i=1}^n h_{i,k}^2)^{-1}\}).
\eneq
This proves the Claim (ii).

By Claim (i) and (ii) as well as \eqref{Lsyn1-3}, we conclude that
\beq
\Psi\circ(\Phi_{S^n}^{\#})^{-1}(f)=\Lambda(f)\rforal f\in V_0^S.
\eneq
Since $V_0^S$ is dense in $C_0(\R^n),$ the lemma then follows.}
%%%%%%%%%
\iffalse
%%%%%%%%%%%%%%
Note that $V_0$ belongs to $C_0(\R)$, closure of the subalgebra generated by $\{V_0^S, V_0^{S,*}\}$. 
Without loss of generality suppose $h_1(t_1)\cdots h_n(t_n)$ of $V_0$ has the form
$\sum_m a_m f_{0,m}(r)\cdots f_{n,m}(s_n)$ where $f_{0,m}(r)\cdots f_{n,m}(s_n)\in V_0^S.$
Then for every $\epsilon>0$, there exists $N$ such that 
\[
\|h_1(t_1)\cdots h_n(t_n)-\sum_{m\le N} a_m f_{0,m}(r)\cdots f_{n,m}(s_n)\|<\epsilon. 
\]
Note that the element in the norm belongs to $V_1$ and $\Pi\circ L_k$ is a homomorhism on $V_1$. The continuity of $\Pi\circ L_k$ implies that $\Psi\circ(\Phi_{S^n}^{\#})^{-1}$ and $\Lambda$ agree on $V_0$.
Similarly, one can further show they agree on the algebra generated by ${\rm span}\{V_0, V_0^S\}$ and hence by continuity they agree on $C_0(\R^n).$ 

It follows that
\beq
\Psi(f)=[\Lambda\circ\Phi^{\#}_{S^n}](f)=\Lambda(f\circ \Phi_{S^n})\rforal f\in  C_0(S^n \setminus \{\zeta^{np}\}).
\eneq
\fi
%%%%%%%%%%%%%%%%%%%%%%
\end{proof}

\begin{lem} \label{LsS-2}
Let  $n\in \N,$ $M\ge 1$ and  $0<\eta<1$ be such that
$F_M(\eta)<1,$ and $0<\dt<1$ such that 
$G_M(\dt)<1.$

 Fix $D_M^\eta, D_M^\dt \subset {{\{\zeta\in\R^n: \|\zeta\|_2\le M\} }}$  as in 
 the end of Definition~\ref{Dsynsp2}. 
 
Then, there exists $\alpha>0$ such that for any $n$-tuple of densely defined self-adjoint operators $h_1, h_2,...,h_n$ on a Hilbert space satisfying 
$\| h_ih_j-h_jh_i\|<\alpha,$ $i,j\in\{1,2,...,n\},$ the following inclusions hold
\beq
&&S{\rm Sp}^\eta_M((h_1, h_2,...,h_n))\subset s{\rm Sp}^{\sigma}_M((h_1, h_2,...,h_n))\andeqn\\
&&s{\rm Sp}^\dt_M ((h_1, h_2,...,h_n){ )}\subset S{\rm Sp}_M^{\gamma}((h_1, h_2,...,h_n)),
%dist_H(S{\rm Sp}^\eta(h_1, h_2,...,h_n), s{\rm Sp}^{F_M(\sqrt{n}\eta)}(h_1, h_2,...,h_n))<2(\eta+F_M(\eta)) 
\eneq
where $\sigma=\max\{\eta, F_M(\eta)+\sigma_1\},$ for any  $0<\sigma_1\le \min\{\eta/4, F_M(\eta)/4\},$ and
 $\gamma=\max\{\dt, G_M(\dt)+\dt_1\},$ for any $0<\dt_1\le \min\{\dt/4, G_M(\dt)/4\}.$
\end{lem}

\begin{proof}
Fix $n, M$ and let $\eta, \delta$ be small enough to satisfy $F_M(\eta)<1$ and $G_M(\delta)<1.$
By way of contradiction, for any $\alpha_m=\frac1{m}$, there exist self-adjoint operators $h_{1,m}, \ldots, h_{n,m}$ such that $\|h_{i,m}h_{j,m}-h_{j,m}h_{i,m}\|<\frac1{m}$ such that 
\beq
&&S{\rm Sp}^\eta_M((h_{1,m}, h_{2,m},...,h_{n,m}))\not\subset s{\rm Sp}^{\sigma}_M((h_{1,m}, h_{2,m},...,h_{n,m})) \text{ or }  \label{eq:SSpsSp} \\
&&s{\rm Sp}^\dt_M((h_{1,m}, h_{2,m},...,h_{n,m}))\not\subset S{\rm Sp}_M^{\gamma}((h_{1,m}, h_{2,m},...,h_{n,m})), \label{eq:sSpSSp}
%dist_H(S{\rm Sp}^\eta(h_1, h_2,...,h_n), s{\rm Sp}^{F_M(\sqrt{n}\eta)}(h_1, h_2,...,h_n))<2(\eta+F_M(\eta)) 
\eneq
for all $\sigma, \gamma>0.$ 
We now show none of them holds. 

Following the notation of Lemma \ref{Lsyth1},  associated with the sequences $\{h_{j,m}\},$ $ 1\le j\le n,$ choose 
$\Psi: C_0(S^n\setminus \zeta^{np})\to l^\infty(B(H))/c_0(B(H))$ 
as defined in Lemma \ref{Lsyth1}.  Also we retain  the notation 
$\Lambda$ and $L_k$ as in the statement as well as in the proof of Lemma \ref{Lsyth1} associated 
with the same sequences $\{h_{j,m}\}_{m\in \N},$ $ 1\le j\le n.$ 

Let us first look at (\ref{eq:SSpsSp}).
In what follows, for $x\in S^n\setminus \zeta^{np}$ and $\rho>0,$
denote 
$D(x, \rho)=\{y\in S^n: \|x-y\|_2<\rho\}.$

 Fix $D^{\eta}_M\subset \{\zeta\in\mathbb R^n : \|\zeta\|_2\le M \}$, a finite $\eta$-dense set. 
For every $m\in\mathbb N,$ there exists $\xi_m\in D^{\eta}_M$ such that 
\[
\overline{D(\Phi_{S^n}(\xi_m), \eta)}\subset \Phi_{S^n}(S\mathrm{Sp}_M^{\eta}((h_{1,m}, \ldots, h_{n,m})))\quad \text{but} \quad \overline{B(\xi_m, \sigma)}\not\subset s\mathrm{Sp}_M^{\sigma}((h_{1,m}, \ldots h_{n,m}))
\]
for any $\sigma=\max\{\eta, F_M(\eta)+\sigma_1\}\}$ with   $0<\sigma_1\le \min\{\eta/4, F_M(\eta)/4\}.$
Since $D^{\eta}_M$ is a finite set, there is an infinite subsequence $\{k_m\}\subset\mathbb N$ such that $\xi_{k_l}=\xi_{k_m}$ for all $l,m$. Let us denote this vector by $\xi.$
Hence 
\[
\overline{D(\Phi_{S^n}(\xi), \eta)}\subset \Phi_{S^n}(S\mathrm{Sp}_M^{\eta}((h_{1,k_m}, \ldots, h_{n,k_m}))) \rforal  m\in\mathbb N
\] 
and by definition, this implies 
\[
\|\Theta^{ S}_{\Phi_{S^n}(\xi), \eta}(h_{1,k_m}, \ldots, h_{n,k_m})\|\ge 1-\eta \rforal m\in\mathbb N.
\]
Consider $\Theta_{\Phi_{S^n}(\xi),\eta}\in V{{_{00}^S}}$ and then
\begin{align*}
\|\Psi(\Theta_{\Phi_{S^n}(\xi),\eta}^S)\|=&\|\Pi(\{\Theta_{\Phi_{S^n}(\xi), \eta}^S(h_{1,k_m}, \ldots, h_{n,k_m})\}_{ m})\| \\
=& \limsup_m\|\Theta_{\Phi_{S^n}(\xi), \eta}^S(h_{1,k_m}, \ldots, h_{n,k_m})\|>1-\eta.
\end{align*}

Denoting $\Phi_{S^n}^{\#}: C_0(S^n\backslash{\zeta^{np}})\rightarrow C_0(\mathbb R^n)$ induced by $\Phi_{S^n}: \mathbb R^n\rightarrow S^n\backslash \{\zeta^{np}\}$. 
Then, by Lemma~\ref{Lsyth1}, $\Lambda\circ\Phi_{S^n}^{\#}=\Psi,$ and applying to $\Theta_{\Phi_{S^n}(\xi),\eta}^S$, one has 
\[
\Lambda(\Theta_{\Phi_{S^n}(\xi),\eta}^S\circ\Phi_{S^n})=\Psi(\Theta_{\Phi_{S^n}(\xi),\eta}^S).
\]
Put $0<\sigma_1\le \min\{\eta/4, F_M(\eta)/4\}.$ 
Define $f_{\xi_i}\in C(\{\zeta\in \R^n: \|\zeta\|_2\le M\})$ by $0\le f_{\xi_i}\le 1;$ 
$f_{\xi_i}(\zeta)=1$ if $\zeta-\xi_i\in F_M(\eta)^n;$ and 
$f_{\xi_i}(\zeta)=0$ if $\zeta-\xi_i\not\in  (F_M(\eta)+\sigma_1)^n.$
Since 
\beq
\Phi_{S^n}^{-1}(D(\Phi_{S^n}(\xi_i), \eta))\subset B(\xi_i, F_M(\eta)),
\eneq
we have 
\beq
f_{\xi_i}\ge \Theta^S_{\Phi_{S^n}(\xi_i), \eta}\circ \Phi_{S^n}.
\eneq
%It follows that 
%\beq
%\|\psi(f_{\xi_i})\|\ge 1-\eta.
%\eneq
Choose $\sigma_2=\max \{\eta,  F_M(\eta)+\sigma_1\}.$ Then 
\beq
\Theta_{\xi_i, \sigma_2}\ge f_{\xi_i}.
\eneq
It follows that $$\|\Lambda(\Theta_{\xi,\sigma_2})\|\ge \|\Lambda(\Theta_{\Phi_{S^n}(\xi),\eta}^S\circ\Phi_{S^n})\|=\|\Psi(\Theta_{\Phi_{S^n}(\xi),\eta}^S)\|>1-\eta\ge 1-\sigma_2.$$
By definition, $\Lambda$ is identical to $\Pi\circ \{L_k\}$ when restricted to $V_{00}$, we have,
\[\|\Lambda(\Theta_{\xi,\sigma_2})\|
=\|\Pi(\{\Theta_{\xi, \sigma_2}(h_{1,k_m}, h_{2,k_m},...,h_{n, k_m})\})\|=\limsup_m \|\Theta_{\xi, \sigma_2}(h_{1,k_m}, h_{2, k_m},...,h_{n, k_m})\|.
\]
Then there is an infinite subsequence $\{l_m\}\subset\{k_m\}$ such that 
\[
\|\Theta_{\xi, \sigma_2}(h_{1,k_m}, h_{2, k_m},...,h_{n, k_m})\|>1-\sigma_2 {{\rforal}} m\in\mathbb N,
\]
which means that 
\[
\overline{B(\xi,\sigma_2)}\subset s\mathrm{Sp}_M^{\sigma_2}((h_{1,l_m}, \ldots, h_{n, l_m})) {{\rforal}} m\in\mathbb N.
\]
This is a contradiction and thus (\ref{eq:SSpsSp}) does not hold.

The case for (\ref{eq:sSpSSp}) is similar. To derive a contradiction to arbitrary $\delta>0$, one just 
needs to find $\gamma$ depending on $\delta$ as follows. 
Put $0<\dt_1\le \min\{G_M(\dt)/4, \dt/4\}.$
Let $g_{\xi_j}\in C({{S_M}})$ be  such that
$0\le g_{\xi_j}\le 1;$ 
$g_{\xi_j}(\xi)=1,$ if $\xi-\xi_j\in G_M(\dt)^{n+1}\subset \R^{n+1};$
and $g_{\xi_j}(\xi)=0,$ if $\xi-\xi_j\not\in (G_M(\dt)+\dt_1)^{n+1}.$ 
Since 
\beq\label{LsS-1-30}
\Phi_{S^n}(B(\xi_j, \dt))\subset D(\Phi_{S^n}(\xi_j), G_M(\dt)),
\eneq
\beq
g_{\xi_j}\ge \Theta_{\xi_j, \dt}\circ \Phi_{S^n}^{-1}.
\eneq 
Choose $\gamma=\max \{\dt, G_M(\dt)+\dt_1\}.$ Then
\beq
\Theta_{\Phi_{S^n}(\xi_j), \gamma}^S\ge g_{\xi_j}.
\eneq
The rest of the proof follows similarly as the first part of the proof. 
\end{proof}

Theorem \ref{TAMU} follows from the following more refined result.
%elaborated result.

\begin{thm}\label{thm:6main}
Let $n\in \N,$ $\ep>0$ and $M>1.$  There exist $\eta>0$ (depending on 
$\ep$ and $M$) and $\dt(n,\eta, M, \ep)>0$ satisfying the following:

Let $\{h_1,h_2,...,h_{ n}\}$ be an ${ n}$-tuple of densely defined self-adjoint operators 
on an infinite dimensional separable Hilbert space $H$ such that
\beq\label{TAMU-1.4}
\|h_ih_j-h_{j}h_{i}\|<\dt,\,\,\, i,j=1,2,...,{ n} \quad and \,\,\, {\rm sp}(d^{1/2})\cap [0, M]\not=\emptyset,
\eneq
where $d=\sum_{j=1}^{ n} h_j^2.$
Then, 

{{\rm (i)}} $s{\rm Sp}_M^{\eta}((h_1, h_2,...,h_{ n}))\not=\emptyset,$

{{\rm (ii)}} for any $\lambda=(\lambda_1, \lambda_2,...,\lambda_{ n})\in s{\rm Sp}_M^{\eta}((h_1, h_2,...,h_{ n})),$
there is a unit vector $v\in H$ such that
\beq
%\max_{1\le i\le 2}|\la h_i(v), v\ra -\lambda_i|<\ep\tand
\max_{1\le i\le { n}}\|(h_i-\lambda_i)(v)\|<\ep.
\eneq
\end{thm}

\begin{proof}
Fix $\ep>0$ and $M>1.$  We may assume that $\ep<1/2.$ 
Choose $0<\eta_1<\ep/4$ (as $\eta$) and $\dt_1>0$ (as $\dt$) given by Theorem \ref{TAMU''}
for $\ep/2$ (as $\ep$) and $M$ (with $m$ being replaced by $n$).
%Put $\eta_1=\eta_0/4.$ 

Fix a finite generating subset ${\cal F}_g$  of $C(S^{n})$ as in Corollary \ref{Pextsp}.
%(as well as in Proposition \ref{Pspsub})

%Let $\dt_0$ be given by Proposition \ref{Pspsub} associated with $\eta_1$ (as $\eta$),
%$M$ and a fixed $S_M.$

Choose $0<\eta_2<\eta_1/4$ such that $G_M(\eta_2)<\eta_1/8.$
Put $0<{{\bt}}<\eta_2$ and $\gamma=\max\{\bt, G_M(\eta_2)+d_1\},$
where $d_1=\min\{\eta_2/4, G_M(\eta_2)/4\}.$ It follows that $\gamma<\eta_1/4.$
%{\blue{(Better not to use $d$ because this means the operator $\sum h_i^2$ in the statement.)}}

Choose $\eta_3<\eta_2/4$ such that
$F_M(\eta_3)<\eta_2/8.$
Put $\sigma=\max \{\eta_3, F_M(\eta_3)+\sigma_1\},$ where 
$\sigma_1=\min\{\eta_3/4, F_M(\eta_3)/4\}.$  Note that $0<\sigma <\eta_2/4.$

Choose $\dt_2>0$ (as $\dt$) given by Corollary \ref{Pextsp} associated with $\eta_3$ and $M$ above. 

Choose $\af_1>0$ (as $\af$)  given by Lemma \ref{LsS-2} for $n$, $M,$  $\eta_3$ (as $\eta$) 
and $\eta_2$ (as $\dt$).  Let $\af_2$ be (as $\af$)  required by Lemma \ref{Ladd251112-2}
for $\sigma$ (as $\eta$) $\eta_2$ (as $\dt$),  $n$ and  $M.$

Now let $\dt=\min\{\af_1, \af_2, \dt_2, \dt_1\}>0,$  $\eta=\eta_2$ and 
$\{h_1, h_2,...,h_{n}\}$ be an ${n}$-tuple of 
densely defined self-adjoint operators on $H$ such that \eqref{TAMU-1.4} holds.
By applying Corollary \ref{Pextsp}, we obtain 
that
\beq\label{6main-5}
S{\rm Sp}_M^{\eta_3}((h_1, h_2,...,h_{ n}))\not=\emptyset.
\eneq
By Lemma \ref{LsS-2},
\beq
&&S{\rm Sp}^{\eta_3}_M((h_1, h_2,...,h_{ n}))\subset s{\rm Sp}^{\sigma}_M((h_1, h_2,...,h_{ n}))\andeqn\\\label{eqgamma}
&&s{\rm Sp}^{\eta_2}_M((h_1, h_2,...,h_{ n}))\subset S{\rm Sp}_M^{\gamma}((h_1, h_2,...,h_{ n})),
\eneq
%{Pspsub}.
Since $\sigma<\eta_2/4,$ by applying Lemma \ref{Ladd251112-2}, $s{\rm Sp}_M^{\sigma}(h_1, h_2,...,h_{ n})\subset s{\rm Sp}^{\eta_2}_M(h_1, h_2,...,h_{ n}).$

As $\eta=\eta_2,$ by \eqref{6main-5}, 
  (i) holds. %{\red{In the meanwhile, $S{\rm Sp}_M^{\gamma}((h_1, h_2,...,h_{\red n}))\neq\emptyset$.}}

Since $\gamma<\eta_1/4,$  
$S{\rm Sp}_M^{\gamma}((h_1, h_2,...,h_n))\subset S{\rm Sp}_M^{\eta_1}((h_1, h_2,...,h_n))$ (see Lemma~\ref{Ladd251112-2}).
 By the choice of $\eta_1$ and $\dt_1,$ 
applying  Theorem \ref{TAMU''},  for any $\lambda=(\lambda_1, \lambda_2,...,\lambda_n)\in {{S{\rm Sp}_M^{\eta}}}((h_1, h_2,...,h_n)),$
there is a unit vector $v\in H$ such that
\beq
%\max_{1\le i\le 2}|\la h_i(v), v\ra -\lambda_i|<\ep\tand
\max_{1\le i\le n}\|(h_i-\lambda_i)(v)\|<\ep.
\eneq
Theorem then follows from \eqref{eqgamma} as $\eta=\eta_2.$
\end{proof}

\begin{rem}\label{Rems6}

(1) There is an additional condition ``${\rm sp}(d^{1/2})\cap [0, M]\not=\emptyset$" in Theorem \ref{thm:6main}.
One notes that ${\rm sp}(d^{1/2})\not=\emptyset.$ Hence there is always an 
$M>1$ such that ${\rm sp}(d^{1/2})\cap [0, M]\not=\emptyset.$  This might lead  one to think that 
the condition is redundant.   However,  we are dealing with 
unbounded 
operators,  $M$ could be very large.
  It is crucial that the parameter $\delta$ depends on this value of $M$, which is in fact expected. 
  On the other hand, 
 % Consequently, even though the condition appears redundant, it is necessary to include it for an accurate %and explicit statement of the result.  Moreover, 
  for examples in Section 7, this condition  always holds. 
% We need to require that $\dt$ depends on 
% the size of $M,$ which is in fact expected.   Therefore, while condition ``${\rm sp}(d^{1/2})\cap [0, M]\not=\emptyset$"
% seems redundant, in order to have an accurate statement, we need to include it.
%
%

(2) In general, it is difficult to compute the spectrum of an operator. Moreover the spectra of operators are unstable under small perturbation.  However, the synthetic spectrum such as 
$s{\rm Sp}^\eta_M((T_1, T_2,...,T_n))$ could be tested and is stable with small perturbation. 
 In order to have $\|\Theta_{\lambda,\eta}\|\ge 1-\eta,$
it suffices to have one unit vector $x\in H$  such that
\beq\label{Test}
\la \Theta_{\lambda, \eta}(T_1,T_2,...,T_n)x, x\ra>1-\eta.
\eneq 
This is much more convenient  than $S{\rm Sp}^\eta_M((T_1,T_2,...,T_n)).$ In fact, in the case of the
position-momentum system (see Section~\ref{pme}), a student group is explicitly calculating $s{\rm Sp}^\eta_M((S_1, S_{\hbar}))$ numerically.
\end{rem}

\section{Momentum and position operators}

\begin{NN}\label{pme}\textbf{Position-momentum example.}

{\rm  Let us recall some known facts which will be used here.

Let $H = L^2(\mathbb{R})$ be the Hilbert space of square-integrable functions on $\mathbb{R}$ with respect to the Lebesgue measure. 
Let $D_0$ be the dense subspace of infinitely differentiable functions $f\in L^2(\R)$, where each derivative $f^{(m)}\in L^2(\R),$ and 
\beq\label{eq:Sobolev}
t^m f(t), t^m f'(t)\in L^2(\R)\cap C_0(\R)\rforal m\in \N.
\eneq
Since each derivative $f^{(m)}$ belongs to $L^2(\mathbb{R})$, it follows that $f^{(m-1)}$ is absolutely continuous and vanishes at infinity.
%In fact, from $f\in L^2(\R)$ and $f^{(m)}\in L^2(\R)$ for all $m$, one sees that $D_0\subset H^s(\mathbb{R})$ for every $s\ge0$, so by Sobolev embedding theorem, each $f\in D_0$ is smooth and all its derivatives vanish at infinity. See~\cite{Friedlander} for example.  
Since the Schwartz space \(\mathcal{S}(\mathbb{R})\) is contained in \(D_0\) and \(\mathcal{S}(\mathbb{R})\) is dense in \(L^2(\mathbb{R})\), it follows that \(D_0\) is dense in \(L^2(\mathbb{R})\).

%{\red{Need to find a reference, or provide a short explanation that such $D_0$ is dense 
%in $L^2(]\R).$  Perhaps Sobolev space  is sufficient  ?}}

Over the dense domain $D_0$, consider the \emph{position operator}
\begin{equation}\label{eq:position}
S_1: L^2(\R) \;\longrightarrow\; L^2(\mathbb{R}), 
\qquad
S_1 f(t) \;=\; t\,f(t),
\end{equation}
and the \emph{momentum operator} (with Planck constant \(\hbar\neq0\))
\begin{equation}\label{eq:momentum}
S_{\hbar}: D_0 \;\longrightarrow\; L^2(\mathbb{R}), 
\qquad
S_{\hbar}f(t) \;=\; -\,i\,\hbar\,f'(t).
\end{equation}
It is immediate that $S_1$ is symmetric.  By integrating by parts on $D_0$, one also checks that $S_{\hbar}$ is symmetric.  Moreover, standard results (e.g., \cite{Chernoff}) imply that on the line \(\mathbb{R}\) both \(S_1\) and \(S_{\hbar}\) admit a unique self-adjoint extension (its closure) on its maximal domain.

The operator $S_{\hbar}$ does not become small when $|\hbar|\to 0.$
%The family of operators $S_{\hbar}$ parametrized in $\hbar$ remain unbounded.
% in strong operator topology. 
Fix $\hbar$ with $0<|\hbar|\le 1.$ 
Choose for example  $g(t)=\pi^{-1/4} e^{-t^2/2}.$ Then $\|g\|_2=1$ and 
$
\|g'\|_2={1\over{\sqrt{2}}}.
$
For each $n\in \N,$ define 
\beq
a_n(\hbar)={|\hbar| \over{n\sqrt{2}}}\andeqn f_n(t)={1\over{\sqrt{a_n(\hbar)}}} g(t/a_n(\hbar)).
\eneq
Then 
$\|f_n\|_2=\|g\|_2=1$ for all $n,\, \hbar.$ Moreover,
\beq
\|f_n'\|_2={\|g'\|_2\over{a_n(\hbar)}}={\|g'\|_2\over{|\hbar |\|g'\|_2 /n}}={n\over{|\hbar|}}.
\eneq
We compute that
\beq
\|S_{\hbar}(f_n)\|_2=|\hbar|\|f_n'\|_2=|\hbar | \cdot {n\over{|\hbar|}}=n.
\eneq
The last identity is independent of $|\hbar|.$ 
%%%%%%%%%%
%%%%%%%%%%%%%%%%%%%%
\iffalse
$f_0(t)=te^{-{t^2\over{\hbar^2}}}$ and  
$f_1(t)={1\over{\sqrt{|\hbar|} }}e^{-\frac{t^2}{\hbar^2}}$ for $t\in \R.$ 
Then 
\beq
&&
%|f_1(t)|{\red{\ge}} {e^{-t^2}\over{\sqrt{|\hbar|}}},\,\,\, 
\|f_1\|_2= \|e^{-t^2}\|_2\andeqn\\
&&|S_{\hbar}(f_1)(t)|=|{2t\over{\hbar}} f_1(t)|= {2\over{|\hbar|^{3/2}}} |f_0(t)|\,\,\, \rforal t\in \R.
\eneq
It is  helpful to note that
\beq
\|S_{\hbar}(f_1)\|_2= {2\over{|\hbar|^{3/2}}}\|f_0\|_2={2\over{|\hbar|^{3/2}}}(1/2)(\pi/2)^{1/4} |\hbar|^{3/2}=(\pi/2)^{1/4}.
%\over{|\hbar|^{3/4}}}.
\eneq
\fi
%%%%%%%%%%%%%%%%%%%%%
%
%%%%%%%%
This implies that $S_{\hbar}$ remains ``large" as $|\hbar|\to 0.$
On the other hand, for every $f\in D_0,$ 
\beq
S_{\hbar}\circ S_1(f)-S_1\circ S_{\hbar}(f)&=&S_{\hbar }(tf(t))+i \hbar S_1(f'(t))\\
&=&-i\hbar (f(t)+tf'(t))+i\hbar tf'(t)=-i\hbar f(t).
\eneq
Hence,
\beq
\|S_1S_{\hbar}-S_{\hbar}S_1\|=|\hbar|
\eneq
which goes to zero as $|\hbar|\to 0.$
In other words, $S_1$ and $S_{\hbar}$ become ``almost commuting"
% in the limit 
as $\hbar\to0$.

Let us also mention that 
\beq \label{S1Sh2}
[S_1,\, S_{\hbar}^2]=2i\hbar S_{\hbar}
\eneq
is an unbounded operator.
}
\end{NN}

The following result establishes the existence of Approximately Macroscopic Unique  (AMU) States  in quantum systems governed by position and momentum operators, under the condition that the parameter
$|\hbar|$ is sufficiently small. This yields an affirmative resolution to Mumford's question for such systems.
It should be noted,
 if we only consider the existence of AMU, then we may remove $M$ in the statement 
 of Theorem \ref{TMP} by choosing $M=1+1/2,$ for example.

%{\red{The following result shows that AMU exists in  this position-momentum system  whenever 
%the constant $|\hbar|$ is small which shows that one has an affirmative answer to the  Mumford's question in this classical 
%quantum system. }}

\begin{thm}\label{TMP}
For any $\ep>0$ and $M>1,$ there is $0<\eta$ and $\dt>0$ such that, 
when $|\hbar|<\dt,$

{\rm (i)} $s{{\rm Sp}}_M^{\eta}(S_1, S_{\hbar})\not=\emptyset,$

{\rm (ii)}
 for any $\lambda=(\lambda_1, \lambda_2)\in s{\rm Sp}_M^{\eta}(S_1, S_{\hbar}),$
there exists a unit vector $v\in H$ such that
\beq
\|S_1(v)-\lambda_1 v\|_2<\ep\andeqn \|S_{\hbar} (v)-\lambda_2 v\|_2<\ep.
\eneq
\end{thm}

\begin{proof}
Set $d = S_1^2 + S_{\hbar}^2$.  A direct calculation (or, equivalently, by using the standard spectrum of the harmonic oscillator Hamiltonian $H=\frac{d}{2}$; see (2.3.8) in Section 2.3 of \cite{MQM}) shows that the eigenvalues of $d$ are 
\[
\{\, (2n+1)\,\lvert \hbar\rvert \;:\; n=0,1,2,\ldots\}.
\]
Thus, whenever $\lvert\hbar\rvert \le \tfrac{1}{16}$, the interval $[0,1]$ intersects with the spectrum of $d^{1/2}$.  In particular, for any $M>1$,
${\rm sp}(d^{1/2})\cap [0, M]\not=\emptyset.$
% the spectral projection 
%$e_M\bigl(d^{1/2}\bigr)$ 
%is nonzero.

Let $\ep>0$ and $M>1$ be given. Choose $\eta>0$ and $\dt_1:=\dt(\ep, M)>0$ as in Theorem \ref{thm:6main}
for $\ep$ and $M$ as mentioned.  Choose $\dt=\min\{\dt_1, 1/16\}.$ 
Then whenever $\lvert\hbar\rvert<\delta$, one has
\beq
\|S_1S_{\hbar}-S_{\hbar}S_1\|=|\hbar|<\dt. 
\eneq
%Moreover,  $e_M(d^{1/2})\not=0.$
It then follows from Theorem~\ref{thm:6main}
 that (i)  holds. Moreover, also by Theorem~\ref{thm:6main}, for any $(\lambda_1,\lambda_2)\in s{\rm Sp}_M^\eta((S_1, S_{\hbar})),$
 there is $v\in L^2(\R)$ with $\|v\|=1$ such that
\beq
\|S_1(v)-\lambda_1 v\|_2<\ep\andeqn \|S_{\hbar} (v)-\lambda_2 v\|_2<\ep.
\eneq

\end{proof}

\begin{NN} 
%\textbf{Position-momentum example.}
{\rm Consider the unbounded operator $T_{\hbar}=S_1+i S_{\hbar}$ with domain $D_0$. Then 
\beq
T_{\hbar}(f)(t)=tf(t)+\hbar f'(t)\rforal f\in D_0\andeqn t\in \R.
\eneq
One also has that
\beq
T_{\hbar}^*(f)(t)=tf(t)-\hbar f'(t)\rforal f\in D_0\andeqn t\in \R.
\eneq
%
%%%%%%%%%%%%%%%%
%
\iffalse
%%%%%%%%%%%%%%%%%%%%%
The closure of $T_{\hbar}$ is a (densely defined) closed operator, which we again denote by $T_{\hbar}$.  The domain of the closure $T_{\hbar}$ is \[
D_1 \;=\; \bigl\{\,f\in L^2(\mathbb{R}): t f(t)+\hbar f'(t)\in L^2(\mathbb{R})\bigr\}.
\]

Note that the closedness can be verified using distributions:
if $f_n\to f$ and $Tf_n\to g$ in $L^2(\R),$ then, for any $\phi\in C_c^\infty(\R),$
\beq
\la Tf_n, \phi\ra \to \la tf, \phi\ra +\hbar \la f', \phi\ra=\la g,\phi \ra.
\eneq
Since $C_c^\infty(\R)$ is dense in $L^2(\R),$ we have 
%which implies 
that $Tf=g.$ 
%Since  
%in distribution sense. 
Hence, $f\in D_1$ and the closure is realized on $D_1$.
\fi
%%%%%%%%%%%%%%%%%%%%%%%%%%%%
%

%
%%%%%%%%%%%%%%%%%%%%%%%%%%%%%%
Now, define 
\[
L_{\hbar}:= \bigl(1 + T_{\hbar}\,T_{\hbar}^*\bigr)^{-1/2}\,\;T_{\hbar}.
\]
%
%where $T_{\hbar}:D_1\to L^2(\mathbb{R})$ is a \emph{bounded} operator 
%(see Section 2)
Then $L_{\hbar}$ is bounded
 (\emph{c.f.} Theorem 25.3 of \cite{Shubin}),
 and $(1+T_{\hbar}T_{\hbar}^*)^{-1/2}:L^2(\mathbb{R})\to D_0$ is a bounded  operator.  Since $D_1$ is dense in $L^2(\mathbb{R})$, $L_{\hbar}$ extends to a bounded operator \(L_{\hbar}:L^2(\mathbb{R})\to L^2(\mathbb{R})\). 
} 
\end{NN}

%\begin{NN}
%Let $T_{\hbar}=S_1+i S_{\hbar}$ and 
%$L_{\hbar}:=(1+ (T_{\hbar}T_{\hbar}^*))^{-1/2} T_\hbar.$ 
%Here $T_{\hbar}: D_1\rightarrow L^2(\mathbb R)$ is regarded as a bounded Fredholm operator, while $(1+ (T_{\hbar}T_{\hbar}^*))^{-1/2}: L^2(\mathbb R^2)\rightarrow D_1$ is an bounded invertible operator with zero index. Then since $D_1$ is dense in $L^2(\mathbb R)$, $L_\hbar : L^2(\mathbb R)\rightarrow L^2(\mathbb R)$ is a bounded operator. By the fact that and $T_{\hbar}$ has index $1$ and the property that the index is additive under composition, $L_{\hbar}$ is also a Fredholm operator with index $1$. 
%\end{NN}

\begin{lem}\label{Lnorm}
Denote by $\pi$ the quotient map from $B(L^2(\R))$ to the Calkin algebra 
$B(L^2(\R))/{\cal K}.$ Then
\beq
\pi(L_\hbar)\pi(L_\hbar^*)=1.
\eneq
Moreover, $L_\hbar$ has Fredholm index $1$ if $\hbar>0$ and $-1$ if $\hbar <0.$ 
%\red{Moreover 
%the essential spectrum of $L_\hbar$ is the unit circle.}
\end{lem}

\begin{proof}

First observe that $T_{\hbar}T_{\hbar}^*$ and $T_{\hbar}^*T_{\hbar}$ (initially defined on \(D_0\)) extend to self-adjoint operators whose domain is 
\[
\bigl\{\,f\in L^2(\mathbb{R}): -\,f'' + t^2 f \in L^2(\mathbb{R})\bigr\},
\]
see \cite[Theorem 26.2]{Shubin}. 
Let $\xi_0$ be the unit vector so that $T_{\hbar}\xi_0=0$ ($\xi_0$ is a positive scalar multiple of $e^{-x^2/2}$). 
Then, by setting $\xi_n=(T_{\hbar}^*)^n\xi_0$, one has, as in Theorem 11.2 in \cite{Hall},
\[
T_{\hbar}^*T_{\hbar}\xi_n=2\hbar n\xi_n,
\]
where $\|\xi_n\|=(2\hbar)^n.$ 
In fact, as claimed in the theorem, one has $a^*a\xi_n=n\xi_n$, 
where $a$ is defined in (11.4) of \cite{Hall} and satisfies $\sqrt{2\hbar}a=T_{\hbar}$ by setting $m\omega=1$ in $a$.
Then by Theorem 11.4 in \cite{Hall}, $\{(2\hbar)^{-n}\xi_n :n\ge 0\}$ form an orthonormal function basis for $L^2(\R).$ 
Hence, the set of eigenvalues $\{2\hbar n : n=0,1,\ldots\}$ forms the full spectrum for $T_{\hbar}^*T_{\hbar}.$ 
Observe that $T_{\hbar}T_{\hbar}^*=T_{\hbar}^*T_{\hbar}+2\hbar$. 
Thus, the spectrum of $T_{\hbar}T_{\hbar}^*$ is the set of eigenvalues \[
\{2\hbar (n+1)~:~ n=0,1,\ldots\}
\]
each with finite multiplicity. 

Consequently, by the spectral theory, we write 
\beq
T_\hbar T^*_\hbar =\int_{{\rm sp}(T_\hbar T_\hbar^*)}\lambda d E_\lambda,
\eneq
where $\{E_\lambda\}$ is the associated spectral measure. 
Then 
\beq
(1+T_\hbar T_\hbar ^*)^{-1}&=&\int_{{\rm sp}(T_\hbar T_\hbar^*)}(1+\lambda)^{-1} d E_\lambda
\andeqn \\
(1+T_\hbar T_\hbar ^*)^{-1/2}&=&\int_{{\rm sp}(T_\hbar T_\hbar^*)}(1+\lambda )^{-1/2}d E_\lambda.
\eneq
In particular, $(1+T_\hbar T_\hbar^*)^{-1}$ is a bounded self-adjoint operator 
whose spectrum 
has only $0$ as a limit point, which is not an eigenvalue (so ${\rm ker} (1+T_\hbar T_\hbar^*)^{-1}$ is zero)
%has only zero as a limit point which is not an eigenvalue (so ${\rm ker} (1+T_\hbar T_\hbar^*)^{-1}$ is zero)  
and every non-zero point in the spectrum 
is an eigenvalue with finite dimensional eigenspace. It follows $(1+T_\hbar T_\hbar^*)^{-1}$ 
is a compact operator. Hence, $(1+T_\hbar T_\hbar^*)^{-1/2}$ is also compact.
Moreover $(1+T_\hbar T_\hbar^*)^{-1/2}$ is injective and has dense range.

Now one computes
\[
1 \;-\; L_{\hbar}\,L_{\hbar}^*
= 1 \;-\; \bigl(1 + T_{\hbar}T_{\hbar}^*\bigr)^{-1/2} \,T_{\hbar} \,T_{\hbar}^* \,\bigl(1 + T_{\hbar}T_{\hbar}^*\bigr)^{-1/2} 
= \bigl(1 + T_{\hbar}T_{\hbar}^*\bigr)^{-1},
\]
which is compact.  Consequently, in the Calkin algebra,
\[
\pi\bigl(L_{\hbar}\bigr)\,\pi\bigl(L_{\hbar}^*\bigr) \;=\; 1,
\]
{{where $\pi: B(L^2(\R))\rightarrow B(L^2(\R))/\mathcal{K}(L^2(\R))$ is the quotient map.}}

Next, we calculate the index. 
For
\[
T_{\hbar}: D_0\;\longrightarrow\;L^2(\mathbb{R}),
\]
when $\hbar>0$, one solves $T_{\hbar}f = 0$, i.e.
\[
t f(t) \;+\;\hbar\,f'(t) = 0
\;\Longrightarrow\; \frac{f'(t)}{f(t)} = -\frac{t}{\hbar},
\]
so $f(t) = C\,e^{-t^2/2\hbar}$.  Since $e^{-t^2/2\hbar}\in L^2(\mathbb{R})$, the kernel $\ker T_{\hbar}$ is one-dimensional, spanned by $e^{-t^2/2\hbar}$.  On the other hand, the adjoint equation $T_{\hbar}^*f = 0$ becomes
\[
t\,f(t) \;-\;\hbar\,f'(t) = 0,
\]
whose only $L^2$-solution is the trivial one.  Hence \(\mathrm{coker}\,T_{\hbar}=\ker T_{\hbar}^*=0\).  Consequently, for $\hbar>0$, 
\[
\mathrm{Ind}(T_{\hbar}) \;=\; \dim\ker T_{\hbar} \;-\;\dim\mathrm{coker}\,T_{\hbar} \;=\; 1.
\]
Similarly, when $\hbar<0$, one finds $\ker T_{\hbar}=0$ and $\mathrm{coker}\,T_{\hbar}$ is one-dimensional (generated by $e^{t^2/2\hbar}$ on the adjoint equation), so $\mathrm{Ind}(T_{\hbar})=-1$.

Finally, since $L_{\hbar}$ is $T_{\hbar}$ composed with an injective operator 
with dense range and thus $\mathrm{Ind}(L_{\hbar}) = \mathrm{Ind}(T_{\hbar})$, it follows that $\mathrm{Ind}(L_{\hbar}) = 1$ for $\hbar>0$ and $-1$ for $\hbar<0$.
\end{proof}

The following is a foklore.

\begin{lem}\label{Lclose}
Let $T$ and $S$ be selfadjoint operators defined on a Hilbert space $H,$ with 
dense domains $D(T)$ and $D(S),$ 
respectively. Suppose that  
$D=D(T)\cap D(S)$ is dense in $H.$
Then $T+iS$ with domain $D$ is closable.
\end{lem}

\begin{proof}
Put $A=T+iS.$ 
Let $y\in D(T)\cap D(S)$ and $x\in D,$  we have 
$$\la y,Ax\ra =\la Ty,x\ra - i \la Sy,x\ra =
\la (T- i S)y ,x\ra.$$ 
It follows that $y\in D(A^*)$ and $A^*y=(T-i S)y.$  Hence $D(T)\cap D(S)\subset D(A^*).$
Therefore $A$ is closable.
\end{proof}

\begin{df}\label{Ddifference}
Let $S$ and $T$ be densely defined unbounded 
%self-adjoint 
operators
on an infinite dimensional separable Hilbert space $H$ and $\dt>0.$
In what follows we write 
\beq
\|S-T\|<\dt,
\eneq
if there is a common dense domain $D_0$ such that
\beq
\|(S-T)|_{D_0}\|<\dt.
\eneq
Suppose that $\|S-T\|<\dt.$ 
Then the following hold:
There are increasing sequence of closed subspaces $H_n\subset H_{n+1}$ ($n\in \N$)
such that $S|_{H_n}$ and $T|_{H_n}$ are bounded  for all $n\in \N,$ 
$\bigcup_{n=1}^\infty H_n$ is in $D_0$ and dense in $H,$ and 
\beq
\|(S-T)|_{H_n}\|<\dt\rforal n\in \N.
\eneq
For example, we may choose a dense sequence $\{v_n\}$ in $D_0$ and then define
$H_n={\rm span} \{v_1, v_2,...,v_n\},$ $n\in \N.$ It follows that $S|_{H_n}$ and $T|_{H_n}$ are bounded 
and $\bigcup_{n=1}^\infty H_n$ is dense. 
\end{df}

\begin{lem}\label{Lnorm<2dt}
Let $T, N$ be densely defined closable unbounded operators satisfying $\|T-N\|<\delta$, then \[
\|(1+TT^*)^{-1/2}T-(1+NN^*)^{-1/2}N\|<2\delta.\]
\end{lem}

\begin{proof}
Fix an increasing sequence of closed subspaces $H_n\subset H_{n+1}$ ($n\in \N$) 
such that $\bigcup_{n=1}^\infty H_n$ is dense in $H,$
$T|_{H_n}$ and $N|_{H_n}$ are bounded and 
\beq
\|(T-N)|_{H_n}\|<\dt\rforal n\in \N.
\eneq
Let $A=\begin{bmatrix}
    0 & T^* \\ T & 0
\end{bmatrix}$ and $B=\begin{bmatrix}
    0 & N^* \\ N & 0
\end{bmatrix}.$ Since both are closable,  we may assume that $A$ and $B$ are selfadjoint.
It then is sufficient to show that if $\|A-B\|<\delta$, then 
\[
\|(1+A^2)^{-1/2}A-(1+B^2)^{-1/2}B\|<2\delta.
\]
%By functional calculus, we have 
%\beq
%\eneq
%
%%%%%
Let $p_m\in B(H)$ be the projection on $H_m,$ $m\in \N.$ 
It follows that $Ap_m$ and $Bp_m$ are bounded operators.
%In fact, we have, by the spectral theory,
As in the proof of Lemma \ref{Lappcom-1n},
\[
(1+A^2)^{-1/2}Ap_m-(1+B^2)^{-1/2}Bp_m=\frac{2}{\pi}\int_0^{\infty}[(1+A^2+\lambda^2)^{-1}A-B(1+B^2+\lambda^2)^{-1}]p_md\lambda,
\]
where the integrand is
\begin{align*}
&((1+A^2+\lambda^2)^{-1}A-B(1+B^2+\lambda^2)^{-1})p_m \\
=&((1+A^2+\lambda^2)^{-1}(A(1+B^2+\lambda^2)-(1+A^2+\lambda^2)B)(1+B^2+\lambda^2)^{-1})p_m \\
=& ((1+A^2+\lambda^2)^{-1}((A-B)(1+\lambda^2)+A(B-A)B)(1+B^2+\lambda^2)^{-1})p_m. 
\end{align*}
Using 
\beq
\|A-B\|<\delta,~~\|(1+A^2+\lambda^2)^{-1}(1+\lambda^2)\|\le 1
\,\,\,
\|(1+B^2+\lambda^2)^{-1}\|\le\frac{1}{\lambda^2+1},
\eneq
we obtain
%the norm of first summand of the integrand is bounded above by $\frac{\delta}{1+\lambda^2},$ i.e.,
\beq
 \|(1+A^2+\lambda^2)^{-1}[(A-B)(1+\lambda^2)](1+B^2+\lambda^2)^{-1}\|\le\frac{\delta}{1+\lambda^2}.
\eneq

Similarly, using
\begin{align*}
&\|B-A\|<\delta,~~ \|(1+A^2+\lambda^2)^{-1/2}A\|\le 1, ~~\|B(1+B^2+\lambda^2)^{-1/2}\|\le 1,\\
&\|(1+A^2+\lambda^2)^{-1/2}\|\le (1+\lambda^2)^{-1/2} \andeqn \|(1+B^2+\lambda^2)^{-1/2}\|\le (1+\lambda^2)^{-1/2},
\end{align*}
we obtain
%
%the norm of the second summand of the integrand is bounded above by $\frac{\delta}{1+\lambda^2}:$
\[
 \|(1+A^2+\lambda^2)^{-1}(A(B-A)B)(1+B^2+\lambda^2)^{-1}\|\le \frac{\delta}{1+\lambda^2}.
\]
Therefore, for each $m\in \N,$
\[
\|((1+A^2)^{-1/2}A-(1+B^2)^{-1/2}B)p_m\|\le \frac{2}{\pi}\int_0^{\infty}\frac{2\delta}{1+\lambda^2} d\lambda=2\delta.
\]
It follows that
\beq
\|(1+A^2)^{-1/2}A-(1+B^2)^{-1/2}B\|<2\dt
\eneq
as claimed.
\end{proof}

The following theorem shows that the quantum mechanical system of the position and momentum 
could {\it not} be possibly approximated by classical mechanical systems no matter how small 
$|\hbar|$ might be.

\begin{thm}\label{Tnotrecover}
Let $S_1$ and $S_{\hbar}$ be the position and momentum operators respectively in (\ref{eq:position})-(\ref{eq:momentum}). Then 
\beq
\liminf_{|\hbar|\to 0}\,\,\inf_{h_i=h_i^*}\{\|(S_1-h_1)|_{D_0}\|+\|(S_{\hbar}-h_2)|_{D_0}\|:  h_1h_2=h_2h_1\}\ge 1/32.
\eneq
Thus, one cannot approximate the pair $(S_1, S_{\hbar})$ arbitrarily well by a commuting pair $(h_1,h_2)$ as \(\hbar\to0\).
\end{thm}

\begin{proof}
Suppose to the contrary that for some sequence $\hbar_n\to0$ there exist self-adjoint 
operators $h_{1,n},\,h_{2,n}$ with $h_{1,n}h_{2,n}=h_{2n}h_{1,n}$ such that
\[
\lVert\,T_n - T_{\hbar_n}\rVert 
\;=\; \bigl\lVert(h_{1,n} + i\,h_{2,n}) - (S_1 + i\,S_{\hbar_n})\bigr\rVert 
\;<\; \tfrac{1}{32},
\]
where $T_n:= h_{1,n} + i\,h_{2,n}$ and $T_{\hbar_n}=S_1+i \,S_{\hbar_n}.$
It then follows that $T_n$ is normal hence closed. 
By Lemma \ref{Lclose}, $S_1+i\, S_{\hbar}$ is closable (in fact, in this case it is known
that the annihilation operator is always closable). 
% and ${\rm Ind}(T_n)=0.$ 
Let 
\beq
N_n=(1+T_nT_n^*)^{-1/2}T_n\andeqn  L_{\hbar_n}=(1+T_{\hbar_n}T_{\hbar_n}^*)^{-1/2} T_{\hbar_n}.
\eneq
Then 
\beq
\|N_n\|\le 1\andeqn \|L_{\hbar_n}\|\le 1.
\eneq
In particular, $N_n$ is a bounded normal operator.
Suppose that
\beq
\|T_n-(T_{\hbar_n})\|<1/32,\,\,\, n\in \N.
\eneq
Then, by Lemma \ref{Lnorm<2dt} with $\delta=1/32$, 
\beq
\|N_n-L_{\hbar_n}\|<1/16.
\eneq
Passing to the Calkin algebra, it follows that 
\beq\label{Tnotrecover-10}
\|\pi(N_n)-\pi(L_{\hbar_n})\|<1/16,
\eneq
where $\pi: B(L^2(\R))\rightarrow B(L^2(\R))/\mathcal{K}$ is the quotient map.
Note that, by Lemma \ref{Lnorm},
$\pi(L_{\hbar_n})\pi(L_{\hbar_n}^*)=1$  and  $\|\pi(L_{\hbar_n})\|=1.$  
 Since $N_n$ is normal, we also have 
that
\beq
&&\hspace{-0.4in}\|\pi(N_n^*N_n)-1\|=\|\pi(N_nN_n^*)-1\|\\
&&\le \|\pi(N_nN_n^*)-\pi(N_n L_{\hbar_n}^*)\|+\|\pi(N_n L_{\hbar_n}^*)-\pi( L_{\hbar_n} L_{\hbar_n}^*)\|+
\|\pi( L_{\hbar_n}L_{\hbar_n}^*)-1\|\\
&&\le \|\pi(N_n^*)-\pi(L_{\hbar_n}^*)\|+\|\pi(N_n)-\pi( L_{\hbar_n})\|<1/8.
\eneq
This in particular implies that both $\pi(N_n)$ and $\pi(N_n^*)$ are invertible. 
It follows this and \eqref{Tnotrecover-10}   that $\pi(L_{\hbar_n})$ is invertible. Hence it is a unitary, denoted by $U$. 
We also have 
\beq
\hspace{-0.4in}\|\pi(L_{\hbar_n})-\pi(N_n)|\pi(N_n)|^{-1}\|&\le& \||\pi(L_{\hbar_n})-\pi(N_n)\|+\|\pi(N_n)-\pi(N_n)|\pi(N_n)|^{-1}\|\\
&<&1/16+\|1-|\pi(N_nN_n^*)|^{1/2}\|\\
&<&1/16+1/8<1.
\eneq
Note that $V:=\pi(N_n)|\pi(N_n)|^{-1}$ is a unitary.
So $U$ and $V$ must be in the same component of the set of unitaries $U(B({L^2(\R)})/{\cal K}).$
However since $N_n$ is normal, it has zero  Fredholm index.  In other words, $[V]=0$ in $K_1(B({L^2(\R)})/{\cal K}),$ which contradicts to the fact that $L_{\hbar_n}$ has Fredholm index $1.$ 
\end{proof}

\begin{NN} \textbf{Angular-momentum example.}
{\rm The angular momentum operators $L_x, L_y$ and $L_z$ in $L^2(\R^3)$ are defined as follows:
\beq
L_x&=&-i \hbar \left(y{\partial\over{\partial z}}-z{\partial \over{\partial y}}\right),\\
L_y&=&-i \hbar \left(z{\partial\over{\partial x}}-x{\partial \over{\partial z}}\right)\andeqn\\
L_z&=&-i \hbar \left(x{\partial\over{\partial y}}-y{\partial \over{\partial x}}\right).
\eneq
These are unbounded densely self-adjoint operators. 
They satisfy the following properties:
\beq\label{Lxcomm-1}
{[}L_x, L_y{]}&=&{ i}\hbar L_z, \\\label{Lxcomm-2}
{[}L_x, L_z{]}&=&{i}\hbar L_y,\\\label{Lxcomm-3}
{[}L_y, L_z{]}&=&{i}{\hbar} L_x\andeqn\\\label{Lxcomm-4}
{[}L_j, {\bf L}^2{]}&=&0,\,\,\, j=x,y,z,
\eneq
where ${\bf L}^2=L_x^2+L_y^2+L_z^2.$
Define
\beq
\tilde L_j=L_j(1+{\bf L}^2)^{-1},\,\,\, j=x, y,z.
\eneq
It is also known that ${\bf L}^2$ has eigenvalues $m(m+1)\hbar,$ for $m=0,1,2,\ldots$ (see Section 3.5 of~\cite{MQM}).
}
\end{NN}

We realize that, in the  angular-momentum system, commutators 
such as $[L_x, L_y]$ is a unbounded operator which is not small even $|\hbar |$ is small.
Nevertheless, AMU states also exist in this case as shown in the next theorem.

\begin{thm}\label{LAnglemom}
For any $\ep>0$ and $M>1,$ there exists $\eta>0$ and $\dt>0$ such that, when $|\hbar |<\dt,$

{{\rm (i)}} $S{\rm Sp}_M^{\eta}(L_x, L_y, L_z)\not=\emptyset;$

{{\rm (ii)}}  for each $\lambda=(\lambda_x, \lambda_y, \lambda_z)\in S{\rm Sp}_M^{\eta}(L_x, L_y, L_z),$
there exists $v\in L^2(\R^3)$ with $\|v\|=1$ such that
\beq
\|L_j(v)-\lambda_j v\|<\ep,\,\,\, j=x, y, z.
\eneq

\end{thm}

\begin{proof}
Since ${\bf L}^2$ has eigenvalues $m(m+1)\hbar,$ when $|\hbar|<1/16,$  for any $M>1,$ 
${\rm sp}({\bf L})\cap [0,1]\not=\emptyset. $ 
%It follows that 
%\beq
%e_M({\bf L})\not=0.
%\eneq
Choose $\dt_1=\dt(\ep,M)$ in Theorem \ref{TAMU1} for $\ep$ and $M.$ 
Choose $\dt=\min\{\dt_1, 1/16\}.$ 
One computes that (for all $\hbar$)
\beq
\|\tilde L_j\|\le 1,\,\,\, j=x, y,z.
\eneq
By \eqref{Lxcomm-4}, 
\beq
[L_j,\, (1+{\bf L}^2)^{-1}]=0,\,\,\, j=x, y, z.
\eneq
One also computes that whenever $|\hbar|<\delta,$
\beq
\|[\tilde L_x, \tilde L_y]\|&=&\|L_x(1+{\bf L}^2)^{-1}L_y(1+{\bf L}^2)^{-1}-L_y(1+{\bf L}^2)^{-1}L_x(1+{\bf L}^2)^{-1}\|\\
&=&\|(1+{\bf L}^2)^{-2}[L_x, L_y](1+{\bf L}^2)^{-2}\|\\
&=&|\hbar|\|(1+{\bf L}^2)^{-2}L_z(1+{\bf L}^2)^{-2}\|\\
&\le & |\hbar|<\dt.
\eneq
Similarly
\beq
\|[\tilde L_j,\tilde L_i]\|<\dt\andeqn \tilde L_j^*=\tilde L_j,\,\,\, i,j\in\{x, y, z\}.
\eneq
Therefore Theorem \ref{TAMU1} applies. In other words, (i) holds, and there exists a unit vector $v\in L^2(\R^3)$ such that
\beq
\|L_j(v)-\lambda_j v\|<\ep, \,\,\, j=x, y,z.
\eneq
\end{proof}

%%%%%%
%

%%%%%%%%%%%
\iffalse

\section{Approximation by normal operators}

\begin{df}\label{Dbb}
Let $T$ be a densely defined (unbounded) operator.
Define $\tilde T=T(1+T^*T)^{-1}.$ 
\end{df}

\begin{thm}\label{Tapprox1}
For any $\ep>0,$ there exists $\dt>0$ satisfying the following:
Let $H$ be an infinite dimensional separable Hilbert space and 
 $T$ be a densely defined (unbounded) operator on $H.$ 
 Suppose that
 \beq
&& \|T^*T-TT^*\|<\dt\tand\\
&& {\rm Ind}(\lambda-\tilde T)=0\tforal \lambda\not\in ({\rm sp}_{ess}(\tilde T))_{\ep}.
 \eneq
 Then 
there is a normal operator $N\in B(H)$
such that
\beq	
\|\tilde T-N\|<\ep.
\eneq
\end{thm}

In what follows, if $T\in B(H)$ is an  unbounded normal operator 
$S\subset \C$ be a subset, denote by $P_S(T)$ the spectral projection 
of $T$ associated with $S.$ 

\begin{thm}
Let $\ep>0$ and $M>0$ be given.
There exists $\dt>0$ satisfying the following:

Suppose that there is a pair of  densely defined  (unbounded) self-adjoint operators 
$S_1, S_2$ 
on $H$ such that
\beq
\|S_1S_2-S_2S_1\|<\dt\andeqn\\
(\lambda-(S_1+i S_2))^{-1}\in 
\eneq

\end{thm}
%%%%%%
\fi
%
%
%%%%%%%%%%%%%%%%%%%%%%%%%%%%

\noindent
hlin@uoregon.edu and hlin@simis.cn\\
Shanghai Institute
for Mathematics and Interdisciplinary Sciences, 657 Songhu Road, Shanghai 200433, China

\noindent
wanghang@math.ecnu.edu.cn 
\\
Research Center of Operator Algebras, East China Normal University,  Shanghai 200062, China 

\end{document}